\documentclass[reqno,10pt,centertags,draft]{amsart}
\usepackage{amsmath,amsthm,amscd,amssymb,latexsym,upref,enumerate}
\newcommand{\bbC}{{\mathbb{C}}}

\newcommand{\bbN}{{\mathbb{N}}}

\newcommand{\bbR}{{\mathbb{R}}}

\newcommand{\cB}{{\mathcal B}}
\newcommand{\cC}{{\mathcal C}}

\newcommand{\cH}{{\mathcal H}}

\newcommand{\cN}{{\mathcal N}}

\newcommand{\cU}{{\mathcal U}}

\newcommand{\cX}{{\mathcal X}}
\newcommand{\cY}{{\mathcal Y}}

\newcommand{\tr}{\operatorname{tr}}
\newcommand{\dom}{\operatorname{dom }}

\newcommand{\rank}{\operatorname{rank}}
\newcommand{\ran}{\operatorname{ran}}

\newcommand{\no}{\notag}
\newcommand{\lb}{\label}
\newcommand{\f}{\frac}

\newcommand{\ol}{\overline}

\newcommand{\wti}{\widetilde}

\newcommand{\oh}{o}

\newcommand{\bi}{\bibitem}

\newcommand{\spn}{{\text{\rm span}}}

\newcommand{\Sect}{\text{\rm Sect}}

\renewcommand{\Im}{\text{\rm Im}}
\renewcommand{\ln}{\text{\rm ln}}
\renewcommand{\max}{\text{\rm max}}
\renewcommand{\min}{\text{\rm min}}

\renewcommand{\le}{\leqslant}

\newcommand{\si}{\sigma}

\newcommand{\la}{\lambda}
\newcommand{\La}{\Lambda}

\newcommand{\ga}{\gamma}

\newcommand{\de}{\delta}
\newcommand{\te}{\theta}

\newcommand{\bs}{\backslash}

\newcommand{\AB}{{A,B}}
\newcommand{\ABp}{{A'\!,B'}}

\newcommand{\bz}{\ensuremath{\bar z}}

\newcommand{\Lazz}{\ensuremath{\La_{0,0}}}
\newcommand{\Late}{\ensuremath{\La_{\te_a,\te_b}}}

\newcommand{\Lateq}{\ensuremath{\La_{\te_a,\te_b}^{(\te_a+\frac{\pi}{2})
\, \text{\rm mod} (2 \pi),(\te_b+\frac{\pi}{2}) \, \text{\rm mod} (2 \pi)}}}
\newcommand{\Lazzqq}{\ensuremath{\La_{0,0}^{\frac{\pi}{2},\frac{\pi}{2}}}}

\newcommand{\Hte}{\ensuremath{H_{\te_a,\te_b}}}

\DeclareMathOperator{\Col}{Col}

\DeclareMathOperator{\SL}{SL}
\DeclareMathOperator{\AC}{AC}

\newtheorem{theorem}{Theorem}[section]

\newtheorem{lemma}[theorem]{Lemma}
\newtheorem{corollary}[theorem]{Corollary}
\newtheorem{hypothesis}[theorem]{Hypothesis}
\theoremstyle{definition}

\newtheorem{remark}[theorem]{Remark}

\newtheorem{example}[theorem]{Example}

\allowdisplaybreaks
\numberwithin{equation}{section}

\begin{document}

\title[Boundary Data Maps and Krein's Resolvent Formula]{Boundary Data Maps
and Krein's Resolvent Formula for Sturm--Liouville Operators \\ on a Finite Interval}

\author[S.\ Clark]{Stephen Clark}
\address{Department of Mathematics \& Statistics,
Missouri University of Science and Technology, Rolla, MO 65409, USA}
\email{sclark@mst.edu}
\urladdr{http://web.mst.edu/\~{}sclark/}

\author[F.\ Gesztesy]{Fritz Gesztesy}
\address{Department of Mathematics,
University of Missouri, Columbia, MO 65211, USA}
%\email{\mailto{gesztesyf@missouri.edu}}
\email{gesztesyf@missouri.edu}
%\urladdr{\url{http://www.math.missouri.edu/personnel/faculty/gesztesyf.html}}
\urladdr{http://www.math.missouri.edu/personnel/faculty/gesztesyf.html}

\author[R.\ Nichols]{Roger Nichols}
\address{Department of Mathematics,
University of Missouri, Columbia, MO 65211, USA}
\email{nicholsrog@missouri.edu}
\urladdr{http://www.math.missouri.edu/personnel/other/nicholsr.html}

\author[M.\ Zinchenko]{Maxim Zinchenko}
\address{Department of Mathematics,
University of Central Florida, Orlando, FL 32816, USA}
\email{maxim@math.ucf.edu}
\urladdr{http://www.math.ucf.edu/~maxim/}

%\dedicatory{.}
\thanks{Based upon work partially supported by the US National Science
Foundation under Grant No.\ DMS 0965411.}
\date{\today}
%\date{, 2010}

%%%%%%%%%%%%%%%%%%%%%%%%%%%%%%%%%%%%%%%%
%%%%%%%%%%%%%%%%%%%%%%%%%%%%%%%%%%%%%%%%
\begin{abstract}
We continue the study of boundary data maps, that is, generalizations of spectral parameter
dependent Dirichlet-to-Neumann maps for (three-coefficient) Sturm--Liouville operators on the finite
interval $(a,b)$, to more general boundary conditions, began in \cite{CGM10}
and \cite{GZ11}. While these earlier studies of boundary data maps focused on the case
of general separated boundary conditions at $a$ and $b$, the present work develops a
unified treatment for all possible self-adjoint boundary conditions (i.e., separated as well
as non-separated ones).

In the course of this paper we describe the connections with Krein's resolvent formula for
self-adjoint extensions of the underlying minimal Sturm--Liouville operator (parametrized in terms
of boundary conditions), with some emphasis on the Krein extension, develop the basic
trace formulas for resolvent differences of self-adjoint extensions, especially, in terms of
the associated spectral shift functions, and describe the connections between various
parametrizations of all self-adjoint extensions, including the precise relation to von Neumann's
basic parametrization in terms of unitary maps between deficiency subspaces.
\end{abstract}
%%%%%%%%%%%%%%%%%%%%%%%%%%%%%%%%%%%%%%%%
%%%%%%%%%%%%%%%%%%%%%%%%%%%%%%%%%%%%%%%%

\maketitle

{\scriptsize \tableofcontents}

%%%%%%%%%%%%%%%%%%%%%%%%%%%%%%%%%%%%%%%%
%%%%%%%%%%%%%%%%%%%%%%%%%%%%%%%%%%%%%%%%
\section{Introduction} \lb{s1}
%%%%%%%%%%%%%%%%%%%%%%%%%%%%%%%%%%%%%%%%
%%%%%%%%%%%%%%%%%%%%%%%%%%%%%%%%%%%%%%%%

The principal theme developed in this paper concerns a detailed treatment of generalizations of the spectral parameter dependent Dirichlet-to-Neumann map for (three-coefficient) Sturm--Liouville operators on the finite interval $(a,b)$ to that for all self-adjoint boundary conditions. While the earlier treatments of boundary data maps in \cite{CGM10} and \cite{GZ11} focused on the special case of separated boundary conditions at $a$ and $b$, this paper now treats the case of all self-adjoint boundary conditions in a unified matter. Applications of the formalism discussed in this paper include the precise connections with Krein's resolvent formula for self-adjoint extensions of the underlying minimal Sturm--Liouville operator (parametrized in terms of boundary conditions); in particular, we will describe in detail the connections with Krein's extension of the minimal operator. We also offer a systematic treatment of the basic trace formulas for resolvent differences of self-adjoint extensions, including the connection with the associated spectral shift functions. Moreover,
we describe the interrelations between various parametrizations of all self-adjoint extensions of the minimal Sturm--Liouville operator, including the precise connection with von Neumann's basic parametrization.

We turn to a brief description of the content of this paper: Section \ref{s1a} recalls a variety of convenient parametrizations of all self-adjoint extensions associated with a regular, symmetric, second-order differential expression. This section is of an introductory character and serves as background material for the bulk of this paper. Section \ref{s2} is devoted to a comprehensive discussion of all self-adjoint extensions of the minimal Sturm--Liouville operator in terms of Krein's formula for resolvent differences, given the 
Sturm--Liouville operator with Dirichlet boundary conditions at $a$ and $b$ as a convenient reference operator. In particular, we carefully delineate the cases of separated and non-separated self-adjoint boundary conditions. We conclude this section with a detailed description of the Krein extension of the minimal Sturm--Liouville operator (this result appears to be new in the general case presented in Example \ref{e3.3}). Boundary data maps for general self-adjoint extensions of the minimal operator are the principal topic in Section \ref{s4}. Special emphasis is put on a unified treatment of all self-adjoint extensions (i.e., separated and non-separated ones). In particular, the resolvent difference between {\it any} pair of self-adjoint extensions of the minimal Sturm--Liouville operator is characterized in terms of the general boundary data map and associated boundary trace maps. Again, the precise connection with Krein's resolvent formula is established. Trace formulas for resolvent differences and associated spectral shift functions and symmetrized perturbation determinants are the focus of
Section \ref{s5}. In particular, it is shown that in the general non-degenerate case, the determinant of the boundary data map coincides with the symmetrized perturbation determinant up to a spectral parameter independent constant (the latter depends on the boundary conditions involved). 
In Section \ref{s7} we provide the precise connection with von Neumann's parametrization and the boundary data map $\La_\ABp^\AB(\cdot)$, the principal object studied in this paper. 
A very brief outlook on the applicability of boundary data maps to inverse spectral problems is provided in our last Section \ref{s6}.
(A detailed discussion of this circle of ideas is beyond the scope of this paper and hence will appear elsewhere.) 

To achieve a certain degree of self-containment, we also offer Appendix \ref{sB} which recalls the 
basics of Krein's resolvent formula for any pair of self-adjoint extensions of a symmetric operator 
of finitely-many (equal) deficiency indices.

Finally, we briefly summarize some of the notation used in this paper: Let $\cH$ be
a separable complex Hilbert space, $(\cdot,\cdot)_{\cH}$ the scalar product in $\cH$
(linear in the second argument), and $I_{\cH}$ the identity operator in $\cH$.
Next, if $T$ is a linear operator mapping (a subspace of) a
Banach space into another, then $\dom(T)$ and $\ker(T)$ denote the
domain and kernel (i.e., null space) of $T$.
The closure of a closable operator $S$ is denoted by $\ol S$. At times, and only for typographical reasons, we will also use $S^{\rm cl}$ for the closure of $S$. The spectrum, essential spectrum, discrete spectrum, and resolvent set
of a closed linear operator in $\cH$ will be denoted by $\sigma(\cdot)$,
$\sigma_{\rm ess}(\cdot)$, $\sigma_{\rm d}(\cdot)$, and $\rho(\cdot)$, respectively.
The Banach space of bounded linear operators on $\cH$ is
denoted by $\cB(\cH)$, the analogous notation $\cB(\cX_1,\cX_2)$,
will be used for bounded operators between two Banach spaces $\cX_1$ and
$\cX_2$. Moreover, $\cY_1 \dotplus \cY_2$ denotes the (not necessarily orthogonal)
direct sum of the subspaces $\cY_1$ and $\cY_2$ of a Banach (or Hilbert) space $\cY$.

The Banach space of compact operators
defined on $\cH$ is denoted by $\cB_{\infty}(\cH)$ and the $\ell^p$-based trace
ideals are denoted by $\cB_p(\cH)$, $p \geq 1$. The Fredholm determinant for
trace class perturbations of the identity in $\cH$ is denoted by ${\det}_{\cH}(\cdot)$,
the trace for trace class operators in $\cH$ will be denoted by ${\tr}_{\cH}(\cdot)$.

For brevity, the identity operator in $L^2((a,b); rdx)$ will be denoted by
$I_{(a,b)}$ and that in $\bbC^n$ by $I_n$, $n\in\bbN$. For simplicity of
notation, the subscript
$L^2((a,b); rdx)$ will typically be omitted in the scalar product
$(\cdot, \cdot)_{L^2((a,b); rdx)}$ in the proofs of our results in
Sections \ref{s2}--\ref{s5}. For an $n \times n$ matrix
$M \in \bbC^{n\times n}$, its operator norm, $\|M\|_{\cB(\bbC^n)}$, will
simply be abbreviated by $\|M\|$.

%%%%%%%%%%%%%%%%%%%%%%%%%%%%%%%%%%%%%%%%
%%%%%%%%%%%%%%%%%%%%%%%%%%%%%%%%%%%%%%%%
\section{Basics on the Classification and Parametrization of All Self-Adjoint Regular 
Sturm--Liouville Operators}
\lb{s1a}
%%%%%%%%%%%%%%%%%%%%%%%%%%%%%%%%%%%%%%%%
%%%%%%%%%%%%%%%%%%%%%%%%%%%%%%%%%%%%%%%%

In this section we recall several convenient parametrizations of all self-adjoint
extensions associated with a regular, symmetric, second-order differential expression
as discussed in detail, for instance, in \cite[Theorem\ 13.15]{We03} and \cite[Theorem\ 10.4.3]{Ze05}. 
While the first part of this section is of an introductory character and serves as background material 
for the bulk of this paper, its second part provides a detailed discussion of the extent to which these 
parametrizations uniquely characterize self-adjoint extensions. 

Throughout this paper we make the following set of assumptions:

%%%%%%%%%%%%
\begin{hypothesis} \lb{hA.1}
Suppose $p, q, r$ satisfy the following conditions: \\
$(i)$ \,\, $r>0$ a.e.\ on $(a,b)$, $r\in L^1((a,b); dx)$. \\
$(ii)$ \, $p>0$  a.e.\ on $(a,b)$, $1/p\in L^1((a,b); dx)$. \\
$(iii)$ $q\in L^1((a,b); dx)$, $q$ is real-valued a.e.\ on $(a,b)$.
\end{hypothesis}
%%%%%%%%%%%%

Given Hypothesis \ref{hA.1}, we take $\tau$ to be the Sturm--Liouville-type differential expression
defined by
\begin{equation}\label{A.1}
\tau=\frac{1}{r(x)}\left[-\frac{d}{dx}p(x)\frac{d}{dx} + q(x)\right],\quad x\in(a,b),\quad -\infty < a < b < \infty,
\end{equation}
and note that $\tau$ is regular on $[a,b]$. In addition, the following convenient notation for the 
\emph{first quasi-derivative} is introduced, 
\begin{equation} 
y^{[1]}(x)=p(x)y'(x) \, \text{ for a.e.\ $x\in(a,b)$, } \, y \in AC([a,b]).     
\end{equation}
Here $\AC([a,b])$ denotes the set of absolutely continuous functions on $[a,b]$. 

Given that $\tau$ is regular on $[a,b]$, the \emph{maximal operator}
$H_\max$ in $L^2((a,b);rdx)$ associated with $\tau$ is defined by
\begin{align}
&H_\max f=\tau f,   \\
& \, f \in \dom(H_\max)= \big\{g\in L^2((a,b);rdx) \, \big| \, g,  g^{[1]}\in \AC([a,b]); \,
\tau g\in  L^2((a,b);rdx)\big\},   \no 
\intertext{while the \emph{minimal operator} $H_{\min}$  in $L^2((a,b);rdx)$ associated with 
$\tau$ is given by} 
&H_{\min} f=\tau f,    \no \\
& \, f \in \dom(H_{\min})= \big\{g\in L^2((a,b);rdx) \, \big| \, g,  g^{[1]}\in \AC([a,b]);    \lb{A.6} \\
&\hspace*{7mm}g(a)=g^{[1]}(a)=g(b)=g^{[1]}(b)=0; \, \tau g\in L^2((a,b);rdx)\big\}.      \no 
\end{align}

We recall that an operator $\widetilde H$ in $L^2((a,b);rdx)$ is an extension
of $H_{\min}$, and denoted so by writing $H_{\min}\subseteq\widetilde H$, when
 $\dom(H_{\min})\subseteq\dom(\widetilde H)$,
and $\widetilde Hf=H_{\min} f$ for all $f\in\dom(H_{\min})$. $\widetilde H$ is \emph{symmetric} when its
adjoint operator $\widetilde H^*$ is an extension of $\widetilde H$, that is,
$\widetilde H \subseteq
\widetilde H^*$, and said to be \emph{self-adjoint} when $\widetilde H=\widetilde H^*$.
We note that the operator $H_{\min}$ is symmetric and that
\begin{equation}
H_{\min}^*=H_\max, \quad H_\max^*=H_{\min},
\end{equation}
(cf. Weidmann \cite[Theorem 13.8]{We00}).  If $\widetilde H$ is a symmetric extension of $H_{\min}$, then, by taking adjoints, one has
\begin{equation}
H_{\min}\subseteq \widetilde H \subseteq H_{\max},
\end{equation}
so that any symmetric extension of $H_{\min}$ is actually a restriction of $H_{\max}$.  Thus, in order to completely specify a symmetric (in particular, self-adjoint) extension of $H_{\min}$, it suffices to specify its domain of definition.

We now summarize material found, for instance, in \cite[Ch.\ 13]{We03} and 
\cite[Sects.\ 10.3, 10.4]{Ze05} in which self-adjoint extensions of the minimal operator $H_{\min}$ are characterized.

%%%%%%%%%%%%
\begin{theorem} [See, e.g., \cite{We03}, Theorem\ 13.14; \cite{Ze05}, Theorem\ 10.4.2] \lb{tA.2}
Assume Hypothesis \ref{hA.1} and suppose that $\widetilde H$  is an  extension of the minimal operator $H_{\min}$ defined in \eqref{A.6}. Then the following hold: \\
$(i)$ $\widetilde H$ is a self-adjoint extension of $H_{\min}$ if and only if there exist $2\times 2$ matrices $A$ and $B$ with complex-valued entries satisfying
\begin{equation} \lb{A.7}
\rank (A\ \ B)=2, \quad AJA^*=BJB^*, \quad
J=\begin{pmatrix} 0 & -1 \\ 1 & 0\end{pmatrix},
\end{equation}
 with $\widetilde Hf= \tau f$,  and where
\begin{equation} \lb{A.8}
\dom(\widetilde H)=\left\{ g\in\dom(H_{\max})\, \bigg| \, A\binom{g(a)}{g^{[1]}(a)} =
B\binom{g(b)}{g^{[1]}(b)}\right\}.
\end{equation}
Henceforth, the self-adjoint extension $\widetilde H$ corresponding to the matrices
 $A$ and $B$ will  be denoted by $H_{A,B}$. \\
$(ii)$ For $z\in\rho(H_{A,B})$, the resolvent $H_{A,B}$ is of the form
\begin{equation}
\begin{split}
\left((H_{A,B} - z I_{(a,b)})^{-1}f\right)(x)= \int_a^br(x')dx' \, G_{A,B}(z,x,x')f(x'),\\
f\in L^2((a,b);rdx),
\end{split}
\end{equation}
where the Green's function $G_{A,B}(z,x,x')$ is of the form given by
\begin{equation}
\hspace*{10pt}G_{A,B}(z,x,x')= \begin{cases}
\sum_{j,k=1}^{2} m_{j,k}^+(x)u_j(z,x)u_k(z,x'), & x'\le x,\\[2pt]
\sum_{j,k=1}^{2} m_{j,k}^-(x)u_j(z,x)u_k(z,x'), & x'> x.
\end{cases}
\end{equation}
Here $\{u_1, u_2\}$ represents a fundamental set of solutions for $(\tau -z)u=0$ and
$m_{j,k}^\pm$, $1 \leq j, k \leq 2$, are appropriate constants. In particular,
\begin{equation}
(H_{A,B} - z I_{(a,b)})^{-1} \in \cB\big(L^2((a,b); r dx)\big), \quad z \in \rho(H_{A,B}).
\end{equation}
$(iii)$ $H_{A,B}$ has purely discrete spectrum with eigenvalues of multiplicity
at most $2$. Moreover, if $\sigma(H_{A,B})=\{\lambda_{A,B,j}\}_{j\in\bbN}$,
then
\begin{equation}
\sum_{\substack{j\in\bbN\\ \lambda_{A,B,j}\ne 0}}|\lambda_{A,B,j}|^{-2}<\infty.
\end{equation}
\end{theorem}
%%%%%%%%%%

The characterization of self-adjoint extensions of $H_{\text{min}}$ in terms of pairs of matrices $(A,B)\in \bbC^{2\times 2}\times \bbC^{2\times 2}$ satisfying \eqref{A.7} is not unique in the sense that different pairs may lead to the same self-adjoint extension (i.e., it is possible that $H_{A,B}=H_{A',B'}$ with $(A,B)\neq (A',B')$) as the following simple example illustrates.

%%%%%%%%%%%%%
\begin{example}\lb{eA.6}
Let $A,B\in \bbC^{2\times 2}$ satisfy \eqref{A.7}.  If $C\in \bbC^{2\times 2}$ is
nonsingular, then the pair $(A',B')$ with $A'=CA$ and $B'=CB$ satisfies \eqref{A.7} and one readily 
verifies $\dom(H_{A,B})=\dom(H_{A',B'})$ so that $H_{A,B}=H_{A',B'}$.  One can actually show 
$H_{A,B}=H_{A',B'}$ {\it if and only if} $A'=CA$ and $B'=CB$ for a nonsingular matrix 
$C\in \bbC^{2\times 2}$, see Corollary \ref{cA.8} below.
\end{example}
%%%%%%%%%%%%%

Thus, Theorem \ref{tA.2}\,$(i)$ establishes the existence of a surjective
mapping from the set of all pairs $(A,B)\in \bbC^{2\times 2}\times
\bbC^{2\times 2}$ which satisfy \eqref{A.7}  to the set of self-adjoint
extensions of $H_{\text{min}}$,
\begin{equation}\lb{A.25}
(A,B)\mapsto H_{A,B},\ \text{where $A,B\in \bbC^{2\times 2}$ satisfy \eqref{A.7}.}
\end{equation}
Example \ref{eA.6} shows that the mapping \eqref{A.25} is not injective.

To obtain unique representations for the self-adjoint extensions of $H_{\min}$
described in Theorem \ref{tA.2}, we first take note of some additional
consequences for matrix pairs $(A,B)\in \bbC^{2\times 2}\times \bbC^{2\times 2}$
satisfying the conditions given in \eqref{A.7}.

%%%%%%%%%%
\begin{lemma} \lb{lA.3}
Let $A$ and $B$ be $2\times 2$ matrices with complex-valued entries which satisfy the conditions given in \eqref{A.7}. Then the following hold: \\
$(i)$ $\rank(A)=\rank(B)\ne 0$. \\
$(ii)$ $\Col(A)\cap\Col(B)=\{0\}$ if and only if $\rank(A)=\rank(B)=1$,
 where $\Col(A)$ \hspace*{5mm} represents the span of the columns of a matrix $A$.
\end{lemma}
%%%%%%%%%%%%%%%%%%%%%%%%%%%%%%%%%%%%%%%%%%
\begin{proof}
With $A$ and $B$ representing $2\times 2 $ matrices with
complex-valued entries that satisfy \eqref{A.7}, we note that $|\det(A)|^2=|\det(B)|^2$; which,
together with $\rank(A\ \ B)=2$, implies that $\rank(A)=\rank(B)$. Let
\begin{equation}
\rho=\rank(A)=\rank(B),
\end{equation}
and observe that while a priori $\rho \in \{0,1,2\}$, in fact, $\rho\ne 0$; for otherwise one concludes that $\rank (A\ B)=0$, in violation of the rank condition imposed in \eqref{A.7}.

 If $\rho=1$ and $\Col(A)\cap\Col(B)$ contains a nonzero vector, then
 $\Col(A)=\Col(B)$, and $\rank (A\ B)=1$, thus violating the rank condition given in \eqref{A.7}. If $\rho=2$, then $\Col(A)\cap\Col(B)=\bbC^2$. Thus,
 $\Col(A)\cap\Col(B)=\{0\}$ if and only if $\rho=1$ when $A$ and $B$
 satisfy the conditions provided in \eqref{A.7}.
\end{proof}
%%%%%%%%%%%

The next result provides unique characterizations for all self-adjoint
extensions of $H_{\min}$ and hence can be viewed as a refinement of Theorem \ref{tA.2}.

%%%%%%%%%%%
\begin{theorem} [See, e.g., \cite{We03}, Theorem\ 13.15; \cite{Ze05}, Theorem\ 10.4.3] \lb{tA.4}
Assume Hypothesis \ref{hA.1}.  Let $H_{\min}$ be the minimal operator associated with
$\tau$ and defined in \eqref{A.6} and $H_{A,B}$ a self-adjoint extension of the minimal operator as characterized
in Theorem\ \ref{tA.2}; then, the following hold: \\
$(i)$ $H_{A',B'}$ is a self-adjoint extension of $H_{\min}$ where $\rank(A')=\rank(B')=1$
if and only if $H_{A',B'}=H_{A,B}$,  where
\begin{equation} \lb{A.13}
A=\begin{pmatrix}\cos(\theta_a)&\sin(\theta_a)\\0&0 \end{pmatrix},\quad
B=\begin{pmatrix}0&0\\ -\cos(\theta_b)&\sin(\theta_b) \end{pmatrix},
\end{equation}
for a unique pair $\theta_a, \theta_b\in[0,\pi),$ where
\begin{equation} \lb{A.14}
\begin{split}
\dom(H_{A,B})=\{ g\in\dom(H_{\max}) \, | \,
 g(a)\cos(\theta_a) + g^{[1]}(a)\sin(\theta_a) & = 0,  \\
 g(b)\cos(\theta_b) - g^{[1]}(b)\sin(\theta_b) & = 0 \}.
\end{split}
\end{equation}
$(ii)$ $H_{A',B'}$ is a self-adjoint extension of $H_{\min}$  with
$\rank(A')=\rank(B')=2$ if and only if $H_{A',B'}=H_{A,B}$,  where
\begin{equation} \lb{A.15}
A=e^{i\phi} F,\quad B=I_2,
\end{equation}
for a unique $\phi\in[0,2\pi)$, and unique $F \in \SL_2(\mathbb{R})$, and
where
\begin{equation} \lb{A.16}
\dom(H_{A,B})=\left\{g\in\dom(H_{\max})
 \, \bigg| \, \binom{g(b)}{g^{[1]}(b)}=e^{i\phi} F
\binom{g(a)}{g^{[1]}(a)} \right\} .
\end{equation}
\end{theorem}
%%%%%%%%%%%%
\begin{proof}
First, with $A$ and $B$ defined either by \eqref{A.13} or by \eqref{A.15}, one
notes that $A$ and $B$ satisfy the properties in \eqref{A.7}. Hence, by
Theorem \ref{tA.2},  $H_{A,B}$ is a self-adjoint extension of the minimal operator $H_{\min}$. Clearly, when \eqref{A.13} holds, $\rank(A)=\rank(B)=1$, and when \eqref{A.15} holds, $\rank(A)=\rank(B)=2$.

With $A, B \in \bbC^{2 \times 2}$ satisfying \eqref{A.7}, $H_{A,B}$ as
characterized in Theorem \ref{tA.2} represents a self-adjoint extension of
the minimal operator $H_{\min}$.  By Lemma \ref{lA.3},
$\rho=\rank(A)=\rank(B)\ne 0$.

When $\rho=1$, the row vectors of $A$ and $B$ are linearly dependent, and
\begin{equation} \lb{A.17}
A=\begin{pmatrix}c\alpha_1& c\alpha_2\\ d\alpha_1& d\alpha_2 \end{pmatrix},\quad
B=\begin{pmatrix}c'\beta_1& c'\beta_2\\ d'\beta_1& d'\beta_2 \end{pmatrix},
\end{equation}
with $(c,d)\ne(0,0)$, $(c',d')\ne(0,0)$, $(\alpha_1,\alpha_2)\ne(0,0)$, $(\beta_1,\beta_2)\ne(0,0)$.
We note that $\Col(A)\cap\Col(B)=\{0\}$, which is equivalent to $\rho=1$, implies that
\begin{equation} \lb{A.18}
A\xi=B\eta\  \text{ only when }\, A\xi=0=B\eta,
\end{equation}
and hence that $AJA^*=0=BJB^*$. As a consequence, 
$\Im(\alpha_1\overline\alpha_2)= \Im(\beta_1\overline\beta_2)=0$; by which it follows that the 
$\bbC^2$-vectors $(\alpha_1,\alpha_2)$ and $(\beta_1,\beta_2)$ are  complex multiples of 
$\bbR^2$-vectors; thus, without loss of generality, we may assume in \eqref{A.17} that
\begin{equation} \lb{A.19}
\alpha_1=\cos(\theta_a), \quad \alpha_2=\sin(\theta_a),\quad
\beta_1= - \cos(\theta_b),\quad \beta_2=\sin(\theta_b),
\end{equation}
with $\theta_a,\theta_b\in[0,\pi)$. A second consequence of \eqref{A.18} is that
the domain of $H_{A,B}$, provided in \eqref{A.8}, is then given by
\begin{equation}
\dom(H_{A,B})=\left\{ g\in\dom(H_{\max})\,\bigg| \, A\binom{g(a)}{g^{[1]}(a)} =0=
B\binom{g(b)}{g^{[1]}(b)}\right\},
\end{equation}
and consequently, by \eqref{A.14}.

With $A$ and $B$ defined in \eqref{A.13}, and $\theta_a$ and $\theta_b$ defined
in \eqref{A.8}, $H_{A,B}$ is a self-adjoint extension with
$\rank(A)=\rank(B)=1$ as noted at the beginning of the proof.  Then, as a
consequence of the principle provided in \eqref{A.18} applied to $A$ and
$B$ , we see that $\dom(H_{A,B})=\dom(H_{A',B'})$ and hence that
$H_{A',B'}=H_{A,B}$.

Uniqueness of the representation given in \eqref{A.14} follows by noting that
if \eqref{A.14} holds for the distinct pairs $(\theta_a,\theta_b),
(\theta_a',\theta_b')\in [0,\pi)\times [0,\pi)$, then
$\sin(\theta_a-\theta_a')=\sin(\theta_b-\theta_b')=0$.

When $\rho=2$, then $A$ and $B$ are invertible and hence the boundary condition present in the definition of the domain of $H_{A,B}$ in \eqref{A.18} can be rewritten as
\begin{equation} \lb{A.21}
\binom{g(b)}{g^{[1]}(b)}=B_{b,a}\binom{g(a)}{g^{[1]}(a)},
\quad B_{b,a}=B^{-1}A^{}.
\end{equation}
With $B_{b,a}={B}^{-1}A^{}$, one notes that $B_{b,a}JB_{b,a}^*=J$; hence, that
$|\det(B_{b,a})|=1$, and as a consequence that $\det(B_{b,a})=e^{i\psi}$. In addition,
$B_{b,a}=-J(B_{b,a}^*)^{-1}J= e^{i\psi}\overline{B_{b,a}}$ and hence that
$B_{b,a}=e^{i\psi/2}F$, where $F$ is a $2\times 2$ matrix with real-valued entries for which $\det(F)=1$, that is, $F\in \SL_2(\bbR)$. Thus, the boundary condition
in \eqref{A.21} can now be rewritten as
\begin{equation} \lb{A.22}
\binom{g(b)}{g^{[1]}(b)}=e^{i\phi}F \binom{g(a)}{g^{[1]}(a)},\quad \phi\in[0,2\pi),
\; F \in \SL_2(\bbR).
\end{equation}

Uniqueness of the representation given in \eqref{A.16} follows by noting that
if \eqref{A.16} holds for the distinct pairs $(\phi, F), (\phi', F') \in
[0,2\pi)\times \SL_2(\mathbb{R}) $, then $e^{i(\phi'-\phi)}F'F^{-1}=I_2$ and
hence that $\phi'=\phi$, $F'=F$.
\end{proof}
%%%%%%%%%%%%

We now elaborate on two alternative characterizations for the self-adjoint
extensions of $H_{\min}$. These characterizations are summarized below in
Theorems \ref{tA.8} and \ref{tA.9}. The characterization given in
Theorem \ref{tA.8} is directly related to that found in Theorem \ref{tA.2}
and proves central to the development of Sections \ref{s4} and \ref{s5}. Like in
Theorem \ref{tA.2}, the extensions in Theorem \ref{tA.8} are not uniquely characterized.
By contrast, the characterization given in Theorem \ref{tA.9} provides a unique association between elements of the space of $2\times 2$ unitary matrices and
the set of all self-adjoint extensions of $H_{\min}$. Theorem \ref{tA.9} can be
derived from the theory of Hermitian relations as developed by Rofe-Beketov and Kholkin in 
Appendix A of \cite{RBK05}. In particular, Theorem \ref{tA.9} represents the scalar case of 
\cite[Theorem\ A.7]{RBK05}. We begin with a characterization of the self-adjoint extensions of 
$H_{\text{min}}$ in the language of {\it boundary trace maps} to be discussed in detail in Section \ref{s4}.

For a pair $\AB\in \bbC^{2\times 2}$ satisfying \eqref{A.7}, one introduces
the {\it general boundary trace map}, $\gamma_{A,B}$, associated with the
boundary $\{a,b\}$ of $(a,b)$ by
\begin{align}
\ga_{A,B} \colon \begin{cases}
C^1({[a,b]}) \rightarrow \bbC^2,
\\
u \mapsto A\begin{pmatrix}u(a)\\u^{[1]}(a)\end{pmatrix} - B\begin{pmatrix}u(b)\\u^{[1]}(b)\end{pmatrix}.
\end{cases}      \lb{A.26}
\end{align}
Comparing \eqref{A.26} with \eqref{A.8}, the boundary trace formalism allows one to write
\begin{equation} \lb{A.26b}
\dom(H_{A,B})=\{ g\in\dom(H_{\max})\, | \, \gamma_{A,B}g=0\}.
\end{equation}
Two special cases of \eqref{A.26} are to be distinguished, namely,
\begin{equation}
\gamma_D u = \begin{pmatrix}u(a)\\u(b)\end{pmatrix}, \quad
\gamma_N u = \begin{pmatrix}u^{[1]}(a)\\-u^{[1]}(b)\end{pmatrix}, \lb{A.27}
\end{equation}
that is,
\begin{align}
\gamma_D &= \gamma_{A_D,B_D}, \quad
A_D = \begin{pmatrix} 1 & 0 \\ 0 & 0 \end{pmatrix}, \quad
B_D = \begin{pmatrix} 0 & 0 \\ -1 & 0 \end{pmatrix}, \lb{A.28a}
\\
\gamma_N &= \gamma_{A_N,B_N}, \quad
A_N = \begin{pmatrix} 0 & 1 \\ 0 & 0 \end{pmatrix}, \quad
B_N = \begin{pmatrix} 0 & 0 \\ 0 & 1 \end{pmatrix}. \lb{A.28b}
\end{align}
The boundary trace maps $\gamma_D$ and $\gamma_N$ are canonical in the sense that any other boundary trace map $\gamma_{A,B}$ can be directly expressed in terms of $\gamma_D$ and $\gamma_N$ by
\begin{equation}\lb{A.29}
\gamma_{A,B}=D_{A,B}\gamma_{D}+N_{A,B}\gamma_{N},
\end{equation}
where the $2\times 2$ matrices $D_{A,B}$ and $N_{A,B}$ are given by
\begin{align}
D_{A,B}=\begin{pmatrix}A_{1,1}&-B_{1,1} \\ A_{2,1}&-B_{2,1}\end{pmatrix},\quad
N_{A,B}=\begin{pmatrix}A_{1,2}&B_{1,2} \\ A_{2,2}&B_{2,2}\end{pmatrix}. \lb{A.30}
\end{align}
By the elementary Lemma \ref{lA.7} below, the conditions in \eqref{A.7} are equivalent to
\begin{align}
\rank(D_{A,B}\ \ N_{A,B})=2, \quad D_{A,B}N_{A,B}^*=N_{A,B}D_{A,B}^*. \lb{A.31}
\end{align}
Therefore, one obtains an alternative characterization of all self-adjoint extensions of $H_{\text{min}}$ in terms of pairs of $2\times 2$ matrices satisfying the conditions in \eqref{A.31}.

The following result is elementary, but we include it for future reference.
%%%%%%%%%%%%%
\begin{lemma}\lb{lA.7}
Let $A$ and $B$ denote $2\times 2$ matrices with complex-valued entries.  Then $A$ and $B$ satisfy \eqref{A.7} if and only if
\begin{align}
X_D=\begin{pmatrix}A_{1,1}&-B_{1,1} \\ A_{2,1}&-B_{2,1}\end{pmatrix},\quad
X_{N}=\begin{pmatrix}A_{1,2}&B_{1,2} \\ A_{2,2}&B_{2,2}\end{pmatrix} \lb{A.31a}
\end{align}
satisfy
\begin{align}
\rank(X_D\ \ X_N)=2, \quad X_DX_N^*=X_NX_D^*. \lb{A.32}
\end{align}
\end{lemma}
%%%%%%%%%%%%%

%%%%%%%%%%%%%
\begin{proof}
The equivalence of the statements regarding the ranks in \eqref{A.7} and \eqref{A.32} is clear.  The equivalence of the matrix identities in \eqref{A.7} and \eqref{A.32} is an elementary calculation.
\end{proof}
%%%%%%%%%%%%%

The alternative characterization of self-adjoint extensions in terms of matrices satisfying \eqref{A.31} is summarized in the following theorem, and its connection to the characterization of self-adjoint extensions given by Theorem \ref{tA.2} is made explicit.

%%%%%%%%%%%%%
\begin{theorem}\lb{tA.8}
Assume Hypothesis \ref{hA.1}.  Suppose that $H$ is a symmetric extension of the minimal operator $H_{\min}$ defined in \eqref{A.6}. Then the following hold: \\
$(i)$  $H$ is a self-adjoint extension of $H_{\min}$ if and only if there exist $2\times 2$ matrices $X_D$ and $X_N$ with complex-valued entries satisfying
\begin{align}
\rank(X_D\ \ X_N)=2, \quad X_DX_N^*=X_NX_D^*. \lb{A.32a}
\end{align}
with
\begin{align}
Hf=\tau f,\quad f\in \dom(H)=\{g\in \dom(H_{\max})\, |\, X_D\gamma_Dg+X_N\gamma_Ng=0\}.\lb{A.33}
\end{align}
Henceforth, the self-adjoint extension $H$ corresponding to the matrices $X_D$ and $X_N$, and defined by \eqref{A.33}, will be denoted by $\widetilde H_{X_D,X_N}$.\\
$(ii)$  Given matrices $A,B\in \bbC^{2\times 2}$ satisfying \eqref{A.7}, the corresponding self-adjoint extension $H_{A,B}$ satisfies
\begin{equation}\lb{A.34}
H_{A,B}=\widetilde H_{X_D,X_N},
\end{equation}
with
\begin{equation}\lb{A.35}
X_D=D_{A,B},\quad X_N=N_{A,B},
\end{equation}
where $D_{A,B}$ and $N_{A,B}$ are defined by \eqref{A.30}.\\
$(iii)$  Given matrices $X_D,X_N\in \bbC^{2\times 2}$ satisfying \eqref{A.32} and the corresponding self-adjoint extension, $\widetilde H_{X_D,X_N}$, one has
\begin{equation}\lb{A.36}
\widetilde H_{X_D,X_N}=H_{A,B},
\end{equation}
with
\begin{align}
A=\begin{pmatrix}X_{D,1,1}&X_{N,1,1} \\ X_{D,2,1}&X_{N,2,1}\end{pmatrix},\quad
B=\begin{pmatrix}-X_{D,1,2}&X_{N,1,2} \\ -X_{D,2,2}&X_{N,2,2}\end{pmatrix}.\lb{A.37}
\end{align}
$(iv)$  $\widetilde H_{X_D,X_N}=\widetilde H_{X_D',X_N'}$ if and only if $X_D'=CX_D$ and $X_N'=CX_N$ for some nonsingular matrix $C\in \bbC^{2\times 2}$.
\end{theorem}
%%%%%%%%%%%%%

%%%%%%%%%%%%%
\begin{proof}
We begin with item $(i)$.  Suppose $H$ is defined by \eqref{A.33} for a pair of matrices $X_D,X_N\in \bbC^{2\times 2}$ which satisfy \eqref{A.32a}.  For $A$ and $B$ as defined in \eqref{A.37}, Lemma \ref{lA.7} guarantees that \eqref{A.7} is satisfied.  By construction (cf. \eqref{A.29} and \eqref{A.30}),
\begin{equation}\lb{A.38}
X_D\gamma_Du+X_N\gamma_Nu=\gamma_{A,B}u, \quad u\in \dom(H_{\max}).
\end{equation}
Thus, comparing \eqref{A.26b} with the definition of $\dom(H)$ in \eqref{A.33}, one concludes that
$u\in \dom(H_{\max})$ belongs to $\dom(H)$ if and only if it belongs to $\dom(H_{A,B})$.  As a result, $H=H_{A,B}$, and it follows that $H$ is a self-adjoint extension of $H_{\min}$.

Conversely, suppose $H$ is a self-adjoint extension of $H_{\min}$.  According to Theorem \ref{tA.2}, $H=H_{A,B}$ for a pair of matrices $A,B\in \bbC^{2\times 2}$ which satisfy \eqref{A.7}.  Choosing $X_D=D_{A,B}$ and $X_N=N_{A,B}$ with $D_{A,B}$ and $N_{A,B}$ as defined in \eqref{A.30}, Lemma \ref{lA.7} guarantees that $X_D$ and $X_N$ satisfy the conditions in \eqref{A.32a}.  Comparing \eqref{A.26b} with \eqref{A.29}, gives \eqref{A.33}.  This completes the proof of item $(i)$.

Items $(ii)$ and $(iii)$ are now immediate consequences of the proof of item $(i)$.

Sufficiency in item $(iv)$ is clear since
\begin{equation}
X_D\gamma_Du+X_N\gamma_Nu=0 \iff CX_D\gamma_Du+CX_N\gamma_Nu=0, \quad
u\in \dom(H_{\max}),
\end{equation}
for any nonsingular $C\in \bbC^{2\times 2}$.  In order to establish necessity, we now assume that
$\widetilde H_{X_D,X_N}=\widetilde H_{X_D',X_N'}$ or, equivalently, that
$\dom(\widetilde H_{X_D,X_N})=\dom(\widetilde H_{X_D',X_N'})$. One observes that the latter
equality (of domains) means that for $u\in \dom(H_{\max})$,
\begin{equation}\lb{A.38a}
X_D\gamma_Du+X_N\gamma_Nu=0\iff X_D'\gamma_Du+X_N'\gamma_Nu=0.
\end{equation}
Viewing the two equations in \eqref{A.38a} as two homogeneous linear systems (in the variables $(u(a),u(b),u^{[1]}(a),-u^{[1]}(b))$) with coefficient matrices $(X_D\ \ X_N)$ and $(X_D'\ \ X_N')$, the condition in \eqref{A.38a} implies that these two systems are equivalent systems.  Therefore, there exists a nonsingular matrix $C\in \bbC^{2\times 2}$ relating the coefficient matrices according to
\begin{equation}\lb{A.38b}
(X_D'\ \ X_N')=C(X_D\ \ X_N).
\end{equation}
Consequently, $X_D'=CX_D$ and $X_N'=CX_N$, implying the necessity claim. This completes the
proof of item $(iv)$.
\end{proof}
%%%%%%%%%%%%%

%%%%%%%%%%%%%
\begin{corollary}\lb{cA.8}
$H_{A,B}=H_{A',B'}$ if and only if $A'=CA$ and $B'=CB$ for some nonsingular matrix
$C\in \bbC^{2\times 2}$.
\end{corollary}
%%%%%%%%%%%%%

%%%%%%%%%%%%%
\begin{proof}
Sufficiency is clear (and has, in fact, already been mentioned in Example \ref{eA.6}).  In order to prove necessity, suppose $H_{A,B}=H_{A',B'}$.  Then by Theorem \ref{tA.8}\,$(ii)$,
\begin{equation}\lb{A.38c}
\widetilde H_{D_{A,B},N_{A,B}} = H_{A,B} = H_{A',B'} = \widetilde H_{D_{A',B'},N_{A',B'}},
\end{equation}
where the matrices $D_{A,B}$ and $N_{A,B}$ are defined by \eqref{A.30} and $D_{A',B'}$ and $N_{A',B'}$ are defined analogously.  In light of \eqref{A.38c} and Theorem \ref{tA.8}\,$(iv)$, there exists a nonsingular matrix $C\in \bbC^{2\times 2}$ such that
\begin{equation}\lb{A.38d}
D_{A',B'} = CD_{A,B} \qquad N_{A',B'} = CN_{A,B}.
\end{equation}
Explicitly computing the matrix products in \eqref{A.38d} yields
\begin{align}
\begin{pmatrix}A_{1,1}'&-B_{1,1}'\\A_{2,1}'&-B_{2,1}'\end{pmatrix} &=
\begin{pmatrix}C_{1,1}A_{1,1}+C_{1,2}A_{2,1}&-C_{1,1}B_{1,1}-C_{1,2}B_{2,1}\\
C_{2,1}A_{1,1}+C_{2,2}A_{2,1}&-C_{2,1}B_{1,1}-C_{2,2}B_{2,1}  \end{pmatrix}\no\\
\begin{pmatrix}A_{1,2}'&B_{1,2}'\\A_{2,2}'&B_{2,2}'\end{pmatrix} &=
\begin{pmatrix}C_{1,1}A_{1,2}+C_{1,2}A_{2,2}&C_{1,1}B_{1,2}+C_{1,2}B_{2,2}\\
C_{2,1}A_{1,2}+C_{2,2}A_{2,2}&C_{2,1}B_{1,2}+C_{2,2}B_{2,2}  \end{pmatrix}.\lb{A.38e}
\end{align}
By equating coefficients in \eqref{A.38e}, one concludes that $A'=CA$ and $B'=CB$.
\end{proof}
%%%%%%%%%%%%%

The next result provides a unique characterization of self-adjoint Sturm--Liouville extensions in terms of unitary $2 \times 2$ matrices:

%%%%%%%%%%%%%
\begin{theorem}\lb{tA.9}
Assume Hypothesis \ref{hA.1}.  Suppose that $H$ is a symmetric extension of the minimal operator $H_{\min}$ defined in \eqref{A.6}.  Then the following hold:\\
$(i)$  $H$ is a self-adjoint extension of $H_{\min}$ if and only if there exists a unitary $U\in \bbC^{2\times 2}$ with
\begin{equation}\lb{A.39}
Hf=\tau f, \quad f\in \dom(H)=\{g\in \dom(H_{\max})\, |\, i(U-I_2)\gamma_Dg=(U+I_2)\gamma_Ng\}.
\end{equation}
Henceforth, the self-adjoint extension $H$ corresponding to the unitary $2\times 2$ matrix $U$ and defined by \eqref{A.39} will be denoted by $H_U$.\\
$(ii)$  Given a unitary matrix $U\in \bbC^{2\times 2}$, the corresponding self-adjoint extension $H_U$ satisfies
\begin{equation}\lb{A.39a}
H_U=\widetilde H_{X_{D,U},X_{N,U}},
\end{equation}
where $X_{D,U},X_{N,U}\in \bbC^{2\times 2}$ are defined by
\begin{equation}\lb{A.39b}
X_{D,U}=\frac{i}{2}(I_2-U)\qquad  X_{N,U}=\frac{1}{2}(U+I_2).
\end{equation}
The matrix $U$ can be recovered from
\begin{equation}\lb{A.39ba}
U=(X_{D,U}+iX_{N,U})^{-1}(iX_{N,U}-X_{D,U}).
\end{equation}
$(iii)$  Given matrices $X_D,X_N\in \bbC^{2\times 2}$ satisfying \eqref{A.32a}, the corresponding self-adjoint extension $\widetilde H_{X_D,X_N}$ satisfies
\begin{equation}\lb{A.39c}
\widetilde H_{X_D,X_N}=H_{U_{X_D,X_N}},
\end{equation}
where $U_{X_D,X_N}\in \bbC^{2\times 2}$ is the unitary matrix
\begin{equation}\lb{A.39d}
U_{X_D,X_N}=(X_D+iX_N)^{-1}(iX_N-X_D).
\end{equation}
$(iv)$  $H_U=H_{U'}$ for $2\times 2$ unitary matrices $U$ and $U'$ if and only if $U=U'$.  Thus, the mapping $U\mapsto H_U$, $U\in \bbC^{2\times 2}$ unitary, is a bijection.
\end{theorem}
%%%%%%%%%%%%%
\begin{proof}
We begin with item $(i)$.  Suppose $H$ is defined by \eqref{A.39} for a fixed unitary matrix
$U\in \bbC^{2\times 2}$ and define the matrices $X_{D,U}$ and $X_{N,U}$ according to \eqref{A.39b}.  We claim that
\begin{equation}\lb{A.40a}
\text{$X_D:=X_{D,U}$ and $X_N:=X_{N,U}$ satisfy the conditions in \eqref{A.32a}.}
\end{equation}
Assuming \eqref{A.40a}, a self-adjoint extension $H_{X_{D,U},X_{N,U}}$ of $H_{\min}$ is defined by \eqref{A.33}.  Evidently, $u\in \dom(H_{\max})$ belongs to $\dom(H_{X_{D,U},X_{N,U}})$ if and only if it belongs to $\dom(H)$ as defined by \eqref{A.39}; hence, $\dom(H)=\dom(H_{X_{D,U},X_{N,U}})$.  As a result, $H=H_{X_{D,U},X_{N,U}}$ is a self-adjoint extension of $H_{\min}$.  We now proceed to verify the claim in \eqref{A.40a}.  To this end, one computes (applying unitarity of $U$),
\begin{equation}\lb{A.41}
\rank(X_{D,U}\ \ X_{N,U})=\rank [(X_{D,U}\ \ X_{N,U})(X_{D,U}\ \ X_{N,U})^*]=\rank I_2=2.
\end{equation}
Once more, unitarity of $U$ yields
\begin{equation}\lb{A.42}
X_{D,U}X_{N,U}^*=\frac{i}{4}(U^*-U),
\end{equation}
and consequently,
\begin{equation}\lb{A.43}
X_{N,U}X_{D,U}^*=(X_{D,U}X_{N,U}^*)^*=\bigg(\frac{i}{4}(U^*-U) \bigg)^*=\frac{i}{4}(U^*-U)=X_{D,U}X_{N,U}^*.
\end{equation}
Hence, \eqref{A.40a} is established, completing the proof that $H$ defined by \eqref{A.39} is a self-adjoint extension of $H_{\min}$.

Conversely, supposing that $H$ is a self-adjoint extension of $H_{\min}$, we will show that $H$ satisfies \eqref{A.39} for some unitary matrix $U\in \bbC^{2\times 2}$.  Since $H$ must be a restriction of
$H_{\max}$, \eqref{A.39} reduces to proving the existence of a unitary $U\in \bbC^{2\times 2}$ for which
\begin{equation}\lb{A.43a}
\dom(H)=\{g\in \dom(H_{\max})\, | \, i(U-I_2)\gamma_D g=(U+I_2)\gamma_Ng\}.
\end{equation}
According to Theorem \ref{tA.8}\,$(i)$, $H=\widetilde H_{X_D,X_N}$ for two matrices $X_D,X_N\in \bbC^{2\times 2}$ satisfying \eqref{A.32a}, and hence $\dom(H)$ is characterized by
\begin{equation}\lb{A.43b}
\dom(H)=\{g\in \dom(H_{\max})\, | \, X_D\gamma_Dg+X_N\gamma_Ng=0\}.
\end{equation}
Next, one observes that the matrix $(X_D+iX_N)$ is nonsingular; in fact,
\begin{align}
\rank(X_D+iX_N)&=\rank[(X_D+iX_N)(X_D+iX_N)^*]\no\\
&= \rank(X_DX_D^*+X_NX_N^*)\lb{A.44}\\
&= \rank[(X_D\ \ X_N)(X_D\ \ X_N)^*]\no\\
&= \rank(X_D\ \ X_N)=2.\lb{A.45}
\end{align}
To get \eqref{A.44} and \eqref{A.45} one makes use of the fact that $X_D$ and $X_N$ satisfy the conditions in \eqref{A.32a}.  The matrix $U_{X_D,X_N}$ defined by \eqref{A.39d} is unitary since
\begin{align}
&U_{X_D,X_N}U_{X_D,X_N}^*\no\\
&\quad=(X_D+iX_N)^{-1}(iX_N-X_D)(-iX_N^*-X_D^*)(X_D^*-iX_N^*)^{-1}\no\\
&\quad=(X_D+iX_N)^{-1}(X_NX_N^*+X_DX_D^*)(X_D^*-iX_N^*)^{-1}\lb{A.47}\\
&\quad=(X_D+iX_N)^{-1}(X_D+iX_N)(X_D^*-iX_N^*)(X_D^*-iX_N^*)^{-1}\no\\
&\quad=I_2,\no
\end{align}
using the identity $X_DX_N^*=X_NX_D^*$ twice.  Finally, \eqref{A.43b} together with the following chain of equivalences,
\begin{align}
& i(U_{X_D,X_N}-I_2)\gamma_Du=(U_{X_D,X_N}+I_2)\gamma_Nu\no\\
& \quad \Longleftrightarrow  U_{X_D,X_N}(i\gamma_Du-\gamma_Nu)=i\gamma_Du+\gamma_Nu \no\\
& \quad \Longleftrightarrow  (iX_N-X_D)(i\gamma_Du-\gamma_Nu)=(X_D+iX_N)(i\gamma_Du+\gamma_Nu)\no\\
& \quad \Longleftrightarrow  -X_N\gamma_Du-iX_N\gamma_Nu-iX_D\gamma_Du+X_D\gamma_Nu\no\\
&\qquad \quad \qquad = iX_D\gamma_Du+X_D\gamma_Nu-X_N\gamma_Du+iX_N\gamma_Nu\no\\
& \quad \Longleftrightarrow  X_D\gamma_Du+X_N\gamma_Nu=0,\quad u\in \dom(H_{\max}),\lb{A.48}
\end{align}
yields \eqref{A.43a} with $U=U_{X_D,X_N}$ defined by \eqref{A.39d}.  This completes the proof of item $(i)$.

Regarding item $(ii)$, \eqref{A.39a} with \eqref{A.39b} is an immediate consequences of the proof
of item $(i)$, while \eqref{A.39ba} is an elementary calculation using \eqref{A.39b}.  Item $(iii)$ is an immediate consequence of the proof of item $(i)$.

Sufficiency in item $(iv)$ is clear.  To establish necessity, suppose $H_U=H_{U'}$ for some
$2\times 2$ unitary matrices $U$ and $U'$.  Then one has
\begin{equation}\lb{A.49}
\widetilde H_{X_{D,U},X_{N,U}}=H_U=H_{U'}=\widetilde H_{X_{D,U'},X_{N,U'}},
\end{equation}
with $X_{D,U}$ and $X_{N,U}$ as defined in \eqref{A.39b}, and $X_{D,U'}$ and $X_{N,U'}$ defined analogously.  By Theorem \ref{tA.8}\,$(iv)$, \eqref{A.49} implies the existence of a nonsingular matrix $C\in \bbC^{2\times 2}$ such that
\begin{equation}\lb{A.50}
X_{D,U'}=CX_{D,U}, \quad X_{N,U'}=CX_{N,U}.
\end{equation}
Using \eqref{A.39ba} to recover $U'$ along with the identities in \eqref{A.50}, one has
\begin{align}
U'&=(X_{D,U'}+iX_{N,U'})^{-1}(iX_{N,U'}-X_{D,U'})\no\\
&=(CX_{D,U}+iCX_{N,U})^{-1}(iCX_{N,U}-CX_{D,U})\no\\
&=(X_{D,U}+iX_{N,U})^{-1}C^{-1}C(iX_{N,U}-X_{D,U})\no\\
&=(X_{D,U}+iX_{N,U})^{-1}(iX_{N,U}-X_{D,U})\no\\
&=U.\lb{A.51}
\end{align}
To get \eqref{A.51}, one applies the reconstruction formula in \eqref{A.39ba} (this time for $U$).
\end{proof}
%%%%%%%%%%%%%

Interest in the issue of parametrizing self-adjoint extensions was revived by Kostrykin and Schrader in the context of quantum graphs in \cite{KS99}, \cite{KS00}. In addition to the fundamental treatment of unique characterizations of all self-adjoint extensions in terms of unitary matrices and boundary conditions of the type appearing in \cite[Theorem\ A.7]{RBK05}, the corresponding extension to the more general case of Laplacians on quantum graphs has also been studied in \cite{BL10}, \cite{Ha00a}--\cite{Ha05},
\cite[Ch.\ 3]{Ku12}, \cite{KN10}, and \cite[Sect.\ 3]{No07}. The characterization of self-adjoint extensions
in terms of pairs of matrices $X_D,X_N\in \bbC^{2\times 2}$ satisfying \eqref{A.32} is given in \cite{BL10}
in the more general context of Laplacians on quantum graphs.

%%%%%%%%%%%%%%%%%%%%%%%%%%%%%%%%%%%%%%%%
%%%%%%%%%%%%%%%%%%%%%%%%%%%%%%%%%%%%%%%%
\section{Self-Adjoint Extensions in Terms of Krein's Formula}    \lb{s2}
%%%%%%%%%%%%%%%%%%%%%%%%%%%%%%%%%%%%%%%%
%%%%%%%%%%%%%%%%%%%%%%%%%%%%%%%%%%%%%%%%

The principal aim in this section is to relate resolvents of different self-adjoint extensions of 
$H_{\min}$ via Krein's resolvent formula. For an abstract approach to the latter we refer to 
Appendix \ref{sB}.

In accordance with Theorem \ref{tA.4}, we now introduce the following two families of self-adjoint 
extensions of the minimal operator $H_{\min}$: The operator $\Hte$ in $L^2((a,b);rdx)$,  
\begin{align}
& \Hte f= \tau f,  \quad \te_a, \te_b\in [0,\pi),     \no   \\
& \, f\in\dom(\Hte)=\big\{ g \in L^2((a,b);rdx)\,\big|\, g, g^{[1]} \in \AC([a,b]);  \lb{3.1}  \\
& \hspace*{7mm} \cos(\theta_a)g(a)+\sin(\theta_a)g^{[1]}(a)=0, \,
\cos(\theta_b)g(b)-\sin(\theta_b)g^{[1]}(b)=0; \no \\
& \hspace*{7.5cm}\tau g \in L^2((a,b);rdx)\big\},   \no
\end{align}
and for each $F=(F_{j,k})_{1\leq j,k\leq 2}\in \SL_2(\bbR)$
and $\phi \in [0,2\pi)$, the self-adjoint extension $H_{F,\phi}$ in $L^2((a,b); rdx)$
of $H_{\min}$ defined by
\begin{align}
&H_{F,\phi}f=\tau f, \quad F \in \SL_2(\bbR), \ \phi \in [0,2\pi),\no\\
& \, f\in \dom(H_{F,\phi})= \bigg\{g\in L^2((a,b);rdx) \,\bigg|\, g,g^{[1]}\in \AC([a,b]);
\lb{3.4}\\
& \hspace*{3cm} \begin{pmatrix} g(b) \\ g^{[1]}(b) \end{pmatrix}=
e^{i\phi} F \begin{pmatrix} g(a) \\ g^{[1]}(a) \end{pmatrix}; \,
\tau g\in L^2((a,b);rdx)\bigg\}.   \no
\end{align}

As discussed in detail in Section \ref{s1a}, $\Hte$ and $H_{F,\phi}$ characterize all
self-adjoint extensions of $H_{\min}$.

The generalized Cayley transform of $H_{0,0}$ (a convenient reference operator) is defined by
\begin{align}
\begin{split}
U_{z,z'}&=(H_{0,0}-z'I_{(a,b)})(H_{0,0}-zI_{(a,b)})^{-1}      \\
&=I_{(a,b)}+(z-z')(H_{0,0}-zI_{(a,b)})^{-1}, \quad z,z'\in \rho(H_{0,0}),  \lb{3.6}
\end{split}
\end{align}
and forms a bijection from $\ker(H_{\max} - z'I_{(a,b)})$ to
$\ker(H_{\max} - zI_{(a,b)})$ (cf.\ \eqref{B.18}, Appendix A).  In particular,
\begin{equation}\lb{3.7}
\text{dim}(\ker(H_{\max} - zI_{(a,b)}))=2, \quad z\in \rho(H_{0,0}).
\end{equation}
For each $z\in \rho(H_{0,0})$, a basis for $\ker(H_{\max} - zI_{(a,b)})$, denoted $\{u_j(z,\cdot)\}_{j=1,2}$, is fixed by specifying
\begin{equation}\lb{3.8}
\begin{split}
&u_1(z,a)=0, \quad u_1(z,b)=1,  \\
&u_2(z,a)=1, \quad u_2(z,b)=0,
\end{split}
\quad z\in \rho(H_{0,0}).
\end{equation}
One verifies
\begin{equation} \lb{3.9}
\begin{split}
U_{z,z'}u_1(z',\cdot)&=u_1(z,\cdot), \\
U_{z,z'}u_2(z',\cdot)&=u_2(z,\cdot),
\end{split}
\quad   j\in \{1,2\}, \; z,z'\in \rho(H_{0,0}).
\end{equation}
The identities \eqref{3.9} follow easily from the representation \eqref{3.6}.  In fact, since $U_{z,z'}$ maps into $\ker(H_{\max} - zI_{(a,b)})$,
\begin{equation} \lb{3.11}
\begin{split}
U_{z,z'}u_1(z',\cdot)&=c_{1,1}u_1(z,\cdot)+c_{1,2}u_2(z,\cdot),   \\
U_{z,z'}u_2(z',\cdot)&=c_{2,1}u_1(z,\cdot)+c_{2,2}u_2(z,\cdot),
\end{split}
\quad z,z'\in \rho(H_{0,0}),
\end{equation}
for certain scalars $c_{11},c_{12},c_{21},c_{22}\in \mathbb{C}$.
On the other hand, by \eqref{3.6},
\begin{align}
U_{z,z'}u_1(z',\cdot)&= u_1(z',\cdot)+(z-z')(H_{0,0}-zI_{(a,b)})^{-1}u_1(z',\cdot)\lb{3.13}\\
U_{z,z'}u_2(z',\cdot)&= u_2(z',\cdot)+(z-z')(H_{0,0}-zI_{(a,b)})^{-1}u_2(z',\cdot),\lb{3.14}\\
&\hspace*{4.93cm} z,z'\in \rho(H_{0,0}),\no
\end{align}
so that
\begin{equation}\lb{3.15}
\begin{split}
\big[U_{z,z'}u_1(z',\cdot)\big](a)=u_1(z',a), \quad \big[U_{z,z'}u_1(z',\cdot)\big](b)=u_1(z',b), \ \\
\big[U_{z,z'}u_2(z',\cdot)\big](a)=u_2(z',a), \quad \big[U_{z,z'}u_2(z',\cdot)\big](b)=u_2(z',b).
\end{split}
\end{equation}
Evaluating \eqref{3.11} at $a$ (resp., $b$) and comparing to \eqref{3.15} yields $c_{1,2}=0$ and $c_{2,2}=1$ (resp., $c_{1,1}=1$ and $c_{2,1}=0$), implying \eqref{3.9}.  Moreover, due to reality of the coefficients $p$, $q$, and $r$, one also verifies that
\begin{equation}\lb{3.15a}
\overline{u_j(z,\cdot)}=u_j(\overline{z},\cdot),\quad j=1,2, \; z\in \rho(H_{0,0}).
\end{equation}

Using a resolvent formula due to Krein (cf.\ \eqref{B.16}, Appendix A), the
next result provides a characterization, in terms of the Dirichlet resolvent
$(H_{0,0} - zI_{(a,b)})^{-1}$, for the resolvents of all self-adjoint
extensions of the minimal operator $H_\min$.
%%%%%%%%%%
\begin{theorem} \lb{t3.1}
Assume Hypothesis \ref{hA.1}, let $\theta_a, \theta_b \in [0, \pi)$, and denote by $u_j(z,\cdot)$, $j=1,2$, the basis for $\ker(H_{\max}-zI_{(a,b)})$ as defined in \eqref{3.8}. \\
$(i)$  If $\theta_a\neq 0$ and $\theta_b\neq 0$, then the maximal common part
$($cf.\ Appendix \ref{sB}$)$ of $H_{\theta_a,\theta_b}$ and $H_{0,0}$ is $H_{\min}$.  The matrix
\begin{align}
&D_{\theta_a,\theta_b}(z)= \begin{pmatrix}
\cot(\theta_b) - u_1^{[1]}(z,b) & - u_2^{[1]}(z,b)\\
u_1^{[1]}(z,a) & \cot(\theta_a) + u_2^{[1]}(z,a)
\end{pmatrix},  \quad z\in \rho(H_{\theta_a,\theta_b})\cap\rho(H_{0,0}),      \lb{3.16}
\end{align}
is invertible and
\begin{align}\lb{3.17}
&(H_{\theta_a,\theta_b}-zI_{(a,b)})^{-1}= (H_{0,0}-zI_{(a,b)})^{-1}    \\
& \quad
- \sum_{j,k=1}^2 D_{\theta_a,\theta_b}(z)_{j,k}^{-1}
(u_k(\overline{z},\cdot),\cdot)_{L^2((a,b); rdx)} u_j(z,\cdot),
\quad z\in \rho(H_{\theta_a,\theta_b})\cap\rho(H_{0,0}).     \no
\end{align}
$(ii)$  If $\theta_a\neq 0$, then the maximal common part of $H_{\theta_a,0}$ and $H_{0,0}$ is the restriction, $\widetilde{H}_{\text{min}}$, of $H_{\max}$ with domain
\begin{equation}\lb{3.18}
\dom\big(\widetilde{H}_{\text{min}}\big)=\dom(H_{\max})\cap
\{g\in \AC([a,b]) \,|\, g(b)=g(a)=g^{[1]}(a)=0\}.
\end{equation}
The quantity
\begin{equation}\lb{3.19}
d_{\theta_a,0}(z)= \cot(\theta_a) + u_2^{[1]}(z,a), \quad
z\in \rho(H_{\theta_a,0})\cap \rho(H_{0,0}),
\end{equation}
is nonzero and
\begin{equation}\lb{3.20}
\begin{split}
& (H_{\theta_a,0}-zI_{(a,b)})^{-1}=(H_{0,0}-zI_{(a,b)})^{-1}  \\
& \quad
- d_{\theta_a,0}(z)^{-1}(u_2(\overline{z},\cdot),\cdot)_{L^2((a,b); rdx)}
u_2(z,\cdot),
\quad z\in \rho(H_{\theta_a,0})\cap \rho(H_{0,0}).
\end{split}
\end{equation}
$(iii)$  If $\theta_b\neq 0$, then the maximal common part of $H_{0,\theta_b}$ and $H_{0,0}$ is the restriction, $\widehat{H}_{\text{min}}$, of $H_{\max}$ with domain
\begin{equation}\lb{3.21}
\dom\big(\widehat{H}_{\text{min}}\big)=\dom(H_{\max})\cap
\{g\in \AC([a,b]) \,|\, g(b)=g(a)=g^{[1]}(b)=0\}.
\end{equation}
The quantity
\begin{equation}\lb{3.22}
d_{0,\theta_b}(z)= \cot(\theta_b) - u_1^{[1]}(z,b), \quad
z\in \rho(H_{0,\theta_b})\cap \rho(H_{0,0}),
\end{equation}
is nonzero and
\begin{equation}\lb{3.23}
\begin{split}
& (H_{0,\theta_b}-zI_{(a,b)})^{-1}=(H_{0,0}-zI_{(a,b)})^{-1}   \\
& \quad  - d_{0,\theta_b}(z)^{-1}
(u_1(\overline{z},\cdot),\cdot)_{L^2((a,b); rdx)} u_1(z,\cdot),   \quad
z\in \rho(H_{0,\theta_b})\cap \rho(H_{0,0}).
\end{split}
\end{equation}
\end{theorem}
%%%%%%%%%%
\begin{proof}
We begin with the {\it proof of item $(i)$}.  The maximal common part of $H_{\theta_a,\theta_b}$ and $H_{0,0}$ is $H_{\min}$ since $\theta_a\neq 0$ and $\theta_b\neq0$ imply
\begin{equation}\lb{3.24}
\dom(H_{\theta_a,\theta_b})\cap\dom(H_{0,0})=\dom(H_{\min}).
\end{equation}

By way of contradiction, suppose $\det(D_{\theta_a,\theta_b}(z_0))=0$ for some $z_0\in \rho(H_{\theta_a,\theta_b})\cap\rho(H_{0,0})$.  Then
\begin{align}
&\det\begin{pmatrix}
\cos(\theta_a)u_2(z_0,a)+\sin(\theta_a)u_2^{[1]}(z_0,a) & \cos(\theta_a)u_1(z_0,a)
+ \sin(\theta_a) u_1^{[1]}(z_0,a)\\
\cos(\theta_b)u_2(z_0,b)-\sin(\theta_b)u_2^{[1]}(z_0,b) & \cos(\theta_b)u_1(z_0,b)
- \sin(\theta_b) u_1^{[1]}(z_0,b)
\end{pmatrix}  \no \\
&\quad = \sin(\theta_a)\sin(\theta_b)\det(D_{\theta_a,\theta_b}(z))=0.\lb{3.25}
\end{align}
Thus, there exists a constant $c\in \mathbb{C}$ such that
\begin{align}
\cos(\theta_a)[u_1(z_0,a)+cu_2(z_0,a)]+\sin(\theta_a)[u_1^{[1]}(z_0,a)+cu_2^{[1]}(z_0,a)]=0,
\lb{3.26}\\
\cos(\theta_b)[u_1(z_0,b)+cu_2(z_0,b)]-\sin(\theta_b) [u_1^{[1]}(z_0,b)+cu_2^{[1]}(z_0,b)]=0.\lb{3.27}
\end{align}
As a result, $u_1(z_0,\cdot)+cu_2(z_0,\cdot)\in \dom(H_{\theta_a,\theta_b})$ is an eigenfunction with corresponding eigenvalue $z_0$, contradicting $z_0\in \rho(H_{\theta_a,\theta_b})$.

In order to prove \eqref{3.17}, it suffices to show
\begin{align}
g_f(z,\cdot) &:= (H_{0,0}-zI_{(a,b)})^{-1}f  \lb{3.28} \\
& \quad \;\, - \sum_{j,k=1}^2D_{\theta_a,\theta_b}(z)_{j,k}^{-1}(u_k(\overline{z},\cdot),f)_{L^2((a,b); rdx)}
u_j(z,\cdot)\in \dom(H_{\theta_a,\theta_b}),\no\\
& \hspace*{3.7cm} f\in L^2((a,b);rdx), \ z\in \rho(H_{\theta_a,\theta_b})\cap \rho(H_{0,0}).
\no
\end{align}
One then verifies that
\begin{equation}\lb{3.29}
\begin{split}
(H_{\theta_a,\theta_b}-zI_{(a,b)})g_f(z,\cdot)=(H_{\max}-zI_{(a,b)})g_f(z,\cdot)=f, \\
 f\in L^2((a,b);rdx), \; z\in \rho(H_{\theta_a,\theta_b})\cap \rho(H_{0,0}),
\end{split}
\end{equation}
since $H_{\max}$ is an extension of $H_{\theta_a,\theta_b}$ and $H_{0,0}$ and
$\{u_j(z,\cdot)\}_{j=1,2}\subseteq \ker(H_{\max}-z)$.  In order to show \eqref{3.28},
one need only to show that $g_f(z,\cdot)$ satisfies the boundary conditions in \eqref{3.1}.  One has
\begin{equation}  \lb{3.30}
\begin{split}
[(H_{0,0}-zI_{(a,b)})^{-1}f]^{[1]}(a)&=(u_2(\overline{z},\cdot),f)_{L^2((a,b); rdx)},  \\
[(H_{0,0}-zI_{(a,b)})^{-1}f]^{[1]}(b)&=-(u_1(\overline{z},\cdot),f)_{L^2((a,b); rdx)},
\end{split}
\quad z\in \rho(H_{0,0}),
\end{equation} 
which can be seen using the integral kernel for the resolvent of $H_{0,0}$,
\begin{align}
[(H_{0,0}-zI_{(a,b)})^{-1}f](x)&=W_{2,1}(z)^{-1} \bigg[u_2(z,x)\int_a^x r(x') dx'u_1(z,x')f(x')\no \\
&\hspace*{2.1cm} +u_1(z,x)\int_x^b r(x') dx'u_2(z,x')f(x')\bigg],    \no \\
&\hspace*{.7cm} f\in L^2((a,b);rdx), \ x\in [a,b], \  z\in \rho(H_{0,0}),   \lb{3.32}
\end{align}
where $W_{2,1}(z)$ denotes the Wronskian of $u_2(z,\cdot)$ and $u_1(z,\cdot)$. One recalls 
that the Wronskian of $f$ and $g$ is defined for a.e.\ $x \in (a,b)$ by
\begin{equation}
W(f,g)(x) = f(x) g^{[1]}(x) - f^{[1]}(x) g(x), \quad f, g \in AC([a,b]). 
\end{equation}
A short computation using \eqref{3.8} yields
\begin{equation}\lb{3.33}
W_{2,1}(z)=u_1^{[1]}(z,a)=-u_2^{[1]}(z,b), \quad z\in \rho(H_{0,0}).
\end{equation}
Differentiating \eqref{3.32} and then using \eqref{3.33} yields
\begin{align}
[(H_{0,0}-zI_{(a,b)})^{-1}f]^{[1]}(x)&=-\frac{u_2^{[1]}(z,x)}{u_2^{[1]}(z,b)} \int_a^xdx'u_1(z,x')f(x')\no \\
&\quad + \frac{u_1^{[1]}(z,x)}{u_1^{[1]}(z,a)}\int_x^bdx'u_2(z,x')f(x'),\lb{3.34}\\
& \hspace*{-1.16cm}f\in L^2((a,b);rdx), \ x\in [a,b], \ \ z\in \rho(H_{0,0}),  \no
\end{align}
and relations \eqref{3.30} now follow by evaluating \eqref{3.34} separately at $x=a$ and $x=b$, respectively.

Using \eqref{3.8} and \eqref{3.30}, one obtains
\begin{align}
& g_f(z,a)=\det(D_{\theta_a,\theta_b}(z))^{-1}\big[(-\cot(\theta_b)+ u_1^{[1]}(z,b))(u_2(\overline{z},\cdot),f)_{L^2((a,b); rdx)}\no \\
&\hspace*{4.2cm}+u_1^{[1]}(z,a)(u_1(\overline{z},\cdot),f)_{L^2((a,b); rdx)} \big],\lb{3.35}\\
& g_f(z,b)=\det(D_{\theta_a,\theta_b}(z))^{-1}\big[-(\cot(\theta_a)+ u_2^{[1]}(z,a))(u_1(\overline{z},\cdot),f)_{L^2((a,b); rdx)}\no \\
&\hspace*{4.2cm}-u_2^{[1]}(z,b)(u_2(\overline{z},\cdot),f)_{L^2((a,b); rdx)} \big],\lb{3.36}\\
& g_f^{[1]}(z,a)=(u_2(\overline{z},\cdot),f)_{L^2((a,b); rdx)}\no \\
&\quad+\det(D_{\theta_a,\theta_b}(z))^{-1} \big[(-\cot(\theta_a)-u_2^{[1]}(z,a))(u_1(\overline{z},\cdot),f)_{L^2((a,b); rdx)}u_1^{[1]}(z,a)\no \\
&\quad+(-\cot(\theta_b)+u_1^{[1]}(z,b))(u_2(\overline{z},\cdot),f)_{L^2((a,b); rdx)}u_2^{[1]}(z,a)\no \\
&\quad-u_2^{[1]}(z,b)(u_2(\overline{z},\cdot),f) u_1^{[1]}(z,a)+u_1^{[1]}(z,a)(u_1(\overline{z},\cdot),f)_{L^2((a,b); rdx)}u_2^{[1]}(z,a) \big],\lb{3.37}\\
&g_f^{[1]}(z,b)=-(u_1(\overline{z},\cdot),f)_{L^2((a,b); rdx)}+\det(D_{\theta_a,\theta_b}(z))^{-1}\no\\
&\quad \times \big[(-\cot(\theta_a)-u_2^{[1]}(z,a))(u_1(\overline{z},\cdot),f)_{L^2((a,b); rdx)}u_1^{[1]}(z,b)\no \\
&\qquad+(-\cot(\theta_b)+u_1^{[1]}(z,b))(u_2(\overline{z},\cdot),f)_{L^2((a,b); rdx)}u_2^{[1]}(z,b)\no \\
&\qquad-u_2^{[1]}(z,b)(u_2(\overline{z},\cdot),f)_{L^2((a,b); rdx)} u_1^{[1]}(z,b)\no\\
&\qquad+u_1^{[1]}(z,a)(u_1(\overline{z},\cdot),f)_{L^2((a,b); rdx)}u_2^{[1]}(z,b) \big],\lb{3.38}\\
&\hspace*{7mm}f\in L^2((a,b);rdx), \ z\in \rho(H_{\theta_a,\theta_b})\cap \rho(H_{0,0}),\no
\end{align}
and as a result, one verifies
\begin{align}
0&=\cos(\theta_a)g_f(z,a)+\sin(\theta_a)g_f^{[1]}(z,a),   \lb{3.39}\\
0&=\cos(\theta_b)g_f(z,b)-\sin(\theta_b)g_f^{[1]}(z,b),\lb{3.40}\\
f & \in L^2((a,b);rdx), \ z\in \rho(H_{\theta_a,\theta_b})\cap \rho(H_{0,0}).\no
\end{align}
{\it Proof of item $(ii)$}.  If $\theta_a\neq 0$ and $\theta_b=0$, then one verifies that
\begin{equation}\lb{3.41}
\dom(H_{\theta_a,0})\cap\dom(H_{0,0})=\dom(H_{\max})\cap
\{g\in \AC([a,b]) \,|\, g(b)=g(a)=g^{[1]}(a)=0\}.
\end{equation}
By definition, the maximal common part of $H_{\theta_a,0}$ and $H_{0,0}$ is the restriction of $H_{\max}$ to $\dom(H_{\theta_a,0})\cap\dom(H_{0,0})$, that is, the maximal common part of $H_{\theta_a,0}$ and $H_{0,0}$ is $\widetilde{H}_{\min}$
as defined in item $(ii)$.

By way of contradiction, suppose $d_{\theta_a,0}(z_0)=0$ for some $z_0\in \rho(H_{\theta_a,0})\cap\rho(H_{0,0})$.  Then
\begin{align}
0&= \sin(\theta_a)d_{\theta_a,0}(z_0)   \no\\
&=\cos(\theta_a)+\sin(\theta_a)u_2^{[1]}(z_0,a)\no\\
&=\cos(\theta_a)u_2(z_0,a)+\sin(\theta_a)u_2^{[1]}(z_0,a),\lb{3.42}
\end{align}
together with the trivial identity
\begin{equation}\lb{3.43}
0=\cos(0)u_2(z_0,b)-\sin(0)u_2^{[1]}(z_0,b),
\end{equation}
shows that $u_2(z_0,\cdot)$ is an eigenfunction of $H_{\theta_a,0}$ with eigenvalue $z_0$, contradicting $z_0\in \rho(H_{\theta_a,0})$.

To verify \eqref{3.20}, one only needs to show
\begin{align}
 g_f(z,\cdot)&\equiv(H_{0,0}-zI_{(a,b)})^{-1}f\no\\
&\quad -d_{\theta_a,0}(z)^{-1}(u_2(\overline{z},\cdot),f)_{L^2((a,b); rdx)}u_2(z,\cdot)\in \dom(H_{\theta_a,0}),   \lb{3.44} \\
& \hspace*{2.5cm} f\in L^2((a,b);rdx), \ z\in \rho(H_{\theta_a,0})\cap\rho(H_{0,0}). \no
\end{align}
Repeating the computation in \eqref{3.29}, the proof of \eqref{3.44} reduces to showing that $g_f(z,\cdot)$ satisfies the boundary conditions for $\dom(H_{\theta_a,0})$.  One computes
\begin{align}
g_f(z,b)&=0,  \lb{3.45}\\
g_f(z,a)&= d_{\theta_a,0}(z)^{-1}(u_2(\overline{z},\cdot),f)_{L^2((a,b); rdx)},  \lb{3.46}\\
g_f^{[1]}(z,a)&=(u_2(\overline{z},\cdot),f)_{L^2((a,b); rdx)} \no\\
&\quad + d_{\theta_a,0}(z)^{-1}(u_2(\overline{z},\cdot),f)_{L^2((a,b); rdx)}u_2^{[1]}(z,a),\lb{3.47}\\
& \hspace*{.55cm}f\in L^2((a,b);rdx), \ z\in \rho(H_{\theta_a,0})\cap \rho(H_{0,0}).\no
\end{align}
where the last equality makes use of \eqref{3.30}.  As a result,
\begin{align}
0&=(u_2(\overline{z},\cdot),f)_{L^2((a,b); rdx)}[\cos(\theta_a)d_{\theta_a,0}(z)^{-1}+\sin(\theta_a)
+ \sin(\theta_a)d_{\theta_a,0}(z)^{-1}u_2^{[1]}(z,a) ]    \no \\
&=\cos(\theta_a)g_f(z,a)+\sin(\theta_a)g_f^{[1]}(z,a),
\;\; f\in L^2((a,b);rdx), \ z\in \rho(H_{\theta_a,0})\cap \rho(H_{0,0}),  \lb{3.48}
\end{align}
and \eqref{3.44} follows.

\noindent
{\it Proof of item $(iii)$}. As this is very similar to the proof of item $(ii)$, we only sketch an outline.  The statement regarding the maximal common part follows since, in the case
$\theta_a=0$ and $\theta_b\neq 0$,
\begin{align}\lb{3.49}
\begin{split}
&\dom(H_{0,\theta_b})\cap\dom(H_{0,0})   \\
& \quad =\dom(H_{\max})\cap
\{g\in \AC([a,b]) \,|\, g(b)=g(a)=g^{[1]}(b)=0\}.
\end{split}
\end{align}

If $d_{0,\theta_b}(z_0)=0$, then $u_1(z_0,\cdot)$ is an eigenfunction of $H_{0,\theta_b}$ and $z_0$ is the corresponding eigenvalue, a contradiction.  Verification of \eqref{3.23} reduces to showing
\begin{align}
g_f(z,\cdot)&\equiv(H_{0,0}-zI_{(a,b)})^{-1}f \no\\
&\quad - d_{0,\theta_b}(z)^{-1}(u_1(\overline{z},f),\cdot)_{L^2((a,b); rdx)}u_1(z,\cdot)\in \dom(H_{0,\theta_b}),\lb{3.50} \\
& \hspace*{2.5cm} f\in L^2((a,b);rdx), \ z\in \rho(H_{0,\theta_b})\cap\rho(H_{0,0}),\no
\end{align}
which, in turn, reduces to verifying $g_f(z,\cdot)$ satisfies the boundary conditions for $\dom(H_{0,\theta_b})$:
\begin{align}
& g_f(z,a)=0,  \lb{3.51}\\
& \cos(\theta_b)g_f(z,b)-\sin(\theta_b)g_f^{[1]}(z,b) = 0,\lb{3.52}\\
& f\in L^2((a,b);rdx), \ z\in \rho(H_{0,\theta_b})\cap\rho(H_{0,0}).    \no
\end{align}
\eqref{3.51} and \eqref{3.52} are the results of straightforward calculations.
\end{proof}
%%%%%%%%%%

%%%%%%%%%%
\begin{theorem} \lb{t3.2}
Assume Hypothesis \ref{hA.1}, let $F=(F_{j,k})_{1\leq j,k\leq 2}\in \SL_2(\bbR)$ and $\phi \in [0,2\pi)$, and denote by $u_j(z,\cdot)$, $j=1,2$, the basis for $\ker(H_{\max}-zI_{(a,b)})$ as defined in \eqref{3.8}.  \\
$(i)$ If $F_{1,2}\neq 0$, then the maximal common part of $H_{F,\phi}$ and $H_{0,0}$ is $H_{\min}$.  The matrix
\begin{equation}\lb{3.53}
\begin{split}
Q_{F,\phi}(z)=\begin{pmatrix}
\frac{F_{2,2}}{F_{1,2}} - u_1^{[1]}(z,b) & \frac{-1}{e^{-i\phi}F_{1,2}} - u_2^{[1]}(z,b)\\
\frac{-1}{e^{i\phi}F_{1,2}} + u_1^{[1]}(z,a) & \frac{F_{1,1}}{F_{1,2}} + u_2^{[1]}(z,a)
\end{pmatrix},\\
z\in \rho(H_{F,\phi})\cap\rho(H_{0,0}),
\end{split}
\end{equation}
is invertible and
\begin{align} \lb{3.54}
& (H_{F,\phi}-zI_{(a,b)})^{-1}=(H_{0,0}-zI_{(a,b)})^{-1}   \\
& \quad  - \sum_{j,k=1}^2Q_{F,\phi}(z)^{-1}_{j,k}
(u_k(\overline{z},\cdot),\cdot)_{L^2((a,b); rdx)} u_j(z,\cdot),  \quad
z\in \rho(H_{F,\phi})\cap\rho(H_{0,0}).   \no
\end{align}
$(ii)$ If $F_{1,2}=0$, then the maximal common part of $H_{F,\phi}$ and $H_{0,0}$ is the restriction of $H_{\max}$ to the domain
\begin{equation}\lb{3.55}
\dom(H_{\max})\cap\{g\in L^2((a,b);rdx)\,|\, g(a)=g(b)=0, \,
g^{[1]}(b)=e^{i\phi}F_{2,2}g^{[1]}(a)\}.
\end{equation}
In this case,
\begin{align}
q_{F,\phi}(z)&= F_{2,1}F_{2,2} + F_{2,2}^2u_2^{[1]}(z,a) +
e^{i\phi}F_{2,2}u_1^{[1]}(z,a)\no \\
&\quad - e^{-i\phi}F_{2,2}u_2^{[1]}(z,b) - u_1^{[1]}(z,b),
\quad z\in \rho(H_{F,\phi})\cap\rho(H_{0,0}),    \lb{3.56}
\end{align}
is nonzero and
\begin{align}  \lb{3.57}
& (H_{F,\phi}-zI_{(a,b)})^{-1}=(H_{0,0}-zI_{(a,b)})^{-1}    \\
& \quad - q_{F,\phi}(z)^{-1}(u_{F,\phi}
(\overline{z},\cdot),\cdot)_{L^2((a,b); rdx)} u_{F,\phi}(z,\cdot),  \quad
z\in \rho(H_{F,\phi})\cap\rho(H_{0,0}),   \no
\end{align}
where
\begin{equation}\lb{3.58}
u_{F,\phi}(z,\cdot)=e^{-i\phi}F_{2,2}u_2(z,\cdot)+u_1(z,\cdot), \quad z\in \rho(H_{F,\phi})\cap\rho(H_{0,0}).
\end{equation}
\end{theorem}
%%%%%%%%%%
\begin{proof}
We begin with the {\it proof of item $(i)$:} Let $F$ and $\phi$ satisfy the assumptions of the theorem, and suppose that $F_{1,2}\neq 0$.  By inspecting boundary conditions, one sees that
$\dom(H_{F,\phi})\cap \dom(H_{0,0})\subseteq \dom(H_{\min})$ so that
$H_{\min}$ is the maximal common part of $H_{F,\phi}$ and $H_{0,0}$,
that is, $H_{F,\phi}$ and $H_{0,0}$ are relatively prime with respect to
$H_{\min}$ (in the terminology of Appendix A, cf.\ \eqref{B.4}).

We now show that $Q_{F,\phi}(z)$ is invertible for all $z\in \rho(H_{F,\phi})\cap \rho(H_{0,0})$.  If $Q_{F,\phi}(z_0)$ is singular for some $z_0\in \rho(H_{F,\phi})\cap \rho(H_{0,0})$, then the columns of $e^{i\phi}F_{1,2}Q_{F,\phi}(z_0)$ are linearly dependent.  Therefore, there exists a constant $\alpha\in \mathbb{C}$ such that
\begin{align}
-e^{i\phi}F_{2,2}+e^{i\phi}F_{1,2}u_1^{[1]}(z_0,b)
&= \alpha\bigg(\frac{e^{i\phi}F_{1,2}}{e^{-i\phi}F_{1,2}}
+e^{i\phi}F_{1,2}u_2^{[1]}(z_0,b) \bigg),  \lb{3.59}\\
1-e^{i\phi}F_{1,2}u_1^{[1]}(z_0,a)&= \alpha\big(-e^{i\phi}F_{1,1}
-e^{i\phi}F_{1,2}u_2^{[1]}(z_0,a)\big).\lb{3.60}
\end{align}
We rewrite \eqref{3.60} as
\begin{equation}\lb{3.61}
1=-\alpha e^{i\phi}F_{1,1}+\big(u_1^{[1]}(z_0,a)-\alpha u_2^{[1]}(z_0,a) \big)
e^{i\phi}F_{1,2}.
\end{equation}
Define the function
\begin{equation}\lb{3.62}
g(z_0,\cdot)=u_1(z_0,\cdot)-\alpha u_2(z_0,\cdot),
\end{equation}
and observe that
\begin{align}
\begin{split}
g(z_0,a)&=u_1(z_0,a)-\alpha u_2(z_0,a)=-\alpha,   \lb{3.63}\\
g(z_0,b)&=u_1(z_0,b)-\alpha u_2(z_0,b)=1.
\end{split}
\end{align}
As a result of \eqref{3.61} and \eqref{3.63},
\begin{equation}\lb{3.65}
g(z_0,b)=e^{i\phi}F_{1,1}g(z_0,a)+e^{i\phi}F_{1,2}g^{[1]}(z_0,a).
\end{equation}
Moreover, using \eqref{3.59},
\begin{align}
g^{[1]}(z_0,b)&=u_1^{[1]}(z_0,b)-\alpha u_2^{[1]}(z_0,b)\no \\
&= \frac{F_{2,2}}{F_{1,2}}+\frac{\alpha}{e^{-i\phi}F_{1,2}}\no \\
&=-\alpha e^{i\phi}F_{2,1}+e^{i\phi}F_{2,2}\bigg(\frac{1}{e^{i\phi}F_{1,2}}
+ \alpha\frac{F_{1,1}}{F_{1,2}} \bigg)\lb{3.66}\\
&=e^{i\phi}F_{2,1}g(z_0,a)+e^{i\phi}F_{2,2}g^{[1]}(z_0,a).\lb{3.67}
\end{align}
To get \eqref{3.66}, we have used $\det(F)=1$; \eqref{3.67} follows from \eqref{3.61}.  Now \eqref{3.65} and \eqref{3.67} yield $g(z_0,\cdot)\in \dom(H_{F,\phi})$.  Since $\tau g(z_0,\cdot)=z_0 g(z_0,\cdot)$, the function $g(z_0,\cdot)$ is an eigenfunction of
$H_{F,\phi}$ corresponding to $z_0$, contradicting $z_0\in \rho(H_{F,\phi})$.

Now we verify \eqref{3.54}.  To this end, define
\begin{align}
g_f(z,\cdot)&:=(H_{0,0}-zI_{(a,b)})^{-1}f -\sum_{j,k=1}^2Q_{F,\phi}(z)^{-1}_{j,k}(u_k(\overline{z},\cdot),f)_{L^2((a,b); rdx)}u_j(z,\cdot),\no\\
&\hspace*{3.5cm}f\in L^2((a,b);rdx), \ z\in \rho(H_{F,\phi})\cap\rho(H_{0,0}).\lb{3.68}
\end{align}
If
\begin{equation}\lb{3.69}
g_f(z,\cdot)\in \dom(H_{F,\phi}), \quad f\in L^2((a,b);rdx), \; z\in \rho(H_{F,\phi})\cap \rho(H_{0,0}),
\end{equation}
then the representation \eqref{3.54} is valid.  In fact, if \eqref{3.69} holds, one computes
\begin{align}
(H_{F,\phi}-zI_{(a,b)})g_f(z,\cdot)&=(H_{\max}-zI_{(a,b)})g_f(z,\cdot)\nonumber\\
&=(H_{\max}-zI_{(a,b)})(H_{0,0}-zI_{(a,b)})^{-1}f=f, \lb{3.70}\\
&\; \, f\in L^2((a,b);rdx), \; z\in \rho(H_{F,\phi})\cap\rho(H_{0,0}),\nonumber
\end{align}
since $H_{\max}$ is an extension of both $H_{F,\phi}$ and $H_{0,0}$ and
\begin{equation}\lb{3.71}
(H_{\max}-zI_{(a,b)})u_1(z,\cdot)=(H_{\max}-zI_{(a,b)})u_2(z,\cdot)=0, \quad
z\in \rho(H_{F,\phi})\cap\rho(H_{0,0}).
\end{equation}
Therefore, verification of \eqref{3.54} reduces to establishing \eqref{3.69}.  In turn, \eqref{3.69} reduces to showing $g_f(z,\cdot)$ satisfies the boundary conditions in \eqref{3.4}.  To this end, using
\begin{equation}\lb{3.72}
[(H_{0,0}-zI_{(a,b)})^{-1}f](a)=[(H_{0,0}-zI_{(a,b)})^{-1}f](b)=0, \quad z\in \rho(H_{0,0}),
\end{equation}
one computes
\begin{align}
g_f(z,a)&=\det(Q_{F,\phi}(z))^{-1}\bigg\{\bigg[-\frac{e^{-i\phi}}{F_{1,2}}+ u_1^{[1]}(z,a)\bigg]\big(u_1(\overline{z},\cdot),f\big)_{L^2((a,b);rdx)}\no \\
&\quad + \bigg[-\frac{F_{2,2}}{F_{1,2}}+u_1^{[1]}(z,b) \bigg]\big(u_2(\overline{z},\cdot),f\big)_{L^2((a,b);rdx)}\bigg\},   \lb{3.73}\\
g_f(z,b)&=\det(Q_{F,\phi}(z))^{-1}\bigg\{ \bigg[-\frac{F_{1,1}}{F_{1,2}}-u_2^{[1]}(z,a)\bigg]\big(u_1(\overline{z},\cdot),f \big)_{L^2((a,b);rdx)}\no \\
&\quad + \bigg[-\frac{e^{i\phi}}{F_{1,2}}-u_2^{[1]}(z,b) \bigg]\big(u_2(\overline{z},\cdot),f \big)_{L^2((a,b);rdx)}\bigg\},\lb{3.74}\\
&\hspace*{1.25cm} f\in L^2((a,b);rdx),\ z\in \rho(H_{F,\phi})\cap \rho(H_{0,0}).\no
\end{align}

With \eqref{3.30} one computes
\begin{align}
g_f^{[1]}(z,a)&=\big(u_2(\overline{z},\cdot),f\big)_{L^2((a,b);rdx)}\no\\
&\quad + \det(Q_{F,\phi}(z))^{-1}\bigg\{\bigg[-\frac{F_{1,1}}{F_{1,2}}-u_2^{[1]}(z,a)\bigg]\big( u_1(\overline{z},\cdot),f\big)_{L^2((a,b);rdx)}u_1^{[1]}(z,a)\no \\
&\quad +\bigg[ -\frac{e^{i\phi}}{F_{1,2}}-u_2^{[1]}(z,b)\bigg]\big(u_2(\overline{z},\cdot),f \big)_{L^2((a,b);rdx)}u_1^{[1]}(z,a)\no \\
&\quad + \bigg[-\frac{1}{e^{i\phi}F_{1,2}}+u_1^{[1]}(z,a) \bigg]\big(u_1(\overline{z},\cdot),f\big)_{L^2((a,b);rdx)}u_2^{[1]}(z,a)\no \\
&\quad + \bigg[-\frac{F_{2,2}}{F_{1,2}}+u_1^{[1]}(z,b) \bigg]\big(u_2(\overline{z},\cdot),f \big)_{L^2((a,b);rdx)}u_2^{[1]}(z,a)\bigg\},   \lb{3.75}\\
g_f^{[1]}(z,b)& = -\big(u_1(\overline{z},\cdot),f\big)_{L^2((a,b);rdx)}  \no \\
& \quad +\det(Q_{F,\phi}(z))^{-1}\bigg\{ \bigg[-\frac{F_{1,1}}{F_{1,2}}-u_2^{[1]}(z,a)\bigg]\big(u_1(\overline{z},\cdot),f \big)_{L^2((a,b);rdx)}u_1^{[1]}(z,b)\no \\
&\quad +\bigg[-\frac{e^{i\phi}}{F_{1,2}}-u_2^{[1]}(z,b) \bigg]\big(u_2(\overline{z},\cdot),f \big)_{L^2((a,b);rdx)}u_1^{[1]}(z,b)\no \\
&\quad +\bigg[-\frac{1}{e^{i\phi}F_{1,2}}+u_1^{[1]}(z,a) \bigg]\big(u_1(\overline{z},\cdot),f \big)_{L^2((a,b);rdx)}u_2^{[1]}(z,b)\no \\
&\quad +\bigg[-\frac{F_{2,2}}{F_{1,2}}+u_1^{[1]}(z,b) \bigg]\big(u_2(\overline{z},\cdot),f \big)_{L^2((a,b);rdx)}u_2^{[1]}(z,b)\bigg\},\lb{3.76}\\
&\hspace*{2.55cm}  f\in L^2((a,b);rdx),\ z\in \rho(H_{F,\phi})\cap \rho(H_{0,0}).\no
\end{align}
Using \eqref{3.73}, \eqref{3.74}, and \eqref{3.75} one infers, after accounting for some immediate cancellations, that
\begin{align}
&\det(Q_{F,\phi}(z))\big(e^{i\phi}F_{1,1}g_f(z,a)+e^{i\phi}F_{1,2}g_f^{[1]}(z,a)-g_f(z,b)\big)\no \\
&\quad = \big(u_2(\overline{z},\cdot),f \big)_{L^2((a,b);rdx)}\bigg\{-\frac{e^{i\phi}F_{1,1}F_{2,2}}{F_{1,2}}+ e^{i\phi}F_{1,1}u_1^{[1]}(z,b)\no\\
&\qquad+e^{i\phi}F_{1,2}\det(Q_{F,\phi}(z))\no \\
&\qquad -e^{2i\phi}u_1^{[1]}(z,a)-e^{i\phi}F_{1,2} u_2^{[1]}(z,b)u_1^{[1]}(z,a)
-e^{i\phi}F_{2,2}u_2^{[1]}(z,a)\no \\
&\qquad+e^{i\phi}F_{1,2} u_1^{[1]}(z,b)u_2^{[1]}(z,a)+\frac{e^{i\phi}}{F_{1,2}}+u_2^{[1]}(z,b)\bigg\},\lb{3.77}\\
&\hspace*{1.45cm}  f\in L^2((a,b);rdx),\ z\in \rho(H_{F,\phi})\cap \rho(H_{0,0}).\no
\end{align}
Using the expression for $\det(Q_{F,\phi}(z))$ dictated by \eqref{3.53}, one concludes that the quantity
in the brackets on the right-hand side of \eqref{3.77} is zero. Moreover, since $\det(Q_{F,\phi}(z))\neq0$ for $z\in \rho(H_{F,\phi})\cap\rho(H_{0,0})$, it follows that the function $g_f(z,\cdot)$ satisfies the first boundary condition in \eqref{3.4} (involving $g(b)$).  A similar calculation shows that
\begin{equation}\lb{3.78}
\begin{split}
\det(Q_{F,\phi}(z))\big(e^{i\phi}F_{2,1}g_f(z,a)
+ e^{i\phi}F_{2,2}g_f^{[1]}(z,a)-g_f^{[1]}(z,b)\big)=0,\\
f\in L^2((a,b);rdx), \ z\in \rho(H_{m,\phi})\cap\rho(H_{0,0}),
\end{split}
\end{equation}
and, therefore, that $g_f(z,\cdot)$ satisfies the second boundary condition in \eqref{3.4}
(involving $g^{[1]}(b)$).  Hence, the containment \eqref{3.69} is proven.

\noindent
{\it Proof of item $(ii)$:}  If $F_{1,2}=0$, one infers that
\begin{align}
\begin{split}
&\dom(H_{F,\phi})\cap\dom(H_{0,0})=
\big\{g\in L^2((a,b);rdx) \,\big|\, g,g^{[1]}\in \AC([a,b]);     \\
 &\hspace*{1.15cm}  0=g(a)=g(b), \, g^{[1]}(b)=e^{i\phi}F_{2,2}g^{[1]}(a); \,
 \tau g\in L^2((a,b);rdx)\big\}.     \lb{3.79}
 \end{split}
\end{align}
Thus, the maximal common part of $H_{F,\phi}$ and $H_{0,0}$ is the restriction,
$\widetilde{H}_{F,\phi}$, of $H_{\max}$ with domain
$\dom\big(\widetilde{H}_{F,\phi}\big)=\dom(H_{F,\phi})\cap \dom(H_{0,0})$.  Moreover, one computes
\begin{align}
\begin{split}
&\dom\big(\big(\widetilde{H}_{F,\phi}\big)^{\ast}\big)=
\big\{g\in L^2((a,b);rdx) \,\big|\, g,g^{[1]}\in \AC([a,b]);       \\
& \hspace*{2.6cm} g(a)=e^{-i\phi}F_{2,2}g(b); \, \tau g \in L^2((a,b);rdx)\big\}.      \lb{3.80}
\end{split}
\end{align}

Next we show that $q_{F,\phi}(z)\neq 0$ if $z\in \rho(H_{F,\phi})\cap \rho(H_{0,0})$.
If $q_{F,\phi}(z_0)=0$, then $z_0$ is an eigenvalue of $H_{F,\phi}$ and
\begin{equation}\lb{3.81}
u_{F,\phi}(z_0,\cdot)=e^{-i\phi}F_{2,2}u_2(z_0,\cdot)+u_1(z_0,\cdot).
\end{equation}
 is a corresponding eigenfunction.  The latter reduces to showing $u_{F,\phi}(z_0,\cdot)$ belongs to $\dom(H_{F,\phi})$, that is, $u_{F,\phi}(z_0,\cdot)$ satisfies the boundary conditions in \eqref{3.4} (with $F_{1,2}=0$).   Observe that
\begin{align}
u_{F,\phi}(z_0,a)&=e^{-i\phi}F_{2,2}u_2(z_0,a)+u_1(z_0,a)=e^{-i\phi}F_{2,2},
\lb{3.82}\\
u_{F,\phi}(z_0,b)&=e^{-i\phi}F_{2,2}u_2(z_0,b)+u_1(z_0,b)=1,\lb{3.83}
\end{align}
and as a result,
\begin{equation}\lb{3.84}
u_{F,\phi}(z_0,b)-e^{i\phi}F_{1,1}u_{F,\phi}(z_0,a)=1-F_{1,1}F_{2,2}=1-\det(F)=0,
\end{equation}
that is, $u_{F,\phi}(z_0,\cdot)$ satisfies the first boundary condition in \eqref{3.4} (involving
$g(b)$). Moreover,
\begin{equation}\lb{3.85}
e^{i\phi}F_{2,1}u_{F,\phi}(z_0,a)+e^{i\phi}F_{2,2}u_{F,\phi}^{[1]}(z_0,a) 
- u_{F,\phi}^{[1]}(z_0,b) = q_{F,\phi}(z_0)=0
\end{equation}
implies that $u_{F,\phi}(z_0,\cdot)$ satisfies the second boundary condition in \eqref{3.4}
(involving $g^{[1]}(b)$). Thus,
\begin{equation}\lb{3.86}
u_{F,\phi}(z_0,\cdot)\in \dom(H_{F,\phi}).
\end{equation}
Since $\tau u_{F,\phi}(z_0,\cdot)=z u_{F,\phi}(z_0,\cdot)$, it follows that $z_0\in \sigma(H_{F,\phi})$.  Therefore, $q_{F,\phi}(z)\neq 0$ if $z\in \rho(H_{F,\phi})\cap \rho(H_{0,0})$.

Using the argument in \eqref{3.70}, verification of  \eqref{3.57} reduces to showing that
\begin{align} \lb{3.87}
&g_f(z,\cdot):= (H_{0,0}-zI_{(a,b)})^{-1}f-q_{F,\phi}(z)^{-1}(u_{F,\phi}(\overline{z},\cdot),f) u_{F,\phi}(z,\cdot)\in \dom(H_{F,\phi}), \no \\
&\hspace*{4.7cm} f\in L^2((a,b);rdx), \ z\in \rho(H_{F,\phi})\cap \rho(H_{0,0}),
\end{align}
which, in turn, reduces to proving that $g_f(z,\cdot)$ satisfies the boundary conditions
\begin{align}
g_f(z,b)&=e^{i\phi}F_{1,1}g_f(z,a),  \lb{3.88}\\
g_f^{[1]}(z,b)&=e^{i\phi}F_{2,1}g_f(z,a)+e^{i\phi}F_{2,2}g_f^{[1]}(z,a).\lb{3.89}
\end{align}
To show that $g_f(z,\cdot)$ satisfies the first boundary condition \eqref{3.88}, one can use
\eqref{3.8}, $F_{1,1}F_{2,2}=\det(F)=1$, and \eqref{3.72} to calculate
\begin{align}
&g_f(z,b)-e^{i\phi}F_{1,1}g_f(z,a)\no \\
&\quad = - q_{F,\phi}(z)^{-1}(u_{F,\phi}(\overline{z},\cdot),f)_{L^2((a,b);rdx)}    \no \\
&\qquad+ e^{i\phi}F_{1,1}
q_{F,\phi}(z)^{-1} (u_{R,\phi}(z,\cdot),f)_{L^2((a,b);rdx)}e^{-i\phi}F_{2,2}=0,     \lb{3.90} \\
&\hspace*{5.95cm} z\in \rho(H_{F,\phi})\cap \rho(H_{0,0}).   \no
\end{align}
In view of \eqref{3.30},
\begin{align}
g_f^{[1]}(z,a)&=(u_2(\overline{z},\cdot),f)_{L^2((a,b);rdx)}\no\\
&\quad - q_{F,\phi}(z)^{-1}(u_{F,\phi}(\overline{z},\cdot),f)u_{F,\phi}(z,a)_{L^2((a,b);rdx)},   \lb{3.91}\\
g_f^{[1]}(z,b)&=-(u_1(\overline{z},\cdot),f)_{L^2((a,b);rdx)}\no\\
&\quad - q_{F,\phi}(z)^{-1}(u_{F,\phi}(\overline{z},\cdot),f)u_{F,\phi}^{[1]}(z,b)_{L^2((a,b);rdx)}.\lb{3.92}
\end{align}
Employing \eqref{3.58}, \eqref{3.91}, and \eqref{3.92}, one computes for the difference $g_f^{[1]}(z,b)-e^{i\phi}F_{2,1}g_f(z,a)-e^{i\phi}F_{2,2}g_f^{[1]}(z,a)$ in terms of $(u_1(\overline{z},\cdot),f)_{L^2((a,b);rdx)}$ and \newline
$(u_2(\overline{z},\cdot),f)_{L^2((a,b);rdx)}$,
\begin{align}
&g_f^{[1]}(z,b)-e^{i\phi}F_{2,1}g_f(z,a)-e^{i\phi}F_{2,2}g_f^{[1]}(z,a)\no \\
&\quad =(u_1(\overline{z},\cdot),f)_{L^2((a,b);rdx)}\big\{q_{F,\phi}(z)^{-1}[F_{2,1}F_{2,2}
+ F_{2,2}^2u_2^{[1]}(z,a) + e^{i\phi}F_{2,2}u_1^{[1]}(z,a)\no \\
&\qquad - e^{-i\phi}F_{2,2}u_2^{[1]}(z,b) - u_1^{[1]}(z,b)]-1 \big\}\no \\
&\qquad+e^{i\phi}F_{2,2}(u_2(\overline{z},\cdot),f)_{L^2((a,b);rdx)}
\big\{q_{F,\phi}(z)^{-1}[F_{2,1}F_{2,2} + F_{2,2}^2u_2^{[1]}(z,a)\no \\
&\quad\quad + e^{i\phi}F_{2,2}u_1^{[1]}(z,a) - e^{-i\phi}F_{2,2}u_2^{[1]}(z,b)
- u_1^{[1]}(z,b)]-1 \big\}\no \\
&\quad =(u_1(\overline{z},\cdot),f)_{L^2((a,b);rdx)}\big[q_{F,\phi}(z)^{-1}q_{F,\phi}(z) - 1 \big]  \no \\
& \qquad
+e^{i\phi}F_{2,2}(u_2(\overline{z},\cdot),f)_{L^2((a,b);rdx)}\big[q_{F,\phi}(z)^{-1}q_{F,\phi}(z) - 1 \big]\no \\
&\quad =0, \quad f\in L^2((a,b);rdx), \; z\in \rho(H_{F,\phi})\cap \rho(H_{0,0}).
\lb{3.93}
\end{align}
\end{proof}
%%%%%%%%%%

As an example of a self-adjoint extension of $H_{\text{min}}$ with non-separated
boundary conditions, we now consider in detail the case of the Krein--von Neumann  extension. For background information on this topic we refer to \cite{AGMST10},
\cite{AGMT10} and the extensive list of references therein.

%%%%%%%%%%
\begin{example} \lb{e3.3} 
Suppose $H_{\text{min}}$, defined by \eqref{A.6}, is strictly positive in the sense that there exists an $\varepsilon>0$ for which
\begin{equation}\lb{3.94}
(f,H_{\text{min}}f)\geq \varepsilon \|f\|^2, \quad f\in \dom(H_{\text{min}}).
\end{equation}
Since the deficiency indices of $H_{\text{min}}$ are $(2,2)$, the assumption \eqref{3.94} implies that
\begin{equation}\lb{3.95}
\dim(\ker(H_{\text{min}}^{\ast}))=2.
\end{equation}
As a basis for $\ker(H_{\text{min}}^{\ast})$, we choose $\{u_1(0,\cdot),u_2(0,\cdot)\}$, where $u_1(0,\cdot)$ and $u_2(0,\cdot)$ are real-valued and satisfy \eqref{3.8} (with $z=0$).

The Krein--von Neumann extension of $H_{\text{min}}$ in $L^2((a,b); rdx)$, which we denote by $H_{\text{K}}$, is defined as the restriction of $H_{\text{min}}^{\ast}$ with domain
\begin{equation}\lb{3.96}
\dom(H_{\text{K}})=\dom(H_{\text{min}}) \dotplus \ker(H_{\text{min}}^{\ast}).
\end{equation}
Since $H_{\text{K}}$ is a self-adjoint extension of $H_{\text{min}}$, functions in $\dom(H_{\text{K}})$ must satisfy certain boundary conditions; we now provide a characterization of these boundary conditions.  Let $u\in \dom(H_{\text{K}})$; by \eqref{3.96} there exist $f\in \dom(H_{\text{min}})$ and $\eta\in \ker(H_{\text{min}}^{\ast})$ with
\begin{equation}\lb{3.97}
u(x)=f(x)+\eta(x), \quad x\in [a,b].
\end{equation}
Since $f\in \dom(H_{\text{min}})$,
\begin{equation}\lb{3.98}
f(a)=f^{[1]}(a)=f(b)=f^{[1]}(b)=0,
\end{equation}
and as a result,
\begin{equation}\lb{3.100}
u(a)=\eta(a),  \quad  u(b)=\eta(b).
\end{equation}
Since $\eta\in \ker(H_{\text{min}}^{\ast})$, we write (cf. \eqref{3.8})
\begin{equation}\lb{3.101}
\eta(x)=c_1u_1(0,x)+c_2u_2(0,x), \quad x\in [a,b],
\end{equation}
for appropriate scalars $c_1,c_2\in \bbC$.  By separately evaluating \eqref{3.101} at $x=a$ and $x=b$, one infers from \eqref{3.8} that
\begin{equation}\lb{3.103}
\eta(a)=c_2,   \quad   \eta(b)=c_1.
\end{equation}
Comparing \eqref{3.103} and \eqref{3.100} allows one to write \eqref{3.101} as
\begin{equation}\lb{3.104}
\eta(x)=u(b)u_1(0,x)+u(a)u_2(0,x), \quad x\in [a,b].
\end{equation}
Finally, \eqref{3.97} and \eqref{3.104} imply
\begin{equation}\lb{3.105}
u(x)=f(x)+u(b)u_1(0,x)+u(a)u_2(0,x), \quad x\in [a,b],
\end{equation}
and as a result,
\begin{equation}\lb{3.106}
u^{[1]}(x)=f^{[1]}(x)+u(b)u_1^{[1]}(0,x)+u(a)u_2^{[1]}(0,x), \quad x\in [a,b].
\end{equation}
Evaluating \eqref{3.106} separately at $x=a$ and $x=b$, and using \eqref{3.98}
yields the following boundary conditions for $u$:
\begin{equation}
u^{[1]}(a)=u(b)u_1^{[1]}(0,a)+u(a)u_2^{[1]}(0,a),   \quad
u^{[1]}(b)=u(b)u_1^{[1]}(0,b)+u(a)u_2^{[1]}(0,b).       \lb{3.108}
\end{equation}
Since $u_1^{[1]}(0,a)\neq0$ (one recalls that $u_1(0,a)=0$), relations \eqref{3.108} can be
recast as
\begin{equation}\lb{3.109}
\begin{pmatrix} u(b) \\ u^{[1]}(b)\end{pmatrix} 
= F_{\text{K}}\begin{pmatrix} u(a) \\ u^{[1]}(a)\end{pmatrix},
\end{equation}
where
\begin{equation}\lb{3.110}
F_{\text{K}}=\frac{1}{u_1^{[1]}(0,a)} \begin{pmatrix}
-u_2^{[1]}(0,a) & 1\\
u_1^{[1]}(0,a)u_2^{[1]}(0,b)-u_1^{[1]}(0,b)u_2^{[1]}(0,a) & u_1^{[1]}(0,b)
\end{pmatrix}.
\end{equation}
Then $F_{\text{K}}\in \SL_2(\bbR)$ since \eqref{3.110} and \eqref{3.33} imply
\begin{equation}\lb{3.111}
\det(F_{\text{K}})=-\frac{u_2^{[1]}(0,b)}{u_1^{[1]}(0,a)}=1.
\end{equation}
Thus, we have shown $H_{\text{K}}\subset H_{F_{\text{K}},0}$, where
$H_{F_{\text{K}},0}$ is defined by \eqref{3.4} with $F=F_{\text{K}}$ and $\phi=0$.
Since both $H_{\text{K}}$ and  $H_{F_{\text{K}},0}$ are self-adjoint, we conclude
$H_{\text{K}}=H_{F_{\text{K}},0}$; that is, $H_{F_{\text{K}},0}$ is the
Krein--von Neumann extension of $H_{\text{min}}$.

Applying the result of Theorem \ref{t3.2}, one has
\begin{align} \lb{3.112}
& (H_{{\text{K}}}-zI_{(a,b)})^{-1}=(H_{0,0}-zI_{(a,b)})^{-1}    \\
& \quad  - \sum_{j,k=1}^2Q_{F_{\text{K}},0}(z)^{-1}_{j,k}
(u_k(\overline{z},\cdot),\cdot)_{L^2((a,b); rdx)} u_j(z,\cdot),  \quad
z\in \rho(H_{\text{K}})\cap\rho(H_{0,0}),    \no
\end{align}
where
\begin{align}\lb{3.113}
\begin{split} 
Q_{F_{\text{K}},0}(z)=\begin{pmatrix}
u_1^{[1]}(0,b) - u_1^{[1]}(z,b) & - u_1^{[1]}(0,a) - u_2^{[1]}(z,b)\\
- u_1^{[1]}(0,a) + u_1^{[1]}(z,a) & - u_2^{[1]}(0,a) + u_2^{[1]}(z,a)
\end{pmatrix},&  \\ 
z\in \rho(H_{\text{K}})\cap\rho(H_{0,0}).&
\end{split}
\end{align}

In the special case $q\equiv 0$, and using the corresponding notation $H_{\text{K}}^{(0)}$ (and 
similarly, $\tau^{(0)}$, $u_1^{(0)}(0,\cdot)$, $u_2^{(0)}(0,\cdot)$), the above analysis is 
particularly transparent. In this case, a basis for
$\ker(H_{\text{min}}^{\ast})$ is provided by $\{u_1^{(0)}(0,\cdot),u_2^{(0)}(0,\cdot)\}$, where
\begin{align}
\begin{split} 
u_1^{(0)}(0,x) & = \bigg[\int_a^b ds \, p(s)^{-1}\bigg]^{-1} \int_a^x dt \, p(t)^{-1},   \\ 
u_2^{(0)}(0,x) & = 1 - \bigg[\int_a^b ds \, p(s)^{-1}\bigg]^{-1} \int_a^x dt \, p(t)^{-1}, 
\quad x\in [a,b],  \lb{3.115}
\end{split} 
\end{align}
as one verifies that 
\begin{equation}\lb{3.116}
H_{\text{min}}^{\ast}u_j^{(0)}(0,\cdot)=H_{\text{max}}u_j^{(0)}(0,\cdot)
=\tau^{(0)} u_j^{(0)}(0,\cdot)=0, \quad j=1,2.
\end{equation}
The boundary conditions for $H_{\text{K}}^{(0)}$ then read
\begin{equation}\lb{3.117}
\begin{pmatrix} u(b) \\ u^{[1]}(b)\end{pmatrix}
= F_{\text{K}}^{(0)}\begin{pmatrix} u(a) \\ u^{[1]}(a)\end{pmatrix}, \quad 
u\in \dom\big(H_{\text{K}}^{(0)}\big),
\end{equation}
where
\begin{equation}\lb{3.118}
F_{\text{K}}^{(0)}= \begin{pmatrix} 1 & \int_a^b dt \, p(t)^{-1} \\ 0 & 1 \end{pmatrix}.
\end{equation}
Explicitly,
\begin{equation}
u^{[1]}(b) = u^{[1]}(a) = \bigg[\int_a^b dt \, p(t)^{-1}\bigg]^{-1} [u(b) - u(a)],  
\quad u\in \dom\big(H_{\text{K}}^{(0)}\big).    \lb{3.119}
\end{equation}
\end{example}
%%%%%%%%%%%%

We note that \eqref{3.119} has been derived in \cite{AS80} and \cite[Sect.\ 2.3]{Fu80}
(see also \cite[Sect.\ 3.3]{FOT94}) in the special case where $p = r \equiv 1$. While it appears 
that our characterization \eqref{3.109}, \eqref{3.110} of the Krein--von Neumann boundary condition 
for general Sturm--Liouville operators on a finite interval is new, the special case $q \equiv 0$ was recently 
discussed in \cite{FN09}.

%%%%%%%%%%%%%%%%%%%%%%%%%%%%%%%%%%%%%%%%
%%%%%%%%%%%%%%%%%%%%%%%%%%%%%%%%%%%%%%%%
\section{General Boundary Data Maps and Their Basic Properties} \lb{s4}
%%%%%%%%%%%%%%%%%%%%%%%%%%%%%%%%%%%%%%%%
%%%%%%%%%%%%%%%%%%%%%%%%%%%%%%%%%%%%%%%%

This section is devoted to general boundary data maps  associated with
self-adjoint extensions of the operator $H_{\min}$ defined in
\eqref{A.6}. A special case of the boundary data maps corresponding to
separated boundary conditions was recently introduced in \cite{CGM10} and
further discussed in \cite{GZ11}. At the end of this section we show how the
general boundary data map appears naturally in Krein's resolvent formula for
a difference of resolvents of any two self-adjoint extensions of $H_{\min}$.

We recall the general boundary trace map, $\ga_{A,B}$, introduced in \eqref{A.26} associated with the boundary $\{a,b\}$ of $(a,b)$ and the $2\times2$ (parameter) matrices $A,B$ satisfying \eqref{A.7}, the special cases of the Dirichlet trace $\gamma_D$, and the Neumann trace $\gamma_N$ (in connection 
with the outward pointing unit normal vector at $\partial (a,b) = \{a, b\}$) defined in \eqref{A.27}, the 
matrices $A_D$, $B_D$, $A_N$, and $B_N$ in \eqref{A.28a} and \eqref{A.28b}, and the relations in 
\eqref{A.29}--\eqref{A.31}. 

It follows from Theorems \ref{tA.2} and \ref{tA.4} that
\begin{equation}
H_\AB \, f = \tau f,  \quad f \in \dom(H_\AB)=\{g\in\dom(H_\max) \, | \, \ga_\AB\,g = 0\},
 \lb{4.8b}
\end{equation}
defines a self-adjoint extension of $H_{\min}$ whenever $A,B\in\bbC^{2\times2}$ satisfy \eqref{A.7}. In particular, we note that
\begin{align}
\ga_\AB (H_\AB - z I_{(a,b)})^{-1} = 0, \quad z\in\rho(H_\AB). \lb{4.8c}
\end{align}
\begin{remark}
Given \eqref{4.8b}, the Dirichlet extension of $H_{\min}$ will be defined by
$H_{A_D,B_D}$ and denoted by $H_D$ while the Neumann extension of $H_{\min}$
will be defined by $H_{A_N,B_N}$ and denoted by $H_N$. Note that $H_D$ and
$H_N$  are associated with $\ga_D$ and $\ga_N$, respectively and, relative to
the notation used in Theorem \ref{t3.1}, that $H_{0,0}=H_D$ while
$H_{\pi/2,\pi/2}=H_N$.
\end{remark}
Given the general boundary trace map $\ga_\AB$, we now introduce a complimentary trace
map $\ga_\AB^\perp$ by
\begin{align}
\ga_\AB^\perp = D^\perp_{A,B}\ga_D + N^\perp_{A,B}\ga_N, \lb{4.8d}
\end{align}
where the $2\times2$ matrices $D^\perp_{A,B}$, $N^\perp_{A,B}$ are given by
\begin{align}
\begin{split}
D^\perp_{A,B} &= -[D_{A,B}D_{A,B}^*+N_{A,B}N_{A,B}^*]^{-1} N_{A,B}, \\
N^\perp_{A,B} &= [D_{A,B}D_{A,B}^*+N_{A,B}N_{A,B}^*]^{-1} D_{A,B}, \lb{4.8e}
\end{split}
\end{align}
and where the $2\times 2$ self-adjoint matrix
\begin{equation}
D_{A,B}D_{A,B}^*+N_{A,B}N_{A,B}^* = (D_{A,B}\ \ N_{A,B})(D_{A,B}\ \ N_{A,B})^*
\end{equation}
in \eqref{4.8e} is invertible given the rank condition in \eqref{A.31}. It
follows from \eqref{A.31} and \eqref{4.8e} that
\begin{align}
\rank(D^\perp_{A,B}\ \ N^\perp_{A,B})=2, \quad
D^\perp_{A,B} (N^\perp_{A,B})^*=N^\perp_{A,B} (D^\perp_{A,B})^*, \lb{4.8f}
\end{align}
and hence that $\ga_\AB^\perp$ is also a boundary trace map for which
$\ga_\AB^\perp=\ga_{A^\perp_{A,B},B^\perp_{A,B}}$, where
\begin{align}
A^\perp_{A,B}=
\begin{pmatrix}
D^\perp_{A,B,1,1}&N^\perp_{A,B,1,1}\\D^\perp_{A,B,2,1}&N^\perp_{A,B,2,1}
\end{pmatrix},\quad
B^\perp_{A,B}=
\begin{pmatrix}-D^\perp_{A,B,1,2}&N^\perp_{A,B,1,2}\\
-D^\perp_{A,B,2,2}&N^\perp_{A,B,2,2}
\end{pmatrix}. \lb{4.8g}
\end{align}
In particular, one notes that
\begin{align}
\ga_D^\perp = \ga_N, \quad \ga_N^\perp = -\ga_D, \quad
(\ga_\AB^\perp)^\perp = -\ga_\AB. \lb{4.8h}
\end{align}

One reason for introducing the complimentary boundary  trace map
$\ga_\AB^\perp$ is to obtain a convenient way of connecting two arbitrary
boundary trace maps, $\ga_\ABp$ and $\ga_\AB$. This is done by generalizing
the identity in \eqref{A.29} to obtain
\begin{align}
\ga_\ABp = T_{A',B',A,B} \, \ga_\AB + S_{A',B',A,B} \, \ga_\AB^\perp, \lb{4.8i}
\end{align}
where
\begin{align}
\begin{split}
T_{A',B',A,B} &= D_{A',B'}(N^\perp_{A,B})^* - N_{A',B'}(D^\perp_{A,B})^*,    \\
S_{A',B',A,B} &= N_{A',B'}D^*_{A,B} - D_{A',B'}N^*_{A,B}. \lb{4.8j}
\end{split}
\end{align}
The above formulas \eqref{4.8i} and \eqref{4.8j} easily follow from \eqref{A.29}, \eqref{A.31}, \eqref{4.8d}, and \eqref{4.8e}. Moreover, we note that it follows from \eqref{A.31}, \eqref{4.8e}, \eqref{4.8f}, and \eqref{4.8j} that
\begin{align}
\rank(T_{A',B',A,B}\ \ S_{A',B',A,B})=2, \quad
T_{A',B',A,B}S_{A',B',A,B}^*=S_{A',B',A,B}T^*_{A',B',A,B}. \lb{4.8k}
\end{align}
Conversely, every pair of $2\times2$ matrices  $T_{A',B',A,B}, S_{A',B',A,B}$
satisfying \eqref{4.8k} defines a general boundary trace map $\ga_\ABp$ via
\eqref{4.8i} with $D_{A',B'},N_{A',B'}$ satisfying \eqref{A.31}.

For future use we record the following  two special cases of \eqref{4.8i},
\begin{align}
\ga_D = (N^\perp_{A,B})^*\ga_\AB - N^*_{A,B} \ga_\AB^\perp, \quad
\ga_N = -(D^\perp_{A,B})^*\ga_\AB + D^*_{A,B}\ga_\AB^\perp. \lb{4.8l}
\end{align}
Moreover, interchanging the role of $\AB$ and $\ABp$ in \eqref{4.8i}, yields
\begin{align}
\ga_\AB = T_{A,B,A',B'}\ga_\ABp + S_{A,B,A',B'}\ga_\ABp^\perp, \lb{4.8m}
\end{align}
with
\begin{align}
\begin{split}
T_{A,B,A',B'} &= D_{A,B}(N_{A',B'}^\perp)^*
- N_{A,B}(D_{A',B'}^\perp)^*,    \\
S_{A,B,A',B'} &= N_{A,B} D_{A',B'}^* - D_{A,B} N_{A',B'}^*. \lb{4.8n}
\end{split}
\end{align}
Comparing \eqref{4.8j} with \eqref{4.8n} one observes that
$S_{A,B,A',B'}=-S^*_{A',B',A,B}$ and hence
\begin{align}
\ga_\AB = T_{A,B,A',B'}\ga_\ABp - S^*_{A',B',A,B}\ga_\ABp^\perp. \lb{4.8o}
\end{align}

Finally, we note that the conditions of the type \eqref{A.31} and \eqref{4.8k} imply the following useful property for the matrices involved: If a pair
$T_{A',B',A,B}, S_{A',B',A,B}$ satisfies \eqref{4.8k} then the pair
$T_{A',B',A,B},S_{A',B',A,B,\de}$ with
\begin{align}
S_{A',B',A,B,\de}=S_{A',B',A,B}+\de \, T_{A',B',A,B}, \quad \de\in\bbR, \lb{4.8p}
\end{align}
also satisfies \eqref{4.8k} and, in addition, the matrix $S_{A',B',A,B,\de}$ is necessarily invertible for all sufficiently small $\de\neq0$. The conditions in \eqref{A.31} yield a similar result for the pair of matrices $D_{A,B}, N_{A,B}$.

Next, we recall the following elementary, yet fundamental, fact.

%%%%%%%%%%
\begin{lemma} \lb{l4.2}
Assume Hypothesis \ref{hA.1}, choose matrices $A, B\in\bbC^{2\times2}$ such
that $\rank(A\ \ B) = 2$, and suppose that $z\in\rho(H_{A,B})$. Then the boundary value problem
\begin{align}
-&\big(u^{[1]}\big)' + qu=zru,\quad u, u^{[1]}\in \AC([a,b]), \lb{4.9}
\\
&\ga_\AB u=\begin{pmatrix}c_1\\ c_2 \end{pmatrix}\in\bbC^2, \lb{4.10}
\end{align}
has a unique solution $u(z,\cdot)=u_{A,B}(z,\,\cdot\ ;c_1,c_2)$ for each
$c_1, c_2\in\bbC$. In addition, for each $x \in [a,b]$ and $c_1, c_2\in\bbC$,
$u_{A,B}(\cdot, x; c_1,c_2)$ is analytic on $\rho(H_{A,B})$.
\end{lemma}
%%%%%%%%%%%
\begin{proof}
This is well-known, but for the sake of completeness, we briefly recall the argument.
Let $u_j(z,\cdot)$, $j=1,2$, be a basis for the solutions of \eqref{4.9} and let
\begin{align}
u(z,\cdot)= d_1 u_1(z,\cdot) + d_2 u_2(z,\cdot), \quad d_1, d_2\in\bbC, \lb{4.11}
\end{align}
be the general solution of \eqref{4.9}. Then
\begin{align}
\ga_\AB (u(z,\cdot)) &=
M \begin{pmatrix} d_1 \\ d_2 \end{pmatrix}, \lb{4.12}
\end{align}
where $M\in\bbC^{2\times2}$ and the entries are given by
\begin{align}
\begin{pmatrix}
M_{1,1}\\M_{2,1}
\end{pmatrix}
= \ga_\AB (u_1(z,\cdot)), \quad
\begin{pmatrix}
M_{1,2}\\M_{2,2}
\end{pmatrix}
= \ga_\AB (u_2(z,\cdot)). \lb{4.13}
\end{align}
Thus, by \eqref{4.11}, \eqref{4.12}, the boundary value problem \eqref{4.9}, \eqref{4.10} is equivalent to
\begin{align}
M\begin{pmatrix} d_1 \\ d_2 \end{pmatrix} =
\begin{pmatrix} c_1 \\ c_2 \end{pmatrix}. \lb{4.14}
\end{align}
Given $(c_1 \ c_2)^\top\in\bbC^2$, \eqref{4.14} has a unique solution $(d_1\ d_2)^\top\in\bbC^2$ if and only if $\det(M)\neq0$. Thus, it suffices to show that $\det(M)\neq0$. Assume to the contrary that $\det(M)=0$. Then there is a nonzero vector
$(d_1 \ d_2)^\top\in\bbC^2$ such that
\begin{align}
M\begin{pmatrix} d_1 \\ d_2 \end{pmatrix} =
\begin{pmatrix} 0 \\ 0 \end{pmatrix}, \lb{4.15}
\end{align}
which, by \eqref{4.11}, is equivalent to the existence of a nontrivial
solution $u(z,\cdot)$ of the boundary value problem \eqref{4.9}, \eqref{4.10}
with homogeneous boundary conditions (i.e., with $c_0=c_1=0$). Equivalently,
$u(z,\cdot)$ satisfies
\begin{equation}
H_{A,B} u(z,\cdot) = z u(z,\cdot), \quad u(z,\cdot) \in \dom(H_{A,B}),
\end{equation}
which in turn is equivalent to $z\in\sigma(H_{A,B})$, a contradiction.

The $z$-independence of the initial condition \eqref{4.10} yields that for fixed
$x \in [a,b]$, $c_1, c_2 \in \bbC$,
$u_{A,B}(\cdot, x; c_1,c_2)$ is analytic on $\rho(H_{A,B})$.
\end{proof}
%%%%%%%%%%

Given $\AB,\ABp\in\bbC^{2\times2}$ with $\AB$ and $\ABp$ satisfying
\eqref{A.7} and assuming $z\in\rho(H_{A,B})$, we introduce in association
with the boundary value problem \eqref{4.9}, \eqref{4.10},  the 
\emph{general boundary data map},
$\La_\AB^\ABp(z) : \bbC^2 \rightarrow \bbC^2$, by
\begin{align}
\begin{split}
\La_\AB^\ABp(z) \begin{pmatrix}c_1\\ c_2\end{pmatrix} &=
\La_\AB^\ABp(z) \, \ga_{A,B} (u_\AB(z,\,\cdot \, ;c_1,c_2)) \\
&= \ga_\ABp (u_\AB(z,\,\cdot \, ;c_1,c_2)),
\end{split} \lb{4.17}
\end{align}
where $u_\AB(z, \, \cdot \, ;c_1,c_2)$ is the solution of the boundary value problem
in \eqref{4.9}, \eqref{4.10}.

As defined, $\La_\AB^\ABp(z)$ is a linear transformation and thus
representable as an element of $\bbC^{2\times 2}$. A basis-independent
description for the boundary data map defined in \eqref{4.17} is
provided in the next result.
%%%%%%%%%%
\begin{theorem}\label{t4.2.1}
Assume that $\AB,\ABp\in\bbC^{2\times2}$, where $\AB$ and $\ABp$ satisfy
\eqref{A.7}, and let  $z\in\rho(H_{A,B})$. In addition, denote by $y_j(z,\cdot)$,
$j=1,2$, a basis for the solutions of \eqref{4.9}. Then,
\begin{equation}\label{4.17a}
\La_\AB^\ABp(z) = \begin{pmatrix} \ga_\ABp(y_1(z,\cdot)) & \ga_\ABp(y_2(z,\cdot))
\end{pmatrix}
\begin{pmatrix} \ga_\AB(y_1(z,\cdot)) & \ga_\AB(y_2(z,\cdot))\end{pmatrix}^{-1}.
\end{equation}
Moreover, $\La_\AB^\ABp(z)$ is invariant with respect to change of basis for
the solutions of \eqref{4.9}.
\end{theorem}
%%%%%%%%%%
\begin{proof}
Letting $y(z,\cdot) = d_1 y_1(z,\cdot) + d_2 y_2(z,\cdot)$, $d_1, d_2 \in\bbC$,
be an arbitrary solution of \eqref{4.9}, one observes that
\begin{align}
\La_\AB^\ABp(z)\ga_\AB(y(z,\cdot))&=\La_\AB^\ABp(z)
\begin{pmatrix}\ga_\AB(y_1(z,\cdot)) & \ga_\AB(y_2(z,\cdot))\end{pmatrix}
\begin{pmatrix} d_1 \\ d_2 \end{pmatrix}       \no \\
&= \begin{pmatrix}\ga_\ABp(y_1(z,\cdot)) & \ga_\ABp(y_2(z,\cdot))\end{pmatrix}
\begin{pmatrix} d_1 \\ d_2 \end{pmatrix}
\end{align}
for every $\begin{pmatrix} d_1 & d_2 \end{pmatrix}^\top \in\bbC^2$. Equation
\eqref{4.17a} then follows by the invertibility of
$\begin{pmatrix}\ga_\AB(y_1(z,\cdot)) & \ga_\AB(y_2(z,\cdot))\end{pmatrix}$
noted in Lemma \ref{l4.2}.

Let $\widehat{y}_j(z,\cdot)$, $j=1,2$, denote a second basis for the solutions of
\eqref{4.9}. Then, there is a nonsingular matrix $K\in\bbC^{2\times 2}$ such
that $\begin{pmatrix} y_1 & y_2\end{pmatrix}=\begin{pmatrix} \widehat{y}_1 &
\widehat{y}_2 \end{pmatrix}K$. Next, by \eqref{A.27} and \eqref{A.29}, one notes that
\begin{align}\label{4.17b}
&\begin{pmatrix}\ga_\AB(y_1)&\ga_\AB(y_2)\end{pmatrix} =
D_\AB\begin{pmatrix}\ga_D(y_1)&\ga_D(y_2)\end{pmatrix} +
N_\AB\begin{pmatrix}\ga_N(y_1)&\ga_N(y_2)\end{pmatrix}     \no \\
&\hspace{20pt}  =\left[D_\AB\begin{pmatrix} \widehat{y}_1(z,a)& \widehat{y}_2(z,a)\\
\widehat{y}_1(z,b)& \widehat{y}_2(z,b) \end{pmatrix}
+ N_\AB \left(\negthickspace\begin{array}{rr} \widehat{y}_1^{[1]}(z,a)& \widehat{y}_2^{[1]}(z,a)\\
-\widehat{y}_1^{[1]}(z,b) & - \widehat{y}_2^{[1]}(z,b) \end{array}\negthickspace\right)\right]K     \no \\
&\hspace{20pt} = \begin{pmatrix}\ga_\AB(\widehat{y}_1)&\ga_\AB(\widehat{y}_2)\end{pmatrix}K.
\end{align}
The invariance of $\La_\AB^\ABp(z)$ with respect to change of basis for the
solutions of \eqref{4.9} now follows from \eqref{4.17a} and \eqref{4.17b}.
\end{proof}
%%%%%%%%%%%

%%%%%%%%%%%%%%%%%%%%%%%%%%%%%%%%%%%%%%%
\begin{remark}
In what follows, we let
\begin{equation}
\Lambda_{A,B}^D = \Lambda_{A,B}^{A_D,B_D},\qquad \Lambda_{A,B}^N = \Lambda_{A,B}^{A_N,B_N}, \qquad \Lambda_D^N = \Lambda_{A_D,B_D}^{A_N,B_N}
\end{equation}
where $A_D$, $B_D$ were defined in \eqref{A.28a}, and $A_N$, $B_N$ in \eqref{A.28b}. Similarly, we define $\Lambda_D^{A,B}$,
$\Lambda_N^{A,B}$, and $\Lambda_N^D$. In other words, $\La_D^\AB$ (resp.,
$\La_N^\AB$) will denote the boundary data maps that map the Dirichlet
(resp., Neumann) boundary data into a general $(A,B)$-boundary data set.
In particular, the Dirichlet-to-Neumann boundary data map,
$\La_D^N$, and the Neumann-to-Dirichlet boundary data map, $\La_N^D$, are
special cases of such maps.
\end{remark}
%%%%%%%%%%%%%%%%%%%%%%%%%%%%%%%%%%%%%%%%%%

The following result collects fundamental algebraic properties for general boundary data maps.

%%%%%%%%%%%
\begin{lemma}\label{c4.3}
Assume that $\AB,\ABp\in\bbC^{2\times2}$, where $\AB$ and $\ABp$ satisfy
\eqref{A.7}. Then,
\begin{align}
& \La_\AB^\ABp(z) = D_{A',B'}\La_\AB^D(z)
+ N_{A',B'}\La_\AB^N(z), \quad z\in\rho(H_\AB), \lb{4.18}
\\
& \La_\AB^\AB(z) = I_2, \quad z\in\rho(H_\AB), \lb{4.18a}
\\
& \La_\ABp^{A^{\prime\prime}\!,B^{\prime\prime}}(z) \, \La_\AB^\ABp(z) = \La_\AB^{A^{\prime\prime}\!,B^{\prime\prime}}(z), \quad z\in\rho(H_\AB)\cap\rho(H_\ABp), \lb{4.19}
\\
& \La_\AB^\ABp(z) = \Big[\La_\ABp^\AB(z)\Big]^{-1},
\quad z\in\rho(H_\AB)\cap\rho(H_\ABp). \lb{4.20}\\
& \La_\AB^\ABp(z) = \big[D_{A',B'} + N_{A',B'} \Lambda_D^N (z)\big]
\big[D_{A,B} + N_{A,B} \Lambda_D^N (z)\big]^{-1},     \lb{4.104a} \\
& \hspace{6cm} z \in \rho(H_\AB) \cap \rho(H_D).    \no
\end{align}
In particular, $\La_\AB^\ABp(z)$ is invertible for $z\in\rho(H_\AB)\cap\rho(H_\ABp)$ and
for every fixed $z\in\rho(H_\AB)\cap\rho(H_\ABp)$, $\La_\AB^\ABp(z)$ depends continuously on $\ABp$ and $\AB$. In addition, for fixed $\ABp$ and $\AB$,
$\La_\AB^\ABp(\cdot)$ is analytic on $\rho(H_\AB)$.
\end{lemma}
%%%%%%%%%%%
\begin{proof}
Given the definition of the general boundary data map in \eqref{4.17} and the
description of the boundary trace map given in \eqref{A.29},
equation \eqref{4.18} follows from the observation that
\begin{align}
\begin{split}
\La_\AB^\ABp(z)\gamma_\AB(u(z,\cdot))&= D_\ABp\gamma_D(u(z,\cdot)) + N_\ABp\gamma_N(u(z,\cdot))\\
&=(D_\ABp\La_\AB^D(z) + N_\ABp\La_\AB^N(z))\gamma_\AB(u(z,\cdot)).
\end{split}
\end{align}
The group properties for $\La_\AB^\ABp(z)$ given in equations
\eqref{4.18a}--\eqref{4.20} follow from Theorem \ref{t4.2.1}, (cf.~\eqref{4.17a}). The linear fractional transformation given in \eqref{4.104a}
follows immediately from \eqref{4.18}--\eqref{4.20} and $\La_\AB^\ABp =
\La_{A_D,B_D}^\ABp\La_\AB^{A_D,B_D}$.

By \eqref{A.27}--\eqref{A.30} the boundary trace map $\ga_\ABp$ depends
continuously on the parameter matrices $\ABp$, thus it follows from
\eqref{4.17} that the boundary data map $\La_\AB^\ABp(z)$ depends
continuously on $\ABp$ as well. By \eqref{4.20} $\La_\AB^\ABp(z)$ also
depends continuously on the parameter matrices $\AB$ for every fixed
$z\in\rho(H_\AB)\cap\rho(H_\ABp)$. Finally, analyticity of
$\La_\AB^\ABp(\cdot)$ on $\rho(H_\AB)$ is clear from Lemma \ref{l4.2} and
\eqref{4.17}.
\end{proof}
%%%%%%%%%%%

The linear fractional transformation given in \eqref{4.104a} also implies the
existence of a linear fractional transformation between general boundary data
maps $\La_\AB^\ABp(\cdot)$ and $\Lambda_{A'',B''}^{A''',B'''}(\cdot)$. (Of
course, existence of such linear fractional transformations, even in the
context of infinite deficiency indices, is clear from the general approach to
Krein-type resolvent formulas in \cite{GMT98}.)

More precisely, let $R= \big[R_{j,k}\big]_{1\leq j,k\leq 2}\in \bbC^{4 \times
4}$, with $R_{j,k}\in \bbC^{2 \times 2}$, $1\leq j,k\leq 2$, and $L\in
\bbC^{2 \times 2}$, chosen such that $\ker (R_{1,1} + R_{1,2} L)=\{0\}$; that
is, $(R_{1,1} + R_{1,2} L)$ is invertible in $\bbC^2$. Define for such $R$
(cf., e.g., \cite{KS74}),
\begin{equation}
M_R(L)=(R_{2,1} + R_{2,2} L)(R_{1,1} + R_{1,2} L)^{-1},   \lb{4.109}
\end{equation}
and observe that
\begin{align}
& M_{I_{4}} (L) = L,   \lb{4.110} \\
& M_{R S}(L) = M_R(M_S(L)), \lb{4.111} \\
& M_{R^{-1}} (M_R(L)) = L = M_R(M_{R^{-1}} (L)), \quad R \, \text{ invertible,}
\lb{4.112} \\
& M_R(L)=M_{R S^{-1}}(M_S(L)),  \quad S \in \bbC^{4 \times 4} \, \text{ invertible,} \lb{4.113}
\end{align}
whenever the right-hand sides (and hence the left-hand sides) in
\eqref{4.110}--\eqref{4.113} exist. Thus, with the choices
\begin{align}
\begin{split}
R(A,B,A',B') &= \begin{pmatrix} D_{A,B} & N_{A,B} \\ D_{A',B'} & N_{A',B'} \end{pmatrix},
\\
S(A'',B'',A''',B''') &= \begin{pmatrix} D_{A'',B''} & N_{A'',B''} \\ D_{A''',B'''} & N_{A''',B'''}
\end{pmatrix},
\end{split}
\end{align}
one infers that
\begin{equation}
\La_\AB^\ABp(z) = M_{R(A,B,A',B') S(A'',B'',A''',B''')^{-1}}
\Big(\Lambda_{A'',B''}^{A''',B'''}(z)\Big).    \lb{4.115}
\end{equation}
Unfortunately, the computation of $S(A'',B'',A''',B''')^{-1}$ appears to be
too elaborate to pursue explicit formulas for  \eqref{4.115}. (The special
case of separated boundary conditions, however, is sufficiently simple, and
in this case $S(A'',B'',A''',B''')^{-1}$ was explicitly computed in
\cite{CGM10}).

%%%%%%%%%%%%%%%%%%%%%%%%

We now turn our attention to a derivation of a representation for the general
boundary data map $\La_\AB^\ABp(z)$ in terms of the resolvent $(H_\AB-z
I_{(a,b)})^{-1}$ and the boundary trace map $\ga_\ABp$ (cf.\ Theorem \ref{t4.4a}).

Assuming $z\in\rho(H_D)$, let $u_j(z,\cdot)$, $j=1,2$, denote the solutions of \eqref{4.9} satisfying \eqref{3.8}.
Then the system $\{u_1(z,\cdot),u_2(z,\cdot)\}$ is a basis for solutions of \eqref{4.9}. The solution $u_D(z,\,\cdot \, ;c_2,c_1)$ of \eqref{4.9} with the boundary data $\ga_D(u_D(z,\, \cdot \, ;c_2,c_1))=\begin{pmatrix}c_2 & c_1\end{pmatrix}^\top$ is given by
\begin{align}
u_D(z,\,\cdot \, ;c_2,c_1)=c_1u_1(z,\cdot)+c_2u_2(z,\cdot).
\end{align}
Using the basis $\{u_1(z,\cdot),u_2(z,\cdot)\}$ one can represent the boundary data maps, $\La_D^\AB$, as $2\times2$ complex matrices. First, the special case of the Dirichlet-to-Neumann boundary data map is given by
\begin{align}
\begin{split}
\La_D^N(z) \begin{pmatrix}c_2\\c_1\end{pmatrix} &= \ga_N(c_1u_1(z,\cdot)+c_2u_2(z,\cdot))
\\
&= \begin{pmatrix}
u_2^{[1]}(z,a) & u_1^{[1]}(z,a) \\ -u_2^{[1]}(z,b) & -u_1^{[1]}(z,b)
\end{pmatrix}
\begin{pmatrix}c_2\\c_1\end{pmatrix}, \quad z\in\rho(H_D).
\end{split} \lb{4.22}
\end{align}
Then, by \eqref{A.29}, the boundary data map $\La_D^\AB(z)$ is given by
\begin{align}
\La_D^\AB(z) \begin{pmatrix}c_2\\c_1\end{pmatrix} &= \ga_\AB(c_1u_1(z,\cdot)+c_2u_2(z,\cdot))  \no
\\
&= D_{A,B}\ga_D(c_1u_1(z,\cdot)+c_2u_2(z,\cdot)) + N_{A,B}\ga_N(c_1u_1(z,\cdot)+c_2u_2(z,\cdot))
\no \\
&= \Big(D_{A,B}+N_{A,B}\La_D^N(z)\Big) \begin{pmatrix}c_2\\c_1\end{pmatrix},
\quad z\in\rho(H_D). \lb{4.23}
\end{align}

One can also represent $\La_D^\AB(z)$ in terms of the resolvent $(H_D-z I_{(a,b)})^{-1}$ and the boundary trace $\ga_\AB$. One recalls that
\begin{align}
\begin{split}
\big((H_D-z I_{(a,b)})^{-1}g\big)(x) = \int_a^b \, dx'\, G_D(z,x,x')g(x'),& \\
g\in L^2((a,b);rdx), \; z\in\rho(H_D), \; x\in(a,b),&
\end{split} \lb{4.24}
\end{align}
where the Green's function $G_D(z,x,x')$ is given by (cf.\ \eqref{3.8})
\begin{align}
\begin{split}
G_D(z,x,x') & =
\frac{1}{W(u_2(z,\cdot),u_1(z,\cdot))}
\begin{cases}
u_2(z,x)u_1(z,x'), & 0\le x'\le x,\\
u_1(z,x)u_2(z,x'), & 0\le x\le x',
\end{cases}
\\
&\hspace{4.6cm}  z\in\rho(H_D), \; x,x' \in (a,b).
\end{split} \lb{4.25}
\end{align}
Here $W(u_2(z,\cdot),u_1(z,\cdot))$ is the Wronskian of $u_2(z,\cdot)$ and $u_1(z,\cdot)$,
\begin{align}
\begin{split}
W(u_2(z,\cdot),u_1(z,\cdot))
& = u_2(z,a)u^{[1]}_1(z,a) - u^{[1]}_2(z,a)u_1(z,a) \\
& = u^{[1]}_1(z,a) = -u^{[1]}_2(z,b),
\end{split} \lb{4.26}
\end{align}
and $I_{(a,b)}$ denotes the identity operator in $L^2((a,b);rdx)$.

Now it follows from \eqref{A.27} and \eqref{4.24}--\eqref{4.26} that
\begin{align}
\ga_N(H_D-z I_{(a,b)})^{-1}g &=
\frac{1}{W(u_2(z,\cdot),u_1(z,\cdot))}
\begin{pmatrix}
u_1^{[1]}(z,a)\int_a^b dx'\,u_2(z,x')g(x')
\\[1mm]
-u_2^{[1]}(z,b)\int_a^b dx'\,u_1(z,x')g(x')
\end{pmatrix} \no
\\ &=
\begin{pmatrix}
\big(\,\ol{u_2(z,\cdot)},g\big)_{L^2((a,b);rdx)}
\\[1mm]
\big(\,\ol{u_1(z,\cdot)},g\big)_{L^2((a,b);rdx)}
\end{pmatrix}, \quad g\in L^2((a,b);rdx). \lb{4.27}
\end{align}
Thus, changing $z$ to $\bz$ and noting that $u_j(\bz,\cdot)=\ol{u_j(z,\cdot)}$, $j=1,2$, (cf. \eqref{3.15a}) one obtains from \eqref{4.27},
\begin{align}
\big[\ga_N(H_D-\bz I_{(a,b)})^{-1}\big]^* \begin{pmatrix}c_2\\c_1\end{pmatrix}
= c_1u_1(z,\cdot)+c_2u_2(z,\cdot), \lb{4.28}
\end{align}
and hence, by \eqref{4.23},
\begin{align}
\La_D^\AB(z) &= \ga_\AB \big[\ga_N(H_D - \bz I_{(a,b)})^{-1}\big]^*, \quad z\in\rho(H_D). \lb{4.29}
\end{align}
In addition, we note that, by \eqref{4.8c}, $\ga_D(H_D-z I_{(a,b)})^{-1}=0$, and hence \eqref{A.29} implies
\begin{align}
\ga_\AB(H_D - z I_{(a,b)})^{-1} = N_{A,B}\ga_N(H_D - z I_{(a,b)})^{-1}, \quad z\in\rho(H_D). \lb{4.30}
\end{align}
Thus, combining \eqref{4.29} with \eqref{4.30} yields
\begin{align}
\La_D^\AB(z) N^*_{A,B} &= \ga_\AB \big[\ga_\AB(H_D - \bz I_{(a,b)})^{-1}\big]^*, \quad z\in\rho(H_D). \lb{4.31}
\end{align}
We will obtain analogous formulas for the general boundary data map $\La_\AB^\ABp$ after two short preparatory lemmas.

%%%%%%%%%%%
\begin{lemma} \lb{l4.3}
Assume that $A, B, A', B' \in\bbC^{2\times2}$, where $\AB$ and $\ABp$ satisfy
\eqref{A.7}, and let the self-adjoint extensions $H_\AB, H_\ABp$ be defined as in
\eqref{4.8b}. In addition, let $S_{A',B',A,B}\in\bbC^{2\times2}$ be as in \eqref{4.8i},
\eqref{4.8j}, and suppose that $z\in\rho(H_\AB)$. Then
\begin{align}
&\ran\big(\ga_\ABp(H_\AB - z I_{(a,b)})^{-1}\big) = \ran\big(S_{A',B',A,B}\big), \lb{4.32a}
\\
&\ran\big(\ga_\AB(H_\ABp - z I_{(a,b)})^{-1}\big) = \ran\big(S^*_{A',B',A,B}\big). \lb{4.32b}
\end{align}
In particular,
\begin{align}
&\ran\big(\ga_\AB^\perp(H_\AB - z I_{(a,b)})^{-1}\big) = \bbC^2. \lb{4.32c}
\end{align}
\end{lemma}
%%%%%%%%%%%%
\begin{proof}
First, one notes that it suffices to establish \eqref{4.32c} since \eqref{4.32a} and \eqref{4.32b} follow from \eqref{4.8i}, \eqref{4.8o}, and \eqref{4.32c}.

Let $\phi,\psi\in\dom(H_\max)$, then using integration by parts, \eqref{A.29}, and \eqref{4.8l}, one computes
\begin{align}
\begin{split}
& \big((H_\max - \bz I_{(a,b)}) \phi, \psi\big)_{L^2((a,b);rdx)} -
\big(\phi, (H_\max - z I_{(a,b)}) \psi\big)_{L^2((a,b);rdx)}
\\
& \quad = - \int_a^b dx \, \ol{\big(\phi^{[1]}\big)'(x)} \psi(x) + \int_a^b dx \, \ol{\phi(x)} \big(\psi^{[1]}\big)'(x)
\\
& \quad = - \ol{\phi^{[1]}(b)} \psi(b) + \ol{\phi^{[1]}(a)} \psi(a) +
\ol{\phi(b)} \psi^{[1]}(b) - \ol{\phi(a)} \psi^{[1]}(a)
\\
&\quad = \big(\ga_N\phi,\ga_D\psi\big)_{\bbC^2} - \big(\ga_D\phi,\ga_N\psi\big)_{\bbC^2}
\\
&\quad = \big(\ga_\AB^\perp\phi,\ga_\AB\psi\big)_{\bbC^2} - \big(\ga_\AB\phi,\ga_\AB^\perp\psi\big)_{\bbC^2}.
\end{split} \lb{4.32d}
\end{align}

Next, pick an arbitrary $v=\begin{pmatrix}v_1 & v_2\end{pmatrix}^\top\in\bbC^2$ and using Lemma \ref{l4.2} let $\{\phi_1,\phi_2\}$ be the basis of solutions of $\tau\phi = \bz\phi$ with
\begin{align}
\ga_\AB\phi_1 = \begin{pmatrix}1 & 0\end{pmatrix}^\top, \quad \ga_\AB\phi_2=\begin{pmatrix} 0 & 1\end{pmatrix}^\top. \lb{4.32e}
\end{align}
Since, by construction, the functions $\phi_1$ and $\phi_2$ are linearly independent, the matrix
\begin{align}
M=\big((\phi_j,\phi_k)_{L^2((a,b);rdx)}\big)_{j,k=1,2} \lb{4.32f}
\end{align}
is invertible. To establish \eqref{4.32c}, we will show that the function
\begin{align}
\psi(\cdot)=\big((H_\AB - z I_{(a,b)})^{-1}\phi_1(\cdot),(H_\AB - z I_{(a,b)})^{-1}\phi_2(\cdot)\big)M^{-1}v , \lb{4.32g}
\end{align}
satisfies $\ga_\AB^\perp\psi=v$. Indeed, since by construction $(H_\max-\bz
I_{(a,b)})\phi_j=0$, $j=1,2$, and by \eqref{4.8c}, $\ga_\AB\psi=0$, it
follows from \eqref{4.32d} that
\begin{align}
\big(\phi_j, (H_\max - z I_{(a,b)}) \psi\big)_{L^2((a,b);rdx)} = \big(\ga_\AB\phi_j,\ga_\AB^\perp\psi\big)_{\bbC^2}, \quad j=1,2. \lb{4.32h}
\end{align}
Substituting \eqref{4.32e}--\eqref{4.32g} into \eqref{4.32h} then yields,
\begin{align}
v_j = \big((\phi_j,\phi_1)_{L^2((a,b);rdx)},(\phi_j,\phi_2)_{L^2((a,b);rdx)}\big)M^{-1}v
= (\ga_\AB^\perp\psi)_j, \quad j=1,2.
\end{align}
\end{proof}
%%%%%%%%%%%%

%%%%%%%%%%%%
\begin{lemma} \lb{l4.4}
Assume that $A, B, A', B' \in\bbC^{2\times2}$, where $\AB$ and $\ABp$ satisfy
\eqref{A.7}, and let the self-adjoint extensions $H_\AB, H_\ABp$ be defined
as in \eqref{4.8b}. In addition, let $S_{A',B',A,B}\in\bbC^{2\times2}$ be as in
\eqref{4.8i}, \eqref{4.8j}, and suppose that $z\in\rho(H_\AB)\cap\rho(H_\ABp)$. Then
\begin{align}
\begin{split}
& (H_\ABp - z I_{(a,b)})^{-1} = (H_\AB - z I_{(a,b)})^{-1}    \\
&\quad + \big[\ga_\AB^\perp(H_\AB - \bz I_{(a,b)})^{-1}\big]^*
\big[\ga_\AB(H_\ABp - z I_{(a,b)})^{-1}\big]. \lb{4.33}
\end{split}
\end{align}
In addition, depending on the $\rank(S_{A',B',A,B})$, one of the following three alternatives holds: If $rank(S_{A',B',A,B})=2$, that is, if the matrix $S_{A',B',A,B}$ is invertible, then
\begin{align}
& (H_\ABp - z I_{(a,b)})^{-1} = (H_\AB - z I_{(a,b)})^{-1}    \lb{4.34} \\
&\quad + \big[\ga_\ABp(H_\AB - \bz I_{(a,b)})^{-1}\big]^* \big[S_{A',B',A,B}^{-1}\big]^*
\big[\ga_\AB(H_\ABp - z I_{(a,b)})^{-1}\big].     \no
\end{align}
If $S_{A',B',A,B}$ is not invertible, then either $\rank(S_{A',B',A,B})=1$ and
\begin{align}
&(H_\ABp - z I_{(a,b)})^{-1} = (H_\AB - z I_{(a,b)})^{-1}     \no \\
&\quad + \big[\ga_\ABp(H_\AB - \bz I_{(a,b)})^{-1}\big]^*
  \|S_{A',B',A,B}\|^{-2}  S_{A',B',A,B}   \lb{4.35}  \\
& \qquad \times \big[\ga_\AB(H_\ABp - z I_{(a,b)})^{-1}\big],    \no
\end{align}
or $\rank(S_{A',B',A,B})=0$ $($i.e., $S_{A',B',A,B}=0$$)$ and then
\begin{align}
& (H_\ABp - z I_{(a,b)})^{-1} = (H_\AB - z I_{(a,b)})^{-1}. \lb{4.36}
\end{align}
\end{lemma}
%%%%%%%%%%%%
\begin{proof}
To get started, we pick $f, g \in L^2((a,b);rdx)$ and introduce
\begin{align}
\begin{split}
\phi &= (H_\AB - \bz I_{(a,b)})^{-1} f \in \dom(H_\AB), \\
\psi &= (H_\ABp - z I_{(a,b)})^{-1} g \in \dom(H_\ABp). \lb{4.37}
\end{split}
\end{align}
Then using \eqref{4.32d} and the fact that by \eqref{4.8c}, $\ga_\AB\phi=0$, one computes
\begin{align}
\begin{split}
& \big(f, (H_\ABp - z I_{(a,b)})^{-1}g\big)_{L^2((a,b);rdx)} -
\big(f, (H_\AB - z I_{(a,b)})^{-1}g\big)_{L^2((a,b);rdx)}
\\
& \quad = \big((H_\AB - \bz I_{(a,b)}) \phi, \psi\big)_{L^2((a,b);rdx)} -
\big(\phi, (H_\ABp - z I_{(a,b)}) \psi\big)_{L^2((a,b);rdx)}
\\
& \quad = \big(\ga_\AB^\perp\phi,\ga_\AB\psi\big)_{\bbC^2}
\\
& \quad = \big(\ga_\AB^\perp(H_\AB - \bz I_{(a,b)})^{-1}f,\ga_\AB(H_\ABp - z I_{(a,b)})^{-1}g\big)_{\bbC^2}
\\
& \quad = \big(f,\big[\ga_\AB^\perp(H_\AB - \bz I_{(a,b)})^{-1}\big]^*
\big[\ga_\AB(H_\ABp - z I_{(a,b)})^{-1}\big]g)\big)_{L^2((a,b);rdx)}. \lb{4.39}
\end{split}
\end{align}
Since $f$ and $g$ are arbitrary elements of $L^2((a,b);rdx)$, \eqref{4.33} follows from \eqref{4.39}.

Next, we note that by \eqref{4.8c} and \eqref{4.8i},
\begin{align}
\ga_\ABp(H_\AB - z I_{(a,b)})^{-1} &= S_{A',B',A,B} \, \ga_\AB^\perp(H_\AB - z I_{(a,b)})^{-1},
\quad z\in\rho(H_\AB).      \lb{4.40}
\end{align}
Thus, if $S_{A',B',A,B}$ is invertible, then \eqref{4.34} follows immediately from
\eqref{4.33} and \eqref{4.40}. Alternatively, if $S_{A',B',A,B}$ is not invertible, then
by \eqref{4.32b}, \eqref{4.33} is equivalent to
\begin{align}
& (H_\ABp - z I_{(a,b)})^{-1} = (H_\AB - z I_{(a,b)})^{-1}    \lb{4.42} \\
&\quad + \big[\ga_\AB^\perp(H_\AB - \bz I_{(a,b)})^{-1}\big]^*
P_{\ran(S^*_{A',B',A,B})}\big[\ga_\AB(H_\ABp - z I_{(a,b)})^{-1}\big],    \no
\end{align}
where $P_{\ran(S^*_{A',B',A,B})}$ is the orthogonal projection in $\bbC^2$ onto the range
of $S^*_{A',B',A,B}$.

Finally, if $S_{A',B',A,B}=0$ then $S^*_{A',B',A,B} =0$ as well, and \eqref{4.36} follows from \eqref{4.42}. If $S_{A',B',A,B}$ is not invertible and nonzero then
\begin{equation}
\|S_{A',B',A,B}\|^{-2} S^*_{A',B',A,B} S_{A',B',A,B} =P_{\ran(S^*_{A',B',A,B})}.
\end{equation}
Applying $\|S_{A',B',A,B}\|^{-2}S^*_{A',B',A,B}$ to both sides of \eqref{4.40} one
obtains
\begin{align}
& \|S_{A',B',A,B}\|^{-2}S^*_{A',B',A,B} \ga_\ABp(H_\AB - z I_{(a,b)})^{-1}   \no \\
& \quad =
\|S_{A',B',A,B}\|^{-2}S^*_{A',B',A,B} S_{A',B',A,B} \ga_\AB^\perp(H_\AB - z I_{(a,b)})^{-1}
\no \\
& \quad = P_{\ran(S^*_{A',B',A,B})}\ga_\AB^\perp(H_\AB - z I_{(a,b)})^{-1}. \lb{4.43}
\end{align}
Taking adjoints on both sides of \eqref{4.43}, replacing $z$ by $\bz$, and substituting into \eqref{4.42} then yields \eqref{4.35}.
\end{proof}
%%%%%%%%%%%%

Next, we derive a representation of the general boundary data map $\La_\AB^\ABp(z)$ in terms of the resolvent $(H_\AB-z I_{(a,b)})^{-1}$ and the boundary trace map
$\ga_\ABp$.

%%%%%%%%%%%%
\begin{theorem} \lb{t4.4a}
Assume that $A, B, A', B' \in\bbC^{2\times2}$, where $\AB$ and $\ABp$ satisfy
\eqref{A.7}, let the self-adjoint extensions $H_\AB, H_\ABp$ be defined
as in \eqref{4.8b}, and let $S_{A',B',A,B}\in\bbC^{2\times2}$ be as in
\eqref{4.8i}, \eqref{4.8j}. Then
\begin{align}
\La_\AB^\ABp(z)S^*_{A',B',A,B} = \ga_\ABp\big[\ga_\ABp(H_\AB-\bz I_{(a,b)})^{-1}\big]^*,   \quad z\in\rho(H_\AB). \lb{4.43A}
\end{align}
\end{theorem}
%%%%%%%%%%%%
\begin{proof}
Applying the boundary trace $\ga_\AB$ on both sides of \eqref{4.33} and using  \eqref{4.8c}, one obtains
\begin{align}
\begin{split}
& \ga_\AB(H_\ABp-z I_{(a,b)})^{-1}     \\
& \quad = \ga_\AB\big[\ga_\AB^\perp(H_\AB-\bz I_{(a,b)})^{-1}\big]^*
\big[\ga_\AB(H_\ABp-z I_{(a,b)})^{-1}\big]. \lb{4.44}
\end{split}
\end{align}
Taking $\ABp$ in \eqref{4.44} to be such that $\ga_\AB=\ga_\ABp^\perp$ and recalling \eqref{4.32c} yields
\begin{align}
\ga_\AB\big[\ga_\AB^\perp(H_\AB - \bz I_{(a,b)})^{-1}\big]^* = I_2. \lb{4.45}
\end{align}
Then for every $c=\begin{pmatrix}c_1 & c_2\end{pmatrix}^\top\in\bbC^2$ the function,
\begin{align}
u_\AB(z,\, \cdot \, ;c_1,c_2)=\big[\ga_\AB^\perp(H_\AB - \bz I_{(a,b)})^{-1}\big]^*\begin{pmatrix}c_1\\c_2\end{pmatrix}, \lb{4.46}
\end{align}
solves the boundary value problem \eqref{4.9}, \eqref{4.10}. Indeed,
\begin{align}
\ga_\AB u_\AB(z, \, \cdot \, ;c_1,c_2) = \begin{pmatrix}c_1\\c_2\end{pmatrix}, \lb{4.46a}
\end{align}
by \eqref{4.45}, and $(H_\max - z I_{(a,b)})u_\AB(z, \, \cdot \, ;c_1,c_2) = 0$ since
\begin{align}
& \big((H_\max-z I_{(a,b)})u_\AB(z,\, \cdot \, ;c_1,c_2),f\big)_{L^2((a,b);rdx)}   \no \\
& \quad = \big(u_\AB(z,\, \cdot \, ;c_1,c_2),(H_{\min}-\bz I_{(a,b)})f\big)_{L^2((a,b);rdx)}   \no \\
& \quad = \big(c,\ga_\AB^\perp(H_\AB - \bz I_{(a,b)})^{-1}(H_{\min}-\bz I_{(a,b)})f\big)_{\bbC^2}   \no \\
& \quad = \big(c,\ga_\AB^\perp f\big)_{\bbC^2} = 0, \quad f\in\dom(H_{\min}), \lb{4.46b}
\end{align}
and $\dom(H_{\min})$ is dense in $L^2((a,b);rdx)$. Thus, according to the definition of $\La_\AB^\ABp$ in \eqref{4.17}, one obtains
\begin{align}
\La_\AB^\ABp = \ga_\ABp\big[\ga_\AB^\perp(H_\AB - \bz I_{(a,b)})^{-1}\big]^*, \quad z\in\rho(H_\AB). \lb{4.47}
\end{align}
In addition, we note that \eqref{4.8c} and \eqref{4.8i} imply
\begin{align}
\ga_\ABp(H_\AB - z I_{(a,b)})^{-1}
= S_{A',B',A,B} \, \ga_\AB^\perp(H_\AB - z I_{(a,b)})^{-1}, \quad z\in\rho(H_\AB), \lb{4.48}
\end{align}
and hence, combining \eqref{4.47} with \eqref{4.48} yields \eqref{4.43A}.
\end{proof}
%%%%%%%%%%%%

One can use the representation \eqref{4.43A} to prove that
$\La_\AB^\ABp(\cdot)S^*_{A',B',A,B}$ is a $2 \times 2$ matrix-valued
Nevanlinna--Herglotz function (cf.\ the proof of Theorem\ 4.6 in \cite{CGM10}). In this
paper we will pursue an alternative route based on Krein's resolvent formula in
Corollary \ref{c4.11}.

Next, we explore reflection symmetry of the expressions in \eqref{4.43A}. Applying
$\ga_\ABp$ to both sides of \eqref{4.33} and using \eqref{4.47} and the fact that
$\ga_\ABp(H_\ABp - z I_{(a,b)})^{-1}=0$, by \eqref{4.8c}, one obtains
\begin{align}
\ga_\ABp(H_\AB - z I_{(a,b)})^{-1} = -\La_\AB^\ABp(z) \big[\ga_\AB(H_\ABp - z I_{(a,b)})^{-1}\big]. \lb{4.50}
\end{align}
Using the identities \eqref{4.8c}, \eqref{4.8i}, \eqref{4.8m}, and \eqref{4.8o} in \eqref{4.50} then yields
\begin{align}
\begin{split}
& S_{A',B',A,B}\ga_\AB^\perp(H_\AB - z I_{(a,b)})^{-1}    \\
& \quad = \La_\AB^\ABp(z) S^*_{A',B',A,B} \big[\ga_\ABp^\perp
(H_\ABp - z I_{(a,b)})^{-1}\big].     \lb{4.51}
\end{split}
\end{align}
Changing $z$ to $\bz$, taking adjoints, applying $\ga_\ABp$ to both sides of \eqref{4.51}, and utilizing \eqref{4.45} and \eqref{4.47} then implies,
\begin{align}
\La_\AB^\ABp(z) S^*_{A',B',A,B} &= S_{A',B',A,B} \La_\AB^\ABp(\bz)^*   \no \\
&= \Big(\La_\AB^\ABp(\bz) S^*_{A',B',A,B}\Big)^*, \quad z\in\rho(H_\AB). \lb{4.52}
\end{align}

The principal result of this section, Krein's resolvent formula for the difference of
resolvents of $H_\ABp$ and $H_\AB$, then reads as follows:

%%%%%%%%%%%%
\begin{theorem} \lb{t4.5}
Assume that $A, B, A', B' \in\bbC^{2\times2}$, where $\AB$ and $\ABp$ satisfy
\eqref{A.7}, and let the self-adjoint extensions $H_\AB, H_\ABp$ be defined as in
\eqref{4.8b}. In addition, let $S_{A',B',A,B}\in\bbC^{2\times2}$ be as in
\eqref{4.8i}, \eqref{4.8j}, and suppose that $z\in\rho(H_\AB)\cap\rho(H_\ABp)$. Then
\begin{align}
&(H_\ABp - z I_{(a,b)})^{-1} = (H_\AB - z I_{(a,b)})^{-1}      \lb{4.149}
\\
&\quad - \big[\ga_\AB^\perp(H_\AB - \bz I_{(a,b)})^{-1}\big]^* \La_\AB^\ABp(z)^{-1}
S_{A',B',A,B} \big[\ga_\AB^\perp(H_\AB - z I_{(a,b)})^{-1}\big].    \no
\end{align}
In addition, if $S_{A',B',A,B}$ is invertible $($i.e., $\rank(S_{A',B',A,B})=2$$)$, then
\begin{align}
& (H_\ABp - z I_{(a,b)})^{-1} = (H_\AB - z I_{(a,b)})^{-1} \no \\
& \quad - \big[\ga_\ABp(H_\AB - \bz I_{(a,b)})^{-1}\big]^*
\big[\La_\AB^\ABp(z) S^*_{A',B',A,B}\big]^{-1}      \lb{4.150} \\
& \qquad \times \big[\ga_\ABp(H_\AB - z I_{(a,b)})^{-1}\big].    \no
\end{align}
If $S_{A',B',A,B}$ is not invertible and nonzero $($i.e., $\rank(S_{A',B',A,B})=1$$)$, then
\begin{align}
&(H_\ABp - z I_{(a,b)})^{-1} = (H_\AB - z I_{(a,b)})^{-1} \no \\
&\quad - \big[\ga_\ABp(H_\AB - \bz I_{(a,b)})^{-1}\big]^* \big[\la_\AB^\ABp(z)\big]^{-1}
\big[\ga_\ABp(H_\AB - z I_{(a,b)})^{-1}\big], \lb{4.151}
\end{align}
where
\begin{equation}
\la_\AB^\ABp(z)
= P_{\ran(S_{A',B',A,B})}\La_\AB^\ABp(z) S^*_{A',B',A,B}
P_{\ran(S_{A',B',A,B})}\big\vert_{\ran(S_{A',B',A,B})}. \lb{4.151a}
\end{equation}
\end{theorem}
%%%%%%%%%%%%%
\begin{proof}
First, using \eqref{4.48} and the fact that $\La_\AB^\ABp(z)$ is invertible, one rewrites \eqref{4.50} as
\begin{align}
\begin{split}
& \ga_\AB(H_\ABp - z I_{(a,b)})^{-1} = - \La_\AB^\ABp(z)^{-1} \big[\ga_\ABp(H_\AB - z I_{(a,b)})^{-1}\big]
\\
& \quad = - \La_\AB^\ABp(z)^{-1} S_{A',B',A,B} \big[\ga_\AB^\perp(H_\AB - z I_{(a,b)})^{-1}\big].
\lb{4.152}
\end{split}
\end{align}
Then inserting \eqref{4.152} into \eqref{4.33} yields \eqref{4.149}.

Next, if $S_{A',B',A,B}$ is invertible then combining \eqref{4.48} and
\eqref{4.43A} with \eqref{4.149} implies \eqref{4.150}. In the case
$S_{A',B',A,B}$ is not invertible and nonzero, it follows from \eqref{4.32b}
and \eqref{4.50} that
\begin{align}
& -\ga_\ABp(H_\AB - z I_{(a,b)})^{-1}     \lb{4.153} \\
& \quad = \La_\AB^\ABp(z)S^*_{A',B',A,B} \|S_{A',B',A,B}\|^{-2}
S_{A',B',A,B} \big[\ga_\AB(H_\ABp - z I_{(a,b)})^{-1}\big].    \no
\end{align}
Since $\ran(\ga_\ABp(H_\AB - z I_{(a,b)})^{-1})=\ran(S_{A',B',A,B})$ by \eqref{4.32a}, it follows from
\eqref{4.153} that
\begin{align}
\begin{split}
&  \|S_{A',B',A,B}\|^{-2} S_{A',B',A,B} \big[\ga_\AB(H_\ABp - z I_{(a,b)})^{-1}\big]   \\
& \quad = -\big[\la_\AB^\ABp(z)\big]^{-1}\big[\ga_\ABp(H_\AB - z I_{(a,b)})^{-1}\big]. \lb{4.154}
\end{split}
\end{align}
Inserting \eqref{4.154} into \eqref{4.35} yields \eqref{4.151}.
\end{proof}
%%%%%%%%%%%%

It is instructive to compare the resolvent formulas obtained via the boundary data map approach in Theorem \ref{t4.5} with the resolvent formulas in Krein's abstract approach discussed in Appendix \ref{sB}, and more concretely, in Theorems \ref{t3.1} and \ref{t3.2}. For this purpose we now restate the resolvent formulas \eqref{4.149} and \eqref{4.151} using an explicit basis of $\ker(H_\max-z I_{(a,b)})$.

%%%%%%%%%%%%
\begin{corollary} \lb{c4.6}
Assume that $A, B, A', B' \in\bbC^{2\times2}$, where $\AB$ and $\ABp$ satisfy
\eqref{A.7}, and let the self-adjoint extensions $H_\AB, H_\ABp$ be defined as in
\eqref{4.8b}. In addition, let $S_{A',B',A,B}\in\bbC^{2\times2}$ be as in
\eqref{4.8i}, \eqref{4.8j}, and suppose that $z\in\rho(H_\AB)\cap\rho(H_\ABp)$. \\
$(i)$ If $S_{A',B',A,B}$ is invertible $($i.e., $\rank(S_{A',B',A,B})=2$$)$, then
\begin{align}
\begin{split}
&(H_\ABp - z I_{(a,b)})^{-1} = (H_\AB - z I_{(a,b)})^{-1}    \\
& \quad - \sum_{k,n=1}^2 \big[P_{A',B',A,B}(z)\big]^{-1}_{k,n}
(u_{A,B,n}(\overline{z},\cdot),\cdot)_{L^2((a,b);rdx)} u_{A,B,k}(z,\cdot), \lb{4.155}
\end{split}
\end{align}
where the $2\times2$ matrix $P_{A',B',A,B}(\cdot)$ is given by
\begin{align}
P_{A',B',A,B}(z)=S_{A',B',A,B}^{-1}\La_\AB^\ABp(z) \lb{4.156}
\end{align}
and $\{u_{A,B,1}(z,\cdot),u_{A,B,2}(z,\cdot)\}$ is the basis of $\ker(H_\max-z I_{(a,b)})$ satisfying the boundary conditions
\begin{align}
\ga_\AB u_{A,B,1}(z,\cdot) =\begin{pmatrix} 1\\0\end{pmatrix}, \qquad \ga_\AB u_{A,B,2}(z,\cdot) = \begin{pmatrix} 0\\1\end{pmatrix}. \lb{4.156a}
\end{align}
$(ii)$ If $S_{A',B',A,B}$ is not invertible and nonzero $($i.e., $\rank(S_{A',B',A,B})=1$$)$, then
\begin{align}
\begin{split}
&(H_\ABp - z I_{(a,b)})^{-1}  = (H_\AB - z I_{(a,b)})^{-1}     \\
&\quad - p_{A',B',A,B}(z)^{-1}(u_{A',B',A,B,0}(\overline{z},\cdot),\cdot)_{L^2((a,b);rdx)}
u_{A',B',A,B,0}(z,\cdot),    \lb{4.157}
\end{split}
\end{align}
where the scalar $p_{A',B',A,B}(\cdot)$ is given by
\begin{align}
p_{A',B',A,B}(z) = P_{\ran(S_{A',B',A,B})}\La_\AB^\ABp(z) S^*_{A',B',A,B}
P_{\ran(S_{A',B',A,B})}\big\vert_{\ran(S_{A',B',A,B})} \lb{4.158}
\end{align}
and the element $u_{A',B',A,B,0}(z,\cdot)\in\ker(H_\max-z I_{(a,b)})$ is given by
\begin{align}
u_{A',B',A,B,0}(z,\cdot) = \big[\ga_\ABp(H_\AB-\bz I_{(a,b)})^{-1}\big]^* \big\vert_{\ran(S_{A',B',A,B})}. \lb{4.159}
\end{align}
\end{corollary}
%%%%%%%%%%%%
\begin{proof}
It follows from \eqref{4.46}--\eqref{4.46b} that the maps
\begin{align}
\begin{split}
&\big[\ga_\AB^\perp(H_\AB - \bz I_{(a,b)})^{-1}\big]^*\! : \bbC^2 \to \ker(H_\max-z I_{(a,b)}), \\
&\big[\ga_\AB^\perp(H_\AB-z I_{(a,b)})^{-1}\big]\; : \ker(H_\max-\bz I_{(a,b)}) \to \bbC^2,
\end{split} \lb{4.160}
\end{align}
are given by
\begin{align}
\begin{split}
&\big[\ga_\AB^\perp(H_\AB - \bz I_{(a,b)})^{-1}\big]^*
\begin{pmatrix}c_1\\c_2\end{pmatrix}
= c_1 u_{A,B,1}(z,\cdot) + c_2 u_{A,B,2}(z,\cdot), \quad c_1,c_2\in\bbC,
\\
&\big[\ga_\AB^\perp(H_\AB-z I_{(a,b)})^{-1}\big]f =
\begin{pmatrix}
\big(u_{A,B,1}(\bz,\cdot),f\big)_{L^2((a,b);rdx)}\\
\big(u_{A,B,2}(\bz,\cdot),f\big)_{L^2((a,b);rdx)}
\end{pmatrix}, \quad f\in L^2((a,b);rdx).
\end{split} \lb{4.161}
\end{align}
Thus, if $S_{A',B',A,B}$ is invertible, \eqref{4.155} and \eqref{4.156} follow from
\eqref{4.149} and \eqref{4.161}.

If $S_{A',B',A,B}$ is not invertible and nonzero it follows from
\eqref{4.32a}, \eqref{4.48}, and \eqref{4.160} that $\big[\ga_\ABp(H_\AB - z
I_{(a,b)})^{-1}\big]$  is  surjective, mapping $\ker(H_\max-\bz I_{(a,b)})$
onto the one-dimensional subspace $\ran(S_{\ABp,\AB})\subset\bbC^2$. Hence
$\big[\ga_\ABp(H_\AB-\bz I_{(a,b)})^{-1}\big]^*$ maps $\ran(S_{\ABp,\AB})$
onto a one dimensional subspace of $\ker(H_\max-z I_{(a,b)})$ spanned by the
function $u_{A',B',A,B,0}(z,\cdot)$. Thus,
\begin{align}
\begin{split}
&\big[\ga_\ABp^\perp(H_\AB - \bz I_{(a,b)})^{-1}\big]^*\! : \ran(S_{\ABp,\AB}) \to \spn(u_{A',B',A,B,0}(z,\cdot)), \\
&\big[\ga_\ABp^\perp(H_\AB-z I_{(a,b)})^{-1}\big]\; : \spn(u_{A',B',A,B,0}(\bz,\cdot)) \to \ran(S_{\ABp,\AB}),
\end{split} \lb{4.162}
\end{align}
and hence \eqref{4.157}--\eqref{4.159} follow from \eqref{4.151} and \eqref{4.151a}.
\end{proof}
%%%%%%%%%%%%

The above result shows that, depending on the rank of $S_{A',B',A,B}$, the abstract Krein's formula \eqref{B.16} is equivalent either to \eqref{4.155} (and hence to \eqref{4.149}) or to \eqref{4.157} (and hence to \eqref{4.151}). Moreover, straightforward computations show that in the special case of $H_{A,B}=H_D$, Corollary \ref{c4.6} reduces to Theorem
\ref{t3.1} if $H_{\ABp}$ corresponds to separated boundary conditions \eqref{A.13} and to Theorem \ref{t3.2} if $H_{\ABp}$ corresponds to non-separated boundary conditions
\eqref{A.15}. Explicitly, one obtains the following result.

%%%%%%%%%%%%
\begin{corollary} \lb{c4.7}
Assume that $H_\AB=H_D$ $($i.e., $A=A_D$ and $B=B_D$ given by \eqref{A.28a}$)$
and $\ABp\in\bbC^{2\times2}$ satisfy \eqref{A.7}. Suppose that
$z\in\rho(H_D)\cap\rho(H_\ABp)$, and let $\{u_1(z,\cdot),u_2(z,\cdot)\}$ be the basis
of $\ker(H_\max-z I_{(a,b)})$ dictated by \eqref{3.8}.
\begin{enumerate}[$(i)$]
\item If $A'=\begin{pmatrix}\cos(\te_a)&\sin(\te_a)\\0&0\end{pmatrix}$,
$B'=\begin{pmatrix}0&0\\-\cos(\te_b)&\sin(\te_b)\end{pmatrix}$,
$\te_a,\te_b\in(0,\pi)$, then \eqref{4.155} holds with $P_{A',B',A,B}(z)= \begin{pmatrix}0&1\\1&0\end{pmatrix} D_{\te_a,\te_b}(z) \begin{pmatrix}0&1\\1&0\end{pmatrix}$, where $D_{\te_a,\te_b}(z)$ is given by \eqref{3.16}.

\smallskip

\item If $A'=\begin{pmatrix}\cos(\te_a)&\sin(\te_a)\\0&0\end{pmatrix}$,
$B'=\begin{pmatrix}0&0\\-\cos(\te_b)&\sin(\te_b)\end{pmatrix}$,
$\te_a\in(0,\pi)$, $\te_b=0$, then \eqref{4.157} holds with
$p_{A',B',A,B}(z)= \sin^2(\te_a) d_{\te_a,0}(z)$, where $d_{\te_a,0}(z)$ is given by \eqref{3.19} and $u_{A',B',A,B,0}(z,\cdot)=\sin(\te_a)u_2(z,\cdot)$.

\smallskip

\item If $A'=\begin{pmatrix}\cos(\te_a)&\sin(\te_a)\\0&0\end{pmatrix}$,
$B'=\begin{pmatrix}0&0\\-\cos(\te_b)&\sin(\te_b)\end{pmatrix}$,
$\te_a=0$, $\te_b\in(0,\pi)$, then \eqref{4.157} holds with
$p_{A',B',A,B}(z)= \sin^2 (\te_b) d_{0,\te_b}(z)$, where $d_{0,\te_b}(z)$ is given by \eqref{3.22}
and $u_{A',B',A,B,0}(z,\cdot)=\sin(\te_b)u_1(z,\cdot)$.

\smallskip

\item If $A'=e^{i\phi}F$, $B'=I_2$, $F\in\mathrm{SL}_2(\bbR)$, $F_{1,2}\neq0$,
then \eqref{4.155} holds with $P_{A',B',A,B}(z)= \begin{pmatrix}0&1\\1&0\end{pmatrix} Q_{F,\phi}(z) \begin{pmatrix}0&1\\1&0\end{pmatrix}$, where $Q_{F,\phi}(z)$ is given by \eqref{3.53}.

\smallskip

\item If $A'=e^{i\phi}F$, $B'=I_2$, $F\in\mathrm{SL}_2(\bbR)$, $F_{1,2}=0$, then \eqref{4.157} holds with $p_{A',B',A,B}(z)= q_{F,\phi}(z)$, where the scalar $q_{F,\phi}(z)$ is given by \eqref{3.56}, and for $u_{A',B',A,B,0}(z,\cdot)=u_{F,\phi}(z,\cdot)$, with $u_{F,\phi}(z,\cdot)$ given by \eqref{3.58}.
\end{enumerate}
\end{corollary}
%%%%%%%%%%%%

At this point we are ready to demonstrate the Nevanlinna--Herglotz property of
$\La_\AB^\ABp(\cdot) S_{A',B',A,B}^*$. We denote
$\bbC_+ = \{z \in \bbC \, | \, \Im(z) > 0\}$.

%%%%%%%%%%%%
\begin{corollary} \lb{c4.11}
Assume that $A, B, A', B' \in\bbC^{2\times2}$, where $\AB$ and $\ABp$ satisfy
\eqref{A.7}, and let the self-adjoint extensions $H_\AB, H_\ABp$ be defined as in
\eqref{4.8b}. In addition, let $S_{A',B',A,B}\in\bbC^{2\times2}$ be as in
\eqref{4.8i}, \eqref{4.8j}. If $S_{A',B',A,B}$ is invertible, then
$\La_\AB^\ABp(\cdot) S_{A',B',A,B}^*$ is a $2 \times 2$ matrix-valued
Nevanlinna--Herglotz function satisfying
\begin{equation}
\Im\Big(\La_\AB^\ABp(\cdot) S_{A',B',A,B}^*\Big) > 0, \quad z \in \bbC_+.
\lb{4.103a}
\end{equation}
\end{corollary}
%%%%%%%%%%%%
\begin{proof}
Analyticity of $\La_\AB^\ABp(\cdot) S_{A',B',A,B}^*$ on $z\in\rho(H_\AB)$
follows from that of $\La_\AB^\ABp(\cdot)$ described in Lemma \ref{c4.3}.
Equation \eqref{4.156} then proves that
\begin{align}
S_{A',B',A,B} P(z) S_{A',B',A,B}^* = \La_\AB^\ABp(z) S_{A',B',A,B}^*,
\quad z \in \rho(H_\AB).       \lb{4.103b}
\end{align}
By Theorem \ref{tB.1}\,$(iii)$, $P(\cdot)$ and hence
$S_{A',B',A,B} P(\cdot) S_{A',B',A,B}^*$ is a $2 \times 2$ matrix-valued
Nevanlinna--Herglotz function satisfying \eqref{4.103a} as a consequence of
\eqref{B.36a}.
\end{proof}
%%%%%%%%%%%%

%%%%%%%%%%%%%%%%%%%%%%%%%%%%%%%%%%%%%%%%
%%%%%%%%%%%%%%%%%%%%%%%%%%%%%%%%%%%%%%%%
\section{Trace Formulas, Symmetrized Perturbation Determinants, and Spectral
Shift Functions}    \lb{s5}
%%%%%%%%%%%%%%%%%%%%%%%%%%%%%%%%%%%%%%%%
%%%%%%%%%%%%%%%%%%%%%%%%%%%%%%%%%%%%%%%%

In this section we present the connection between the general boundary data maps, symmetrized perturbation determinants, trace formulas, and spectral shift functions  for general self-adjoint
extensions of $H_{\min}$, described in Theorems \ref{tA.2} and \ref{tA.4}.

Assuming as before Hypothesis \ref{hA.1} and \eqref{A.7}, we start by recalling the sesquilinear
form, denoted by $Q_\AB$, associated with the general self-adjoint extension
$H_\AB$ of $H_{\min}$. If the matrix $N_{A,B}$, defined as in \eqref{A.29}--\eqref{A.30}, is invertible
(i.e., $\rank(N_{A,B})=2$) then  
\begin{align}
& Q_\AB(f,g) = \int_a^b dx \, \big[p(x) \ol{f'(x)} g'(x) + q(x) \ol{f(x)} g(x)\big]
- \big(\ga_D f, N^{-1}_{A,B} D_{A,B} \ga_D g\big)_{\bbC^2},    \no \\
& f, g \in \dom(Q_\AB) = \big\{h\in L^2((a,b); rdx) \,|\,
h \in \AC ([a,b]);       \lb{5.1} \\
& \hspace*{5.1cm} p^{1/2} h' \in L^2((a,b); rdx)\big\}.   \no
\end{align}
If $N_{A,B}$ is not invertible then either $N_{A,B}$ is nonzero (i.e.,
$\rank(N_{A,B})=1$) and
\begin{align}
& Q_\AB(f,g) = \int_a^b dx \big[p(x) \ol{f'(x)} g'(x) + q(x) \ol{f(x)} g(x)\big]
- \f{\big(\ga_D f, N^*_{A,B} D_{A,B} \ga_D g\big)_{\bbC^2}}{\|N_{A,B}\|^2},   \no \\
&  f, g \in \dom(Q_\AB) = \big\{h\in L^2((a,b); rdx) \,|\,
h \in \AC ([a,b]);      \lb{5.2} \\
&  \hspace*{3.4cm} \ga_D h \in \ran(N^*_{A,B}); \, p^{1/2} h' \in L^2((a,b); rdx)\big\},     \no 
\end{align}
or $N_{A,B} =0$ (i.e., $\rank(N_{A,B})=0$) and
\begin{align}
& Q_\AB(f,g) = \int_a^b dx \big[p(x) \ol{f'(x)} g'(x) + q(x) \ol{f(x)} g(x)\big],     \no \\
&  f, g \in \dom(Q_\AB) = \big\{h\in L^2((a,b); rdx) \,|\, h \in \AC ([a,b]); \, \ga_D h = 0;    \lb{5.3} \\
& \hspace*{6.6cm}  p^{1/2} h' \in L^2((a,b); rdx)\big\}. \no
\end{align}

To see the connection between $Q_\AB$ and the self-adjoint extension
$H_\AB$ it suffices to perform an integration by parts. For instance, in the case of \eqref{5.2}, one obtains for all $f\in\dom(Q_\AB)$ and
$g\in\dom(H_\AB)$,
\begin{align}
Q_\AB(f,g) &= \big(f, H_\max g\big)_{L^2((a,b);rdx)} - \big(\ga_D
f,\ga_N g\big)_{\bbC^2} - \f{\big(\ga_D f, N^*_{A,B} D_{A,B} \ga_D
g\big)_{\bbC^2}}{\|N_{A,B}\|^2} \no
\\
&= \big(f, H_\AB g\big)_{L^2((a,b);rdx)} - \f{\big(\ga_D f,
\|N_{A,B}\|^2 \ga_N g + N^*_{A,B} D_{A,B} \ga_D
g\big)_{\bbC^2}}{\|N_{A,B}\|^2}. \lb{5.2a}
\end{align}
Since $\ga_D f \in \ran(N_\AB^*)$ one has $\ga_D f =
\|N_\AB\|^{-2}N_\AB^*N_\AB\ga_D f$ and hence
\begin{align}
\big(\ga_D f, \|N_{A,B}\|^2 \ga_N g\big)_{\bbC^2} = \big(\ga_D f,
N_{A,B}^*N_\AB \ga_N g\big)_{\bbC^2}. \lb{5.2b}
\end{align}
Combining \eqref{5.2a} and \eqref{5.2b} yields,
\begin{align}
Q_\AB(f,g) &= \big(f, H_\AB g\big)_{L^2((a,b);rdx)} - \f{\big(\ga_D f,
N_{A,B}^*(N_\AB \ga_N g + D_{A,B} \ga_D
g)\big)_{\bbC^2}}{\|N_{A,B}\|^2} \no
\\
&=
\big(f, H_\AB g\big)_{L^2((a,b);rdx)},
\end{align}
since $g\in\dom(H_\AB)$, and by \eqref{A.29} and \eqref{4.8b}, 
$\ga_\AB g = D_{A,B} \ga_D g + N_\AB \ga_N g = 0$.

The 2nd representation theorem for densely defined, semibounded, closed quadratic forms
(cf.\ \cite[Sect.\ 6.2.6]{Ka80}) then yields that
\begin{align}
\dom\big((H_\AB - z I_{(a,b)})^{1/2}\big) = \dom\big(|H_\AB|^{1/2}\big) = \dom(Q_\AB),
\quad z\in \bbC\bs[e_\AB,\infty), \lb{5.4}
\end{align}
where we abbreviated
\begin{equation}
e_\AB = \inf (\sigma(H_\AB)). \lb{5.5}
\end{equation}

Here $(H_\AB - z I_{(a,b)})^{1/2}$ is defined with the help of the spectral theorem and
a choice of a branch cut along $[e_\AB,\infty)$. Employing the fact that by \eqref{5.1}--\eqref{5.3},
\begin{align}
& \dom\big((H_\ABp - z I_{(a,b)})^{1/2}\big) = \dom\big(|H_\ABp|^{1/2}\big)    \no \\
& \quad = \big\{h\in L^2((a,b); rdx) \,|\, h \in \AC ([a,b]); \, p^{1/2} h' \in L^2((a,b); rdx)\big\},   \\
& \hspace*{4.75cm} z \in \bbC\bs[e_\ABp,\infty), \; \det(N_{A',B'})\neq0, \no
\\
& \dom\big((H_\AB - z I_{(a,b)})^{1/2}\big) = \dom\big(|H_\AB|^{1/2}\big)    \no \\ 
& \quad \subseteq \big\{h\in L^2((a,b); rdx) \,|\, h \in \AC ([a,b]); \, p^{1/2} h' \in L^2((a,b); rdx)\big\},    \\
& \hspace*{7.55cm} z \in \bbC\bs[e_\AB,\infty),    \no
\end{align}
then shows that
\begin{align}
& \ol{(H_\ABp - z I_{(a,b)})^{1/2} (H_\AB - z I_{(a,b)})^{-1} (H_\ABp - z I_{(a,b)})^{1/2}} \no \\
& \quad = \big[(H_\ABp - z I_{(a,b)})^{1/2} (H_\AB - z I_{(a,b)})^{-1/2}\big]  \no \\
& \qquad \times \big[(H_\ABp - \bz I_{(a,b)})^{1/2} (H_\AB - \bz I_{(a,b)})^{-1/2} \big]^* \no \\
&\quad \in \cB\big(L^2((a,b); rdx)\big), \quad z\in\bbC\bs[e_0,\infty), \;
\det(N_{A',B'})\neq0, \lb{5.8}
\end{align}
where we introduced the abbreviation
\begin{equation}
e_0 =\inf\big(\sigma(H_\AB) \cup \sigma( H_\ABp)\big)
= \min (e_\AB, e_\ABp).
\end{equation}
Then applying Theorem \ref{t4.5} one concludes that actually,
\begin{align}
& \ol{(H_\ABp - z I_{(a,b)})^{1/2} (H_\AB - z I_{(a,b)})^{-1} (H_\ABp - z I_{(a,b)})^{1/2}} - I_{(a,b)} \no
\\
& \quad = \big\{(H_\ABp - z I_{(a,b)})^{1/2}
\big[(H_\AB - z I_{(a,b)})^{-1} - (H_\ABp - z I_{(a,b)})^{-1}\big]    \no \\
& \qquad \times (H_\ABp - z I_{(a,b)})^{1/2}\big\}^{\rm cl} \no
\\
& \quad = (H_\ABp - z I_{(a,b)})^{1/2} \big[\ga_\AB^\perp(H_\AB - \bz I_{(a,b)})^{-1}\big]^*
\La_\AB^\ABp(z)^{-1}S_{A',B',A,B}     \no
\\
&\qquad \times\big[\ga_\AB^\perp(H_\AB - z I_{(a,b)})^{-1}\big] (H_\ABp - z I_{(a,b)})^{1/2},
\lb{5.10} \\
& \hspace*{3.3cm}  z\in\bbC\bs[e_0,\infty), \; \det(N_{A',B'})\neq0.    \no
\end{align}
is a finite-rank (and hence a trace class) operator on $L^2((a,b); rdx)$. Thus, the Fredholm
determinant, more precisely, the symmetrized perturbation determinant,
\begin{align}
& {\det}_{L^2((a,b); rdx)} \Big(\ol{(H_\ABp - z I_{(a,b)})^{1/2}
(H_\AB - z I_{(a,b)})^{-1} (H_\ABp - z I_{(a,b)})^{1/2}}\Big),   \no \\
& \hspace*{6cm} z\in\bbC\bs[e_0,\infty), \; \det(N_{A',B'})\neq0,     \lb{5.11}
\end{align}
is well-defined (cf.\ \cite[Ch.\ IV]{GK69} and \cite[Ch.\ 3]{Si05} for basics on Fredholm determinants).

Next, we show that the symmetrized (Fredholm) perturbation determinant \eqref{5.11} associated
with the pair $(H_\ABp, H_\AB)$ can essentially be reduced to the $2\times2$ determinant of the
general boundary data map $\La_\AB^\ABp(z)$:

%%%%%%%%%%%%%%%%%%%%%%%%%%%%%%%%%%%%%%%%
\begin{theorem} \lb{t5.1}
Assume that $\AB\in\bbC^{2\times2}$, where $\ABp\in\bbC^{2\times2}$ satisfy
\eqref{A.7}, and let the self-adjoint extensions $H_\AB, H_\ABp$ be defined as in
\eqref{4.8b}. In addition, let $N_{A,B},N_{A',B'} \in\bbC^{2\times2}$ be as in
\eqref{A.29}, \eqref{A.30}, and suppose that $\det(N_{A',B'})\neq0$. Then,
\begin{align}
\begin{split}
& {\det}_{L^2((a,b); rdx)}\Big(\ol{(H_\ABp - z I_{(a,b)})^{1/2}
(H_\AB - z I_{(a,b)})^{-1} (H_\ABp - z I_{(a,b)})^{1/2}}\Big) \\
& \quad = \f{{\det}_{\bbC^2}(N_{A,B})}
{{\det}_{\bbC^2}(N_{A',B'})} \, {\det}_{\bbC^2}\Big(\La_\AB^\ABp(z)\Big),
\quad z \in \bbC\bs[e_0,\infty). \lb{5.12}
\end{split}
\end{align}
\end{theorem}
%%%%%%%%%%%%%%%%%%%%%%%%%%%%%%%%%%%%%%%%
\begin{proof}
We start by introducing simplifying abbreviations,
\begin{align}
K_{A,B}(z) &= \big[\ga_\AB^\perp(H_\AB - \bz I_{(a,b)})^{-1}\big]^*, \lb{5.13}
\\
L_{A',B',A,B}(z) &= (H_\ABp - z I_{(a,b)})^{1/2} K_{A,B}(z). \lb{5.14}
\end{align}
Then substitution of \eqref{5.10} into \eqref{5.12} and employing the cyclicity property of the
determinant yields
\begin{align}
& {\det}_{L^2((a,b); rdx)} \Big(\ol{(H_\ABp - z I_{(a,b)})^{1/2}
(H_\AB - z I_{(a,b)})^{-1} (H_\ABp - z I_{(a,b)})^{1/2}}\Big) \no
\\
& \quad = {\det}_{L^2((a,b); rdx)}
\Big(I_{(a,b)} + L_{A',B',A,B}(z) \La_\AB^\ABp(z)^{-1}S_{A',B',A,B} L_{A',B',A,B}(\bz)^*\Big)
\no
\\
& \quad = {\det}_{\bbC^2}
\Big(I_2 + L_{A',B',A,B}(\bz)^*L_{A',B',A,B}(z)
\La_\AB^\ABp(z)^{-1}S_{A',B',A,B} \Big). \lb{5.15}
\end{align}
One notes that $L_{A',B',A,B}(z)$ maps $\bbC^2$ into $L^2((a,b); rdx)$ and hence
the product $L_{A',B',A,B}(\bz)^*L_{A',B',A,B}(z)$ is a linear map on $\bbC^2$.

Next, we turn to the computation of the $2\times2$ matrix representation for the map
$L_{A',B',A,B}(\bz)^*L_{A',B',A,B}(z)$ using \eqref{5.1}--\eqref{5.3},
\begin{align}
&\big(v_1, L_{A',B',A,B}(\bz)^*L_{A',B',A,B}(z)v_2\big)_{\bbC^2}   \no \\
& \quad = \big(L_{A',B',A,B}(\bz) v_1, L_{A',B',A,B}(z)v_2\big)_{L^2((a,b); rdx)} \no
\\
&\quad = Q_\ABp(K_{A,B}(\bz)v_1,K_{A,B}(z)v_2) \no
\\
&\quad = -\big(\ga_D K_{A,B}(\bz)v_1,N_{A',B'}^{-1}D_{A',B'}\ga_D K_{A,B}(z)v_2\big)_{\bbC^2}
\no \\
& \qquad -\big(\ga_D K_{A,B}(\bz)v_1, \ga_N K_{A,B}(z)v_2\big)_{\bbC^2} \no
\\
&\quad = -\big(v_1,[\ga_D K_{A,B}(\bz)]^* N_{A',B'}^{-1}\ga_\ABp K_{A,B}(z)v_2\big)_{\bbC^2}.
\lb{5.16}
\end{align}
Since, by \eqref{4.47} and \eqref{5.13}, $\ga_D K_{A,B}(\bz) = \La_\AB^D(\bz)$ and
$\ga_\ABp K_{A,B}(z) = \La_\AB^\ABp(z)$, it follows from \eqref{5.16} that
\begin{align}
L_{A',B',A,B}(\bz)^*L_{A',B',A,B}(z)
= - \La_\AB^D(\bz)^* N_{A',B'}^{-1}\La_\AB^\ABp(z),
\end{align}
and hence
\begin{align}
& I_2 + L_{A',B',A,B}(\bz)^*L_{A',B',A,B}(z) \La_\AB^\ABp(z)^{-1}S_{A',B',A,B} \no \\
&\quad = I_2 - \La_\AB^D(\bz)^* N_{A',B'}^{-1}S_{A',B',A,B}
\\
& \quad = \big[I_2 - (N_{A',B'}^{-1}S_{A',B',A,B})^*\La_\AB^D(\bz)\big]^*. \lb{5.18}
\end{align}
It follows from \eqref{A.31} and \eqref{4.8j} that
\begin{align}
N_{A',B'}^{-1}S_{A',B',A,B} &= D^*_{A,B} - N_{A',B'}^{-1}D_{A',B'}N_{A',B'}^*  \no \\
&= D^*_{A,B} - D_{A',B'}^* \big(N_{A',B'}^{-1}\big)^* N^*_{A,B},
\end{align}
and hence, by \eqref{4.18} and \eqref{4.18a},
\begin{align}
\begin{split}
& I_2 - \big(N_{A',B'}^{-1}S_{A',B',A,B}\big)^* \La_\AB^D(z)   \no \\
& \quad = \La_\AB^\AB - D_{A,B}\La_\AB^D(z) 
+ N_{A,B} N_{A',B'}^{-1}D_{A',B'}\La_\AB^D(z)
\\
& \quad = N_{A,B} N_{A',B'}^{-1}[N_{A',B'}\La_\AB^N(z)+D_{A',B'}\La_\AB^D(z)]
\\
& \quad = N_{A,B} N_{A',B'}^{-1}\La_\AB^\ABp(z). \lb{5.20}
\end{split}
\end{align}
Substituting \eqref{5.18} and \eqref{5.20} into \eqref{5.15} then yields
\begin{align}
& {\det}_{L^2((a,b); rdx)} \Big(\ol{(H_\ABp - z I_{(a,b)})^{1/2}
(H_\AB - z I_{(a,b)})^{-1} (H_\ABp - z I_{(a,b)})^{1/2}}\Big) \no
\\
& \quad = {\det}_{\bbC^2}\Big(\big[N_{A,B} N_{A',B'}^{-1}\La_\AB^\ABp(\bz)\big]^*\Big).
\end{align}
Changing $z$ to $\bz$ and taking complex conjugation on both sides then
implies \eqref{5.12}.
\end{proof}
%%%%%%%%%%%%%%%%%%%%%%%%%%%%%%%%%%%%%%%%

%%%%%%%%%%
\begin{remark} \lb{r5.1a}
It was crucial in Theorem \ref{t5.1} to use the symmetrized (Fredholm)
perturbation determinant,
\begin{equation}
{\det}_{L^2((a,b); rdx)}\Big(\ol{(H_\ABp - z I_{(a,b)})^{1/2}
(H_\AB - z I_{(a,b)})^{-1} (H_\ABp - z I_{(a,b)})^{1/2}}\Big),
\end{equation}
as in all nontrivial circumstances the ``standard'' perturbation determinant,
\begin{equation}
{\det}_{L^2((a,b); rdx)} \big((H_\ABp - z I_{(a,b)})(H_\AB - z I_{(a,b)})^{-1}\big),
\end{equation}
does not exist since $(H_\AB - z I_{(a,b)})^{-1}$ will not map
$L^2((a,b); rdx)$ into the set $\dom(H_\ABp)$ (it maps into
$\dom(H_\AB)$). On the other hand, the quadratic form domains
depicted in \eqref{5.1}--\eqref{5.3} guarantee that
\begin{equation}
\ol{(H_\ABp - z I_{(a,b)})^{1/2}
(H_\AB - z I_{(a,b)})^{-1} (H_\ABp - z I_{(a,b)})^{1/2}} \in
\cB\big(L^2((a,b); rdx)\big), 
\end{equation}
and a detailed analysis reveals (cf.\ \cite[Sect.\ 4]{GZ11}) that the latter is,
in fact, at most a rank-two
perturbation of the identity $I_{(a,b)}$ in $L^2((a,b); rdx)$. For a
discussion of symmetrized perturbation determinants in an abstract setting, including the case of non-self-adjoint operators, we refer to the detailed
treatment in \cite{GZ11}.
\end{remark}
%%%%%%%%%%

Next, we derive the trace formula for the resolvent difference of
$H_\AB$ and $H_\ABp$ in terms of the spectral shift function
$\xi(\,\cdot\,; H_\ABp, H_\AB)$ and establish the connection between
$\La_\AB^\ABp(\cdot)$ and $\xi(\,\cdot\,; H_\ABp, H_\AB)$.

To prepare the ground for the basic trace formula we now state the following
fact:

%%%%%%%%%%%%%%%%%%%%%%%%%%%%%%%%%%%%%%%%
\begin{lemma} \lb{l5.2}
Assume that $\AB\in\bbC^{2\times2}$, where $\ABp\in\bbC^{2\times2}$ satisfy
\eqref{A.7}, and let the self-adjoint extensions $H_\AB$ and $H_\ABp$ be defined as in \eqref{4.8b}.
Then, with $\La_\AB^\ABp$ given by \eqref{4.47},
\begin{align}
\f{d}{dz} \La_\AB^\ABp(z) &= \ga_\ABp (H_\AB - z I_{(a,b)})^{-1} \big[\ga_\AB^\perp
(H_\AB - \bz I_{(a,b)})^{-1}\big]^*, \quad  z\in\rho(H_\AB).    \lb{5.22}
\end{align}
\end{lemma}
%%%%%%%%%%%%%%%%%%%%%%%%%%%%%%%%%%%%%%%%
\begin{proof}
Employing the resolvent equation for $H_\AB$, one verifies that
\begin{align}
& \f{d}{dz} \ga_\ABp \big[\ga_\AB^\perp
(H_\AB - \bz I_{(a,b)})^{-1}\big]^* = \ga_\ABp
\big[\ga_\AB^\perp (H_\AB - \bz I_{(a,b)})^{-2}\big]^*   \no \\
& \quad = \ga_\ABp (H_\AB - z I_{(a,b)})^{-1}
\big[\ga_\AB^\perp (H_\AB - \bz I_{(a,b)})^{-1}\big]^*.   \lb{5.23}
\end{align}
Together with \eqref{4.47} this proves \eqref{5.22}.
\end{proof}
%%%%%%%%%%%%%%%%%%%%%%%%%%%%%%%%%%%%%%%%

Combining Theorems \ref{t4.5} and \ref{t5.1} with Lemma \ref{l5.2} then yields the following result:

%%%%%%%%%%%%%%%%%%%%%%%%%%%%%%%%%%%%%%%%
\begin{theorem} \lb{t5.3}
Assume that $\AB\in\bbC^{2\times2}$, where $\ABp\in\bbC^{2\times2}$ satisfy
\eqref{A.7}, and let the self-adjoint extensions $H_\AB$ and $H_\ABp$ be defined as in \eqref{4.8b}. Then,
\begin{align}
&  \tr_{L^2((a,b); rdx)}\big((H_\ABp -z I_{(a,b)})^{-1}
- (H_\AB -z I_{(a,b)})^{-1}\big)    \no \\
& \quad = - \tr_{\bbC^2}\bigg(\Big[\La_\AB^\ABp (z)\Big]^{-1}
\f{d}{dz} \Big[\La_\AB^\ABp (z)\Big]\bigg)    \no \\
& \quad = - \f{d}{dz} \ln\Big({\det}_{\bbC^2}\Big(\La_\AB^\ABp (z)\Big)\Big), \quad
z \in \bbC\backslash[e_0,\infty).   \lb{5.24}
\end{align}
If, in addition, both $N_{A,B}$ and $N_{A',B'}$, defined as in \eqref{A.30}, are invertible, then
\begin{align}
&  \tr_{L^2((a,b); rdx)}\big((H_\ABp -z I_{(a,b)})^{-1}
- (H_\AB -z I_{(a,b)})^{-1}\big)   \no \\
& \quad = - \f{d}{dz} \ln\Big({\det}_{L^2((a,b); rdx)}\Big(\big\{(H_\ABp - z I_{(a,b)})^{1/2}
(H_\AB - z I_{(a,b)})^{-1}      \lb{5.25} \\
& \hspace*{4.3cm} \times (H_\ABp - z I_{(a,b)})^{1/2}\big\}^{\rm cl}\Big)\Big),
\quad z \in \bbC\backslash[e_0,\infty).    \no
\end{align}
\end{theorem}
%%%%%%%%%%%%%%%%%%%%%%%%%%%%%%%%%%%%%%%%
\begin{proof}
The second equality in \eqref{5.24} is obvious. The first equality in \eqref{5.24} follows upon rewriting \eqref{4.149}, with the help of \eqref{4.8c} and \eqref{4.8i}, as
\begin{align}
&(H_\ABp - z I_{(a,b)})^{-1} - (H_\AB - z I_{(a,b)})^{-1} = \no
\\
&\quad - \big[\ga_\AB^\perp(H_\AB - \bz I_{(a,b)})^{-1}\big]^* \La_\AB^\ABp(z)^{-1} \big[\ga_\ABp(H_\AB - z I_{(a,b)})^{-1}\big].
\end{align}
taking the trace, using cyclicity of the trace, and applying \eqref{5.22}. Then \eqref{5.25} follows from \eqref{5.12} and \eqref{5.24}.
\end{proof}
%%%%%%%%%%%%%%%%%%%%%%%%%%%%%%%%%%%%%%%%

In particular, in the non-degenerate case, where $N_{A,B}$ and $N_{A',B'}$ are invertible, the
determinant of $\La_\AB^\ABp (\cdot)$ coincides with the symmetrized perturbation determinant under the logarithm in \eqref{5.25} up to a spectral parameter independent constant (the latter depends on the boundary conditions involved).

Next, we note that the rank-two behavior of the difference of resolvents of
$H_\ABp$ and $H_\AB$ permits one to define the spectral shift function
$\xi(\,\cdot\,; H_\ABp, H_\AB)$ associated with the pair of self-adjoint operators
$(H_\ABp, H_\AB)$ in a standard manner. Moreover, using the typical normalization
in the context of self-adjoint operators bounded from below,
\begin{equation}
\xi(\, \cdot \,; H_\ABp, H_\AB) = 0, \quad
\lambda < e_0 =\inf\big(\sigma(H_\AB) \cup \sigma( H_\ABp)\big),
\lb{5.27}
\end{equation}
Krein's trace formula (see, e.g., \cite[Ch.\ 8]{Ya92}, \cite{Ya07}) reads
\begin{align}
\begin{split}
& \tr_{L^2((a,b); rdx)}\big((H_\ABp -z I_{(a,b)})^{-1}
- (H_\AB -z I_{(a,b)})^{-1}\big)   \\
& \quad = - \int_{[e_0,\infty)}
\f{\xi(\lambda; H_\ABp, H_\AB) \, d\lambda}{(\lambda - z)^2},
\quad z\in\rho(H_\AB)\cap \rho(H_\ABp),    \lb{5.28}
\end{split}
\end{align}
where $\xi(\cdot \, ; H_\ABp, H_\AB)$ satisfies
\begin{equation}
\xi(\cdot \, ; H_\ABp, H_\AB)
\in L^1\big(\bbR; (\lambda^2 + 1)^{-1} d\lambda\big).
\end{equation}
Since the spectra of $H_\AB$ and $H_\ABp$ are purely discrete,
$\xi(\, \cdot \,; H_\ABp, H_\AB)$ is an integer-valued piecewise constant function
on $\bbR$ with jumps precisely at the eigenvalues of $H_\AB$ and $H_\ABp$. In particular,
$\xi(\, \cdot \,; H_\ABp, H_\AB)$ represents the difference of the eigenvalue counting functions
of $H_\ABp$ and $H_\AB$.

Moreover, $\xi(\cdot \, ; H_\ABp, H_\AB)$ permits a representation in terms of nontangential boundary values to the real axis of $\det\big(\La_\AB^\ABp (\cdot)\big)$ (resp., of the symmetrized perturbation determinant \eqref{5.8}), to be described next.

%%%%%%%%%%%%%%%%%%%%%%%%%%%%%%%%%%%%%%%%
\begin{theorem} \lb{t5.4}
Assume that $A,B\in\bbC^{2\times2}$, where $\ABp\in\bbC^{2\times2}$ satisfy
\eqref{A.7}, and let the self-adjoint extensions $H_\AB$ and $H_\ABp$ be defined as in \eqref{4.8b}. Then,
\begin{align}
\begin{split}
\xi(\lambda; H_\ABp, H_\AB) =
\pi^{-1} \lim_{\varepsilon \downarrow 0} \Im\Big(
\ln\Big(\eta_{A',B',A,B} \,
{\det}_{\bbC^2}\Big(\La_\AB^\ABp (\lambda + i \varepsilon)\Big)\Big)\Big)&    \lb{5.30} \\
\text{ for a.e.\ $\lambda \in\bbR$,}&
\end{split}
\end{align}
where $\eta_{A',B',A,B}=e^{i\te_{A',B',A,B}}$ for some $\te_{A',B',A,B} \in[0,2\pi)$.
\end{theorem}
%%%%%%%%%%%%%%%%%%%%%%%%%%%%%%%%%%%%%%%%
\begin{proof}
We recall  the definition of
$e_0 =\inf\big(\sigma(H_\AB) \cup \sigma( H_\ABp)\big)$ in \eqref{5.27}.

Combining \eqref{5.24} and \eqref{5.28} one obtains
\begin{align}
\begin{split}
\f{d}{dz} \ln\Big(\eta_{A',B',A,B} \, {\det}_{\bbC^2}\Big(\La_\AB^\ABp (z)\Big)\Big)
= \int_{[e_0,\infty)}
\f{\xi(\lambda; H_\ABp, H_\AB) \, d\lambda}{(\lambda - z)^2},&  \\
z\in\rho(H_\AB)\cap \rho(H_\ABp),&      \lb{5.31}
\end{split}
\end{align}
where $\eta_{A',B',A,B}$ is some $z$-independent constant.

Assuming temporarily that $S_{A',B',A,B}$ is invertible, we note that by \eqref{4.52},
\begin{align}
{\det}_{\bbC^2}\Big(\La_\AB^\ABp (z)\Big) {\det}_{\bbC^2}\big(S^*_{A',B',A,B}) \in\bbR,
\quad z\in\bbR\bs\si(H_\AB),
\end{align}
and since
\begin{align}
{\det}_{\bbC^2} \Big(\La_\AB^\ABp (z)\Big)\neq0, \quad z<e_0,
\end{align}
it follows that there is a unique $\eta_{A',B',A,B} = e^{i\te_{A',B',A,B}}$,
$\te_{A',B',A,B}\in[0,2\pi)$ such that
\begin{equation}
\eta_{A',B',A,B} \, {\det}_{\bbC^2}\Big(\La_\AB^\ABp (z)\Big) > 0, \quad z<e_0.
\lb{5.33}
\end{equation}
In the case $S_{A',B',A,B}$ is not invertible, one considers a slightly perturbed boundary trace $\ga_{\ABp;\de} = \ga_\ABp + \de \, T_{A',B',A,B} \, \ga_\AB^\perp$. Then the corresponding perturbed boundary data map converges to the unperturbed one
$\La_\AB^{\ABp;\de}(z)\to\La_\AB^\ABp(z)$ as $\de\to0$ and
\begin{align}
{\det}_{\bbC^2} \Big(\La_\AB^{\ABp;\de}(z)\Big)
{\det}_{\bbC^2}(S^*_{A',B',A,B} + \de \, T^*_{A',B',A,B})
\in \bbR, \quad z\in\bbR\bs\si(H_\AB), \; \de\in\bbR. \lb{5.33a}
\end{align}
As discussed around \eqref{4.8p},
${\det}_{\bbC^2}(S^*_{A',B',A,B} + \de \, T^*_{A',B',A,B})\neq0$ for all sufficiently
small $\de\neq0$, hence utilizing the identity
\begin{align}
\begin{split}
& {\det}_{\bbC^2}(S^*_{A',B',A,B} + \de \, T^*_{A',B',A,B})
= \de \, {\det}_{\bbC^2}(S^*_{A',B',A,B,1} \ \ T^*_{A',B',A,B,2})    \\
& \quad +
\de \, {\det}_{\bbC^2}(T^*_{A',B',A,B,1} \ \ S^*_{A',B',A,B,2})
+ \de^2 {\det}_{\bbC^2}(T^*_{A',B',A,B}), \lb{5.33b}
\end{split}
\end{align}
where the notation $Z_j$ is used to denote the $j$-th column of a matrix $Z$, substituting \eqref{5.33b} into \eqref{5.33a}, dividing by $\de$, taking $\de\to0$, and invoking the continuity of $\La_\AB^\ABp$ with respect to the parameter matrices $\ABp$ yields either
\begin{align}
\begin{split}
{\det}_{\bbC^2}\Big(\La_\AB^\ABp (z)\Big)
\big[& {\det}_{\bbC^2}(S^*_{A',B',A,B,1} \ \ T^*_{A',B',A,B,2})    \\
& + {\det}_{\bbC^2}(T^*_{A',B',A,B,1} \ \ S^*_{A',B',A,B,2})\big] \in \bbR\bs\{0\},
\quad z<e_0,
\end{split}
\end{align}
or
\begin{align}
{\det}_{\bbC^2} \Big(\La_\AB^\ABp (z)\Big)
{\det}_{\bbC^2} (T^*_{A',B',A,B}) \in \bbR\bs\{0\},
\quad z<e_0.
\end{align}
Thus, \eqref{5.33} holds in the case of a noninvertible matrix
$S_{A',B',A,B}$ as well.

Next, integrating \eqref{5.31} with respect to the $z$-variable along the real axis from
$z_0$ to $z$, assuming $z<z_0<e_0$, one obtains
\begin{align}
& \ln\Big(\eta_{A',B',A,B} \, {\det}_{\bbC^2}\Big(\La_\AB^\ABp (z)\Big)\Big) -
\ln\Big(\eta_{A',B',A,B} \, {\det}_{\bbC^2} \Big(\La_\AB^\ABp (z_0)\Big)\Big)   \no \\
& \quad = \int_{z_0}^z d\zeta
\int_{[e_0,\infty)} \f{\xi(\lambda; H_\ABp, H_\AB) \, d\lambda}
{(\lambda - \zeta)^2}  \no \\
& \quad = \int_{z_0}^z d\zeta
\int_{[e_0,\infty)} \f{[\xi_+ (\lambda; H_\ABp, H_\AB)
- \xi_- (\lambda; H_\ABp, H_\AB)] \, d\lambda}
{(\lambda - \zeta)^2}   \no \\
& \quad = \int_{[e_0,\infty)} [\xi_+ (\lambda; H_\ABp, H_\AB)
- \xi_- (\lambda; H_\ABp, H_\AB)] \, d\lambda
\int_{z_0}^z \f{d\zeta}{(\lambda - \zeta)^2}   \no \\
& \quad = \int_{[e_0,\infty)} \xi(\lambda; H_\ABp, H_\AB) \, d\lambda
\bigg(\f{1}{\lambda - z} - \f{1}{\lambda - z_0}\bigg), \quad z<z_0<e_0.   \lb{5.34}
\end{align}
Here we split $\xi$ into its positive and negative parts, $\xi_{\pm} = [|\xi| \pm \xi]/2$,
and applied the Fubini--Tonelli theorem to interchange the integrations with respect
to $\lambda$ and $\zeta$. Moreover, we chose the branch of $\ln(\cdot)$ such
that $\ln(x) \in \bbR$ for $x>0$, compatible with the normalization of
$\xi(\, \cdot \,; H_\ABp, H_\AB)$ in \eqref{5.27}.

An analytic continuation of the first and last line of \eqref{5.34} with respect to $z$ then
yields
\begin{align}
& \ln\Big(\eta_{A',B',A,B} \, {\det}_{\bbC^2} \Big(\La_\AB^\ABp (z)\Big)\Big) -
\ln\Big(\eta_{A',B',A,B} \, {\det}_{\bbC^2} \Big(\La_\AB^\ABp (z_0)\Big)\Big)   \no \\
& \quad = \int_{[e_0,\infty)} \xi(\lambda; H_\ABp, H_\AB) \, d\lambda
\bigg(\f{1}{\lambda - z} - \f{1}{\lambda - z_0}\bigg),
\quad z \in \bbC \backslash [e_0,\infty).
\end{align}

Since by \eqref{5.33},
\begin{equation}
\ln\Big(\eta_{A',B',A,B} \, {\det}_{\bbC^2} \Big(\La_\AB^\ABp (z_0)\Big)\Big) \in \bbR,
\quad z_0 < e_0,
\end{equation}
the Stieltjes inversion formula separately applied to the absolutely continuous
measures $\xi_{\pm} (\lambda; H_\ABp, H_\AB) \, d\lambda$
(cf., e.g., \cite[p.\ 328]{AD56}, \cite[App.\ B]{We80}), then yields \eqref{5.30}.
\end{proof}
%%%%%%%%%%%%%%%%%%%%%%%%%%%%%%%%%%%%%%%%

%%%%%%%%%%%%%%%%%%%%%%%%%%%%%%%%%%%%%%%%
%%%%%%%%%%%%%%%%%%%%%%%%%%%%%%%%%%%%%%%%
\section{Connecting Von Neumann's Parametrization of All Self-Adjoint Extensions 
of $H_{\min}$ and the Boundary Data Map $\La_\ABp^\AB(\cdot)$}
\lb{s7}
%%%%%%%%%%%%%%%%%%%%%%%%%%%%%%%%%%%%%%%%
%%%%%%%%%%%%%%%%%%%%%%%%%%%%%%%%%%%%%%%%

In this section, we turn to the precise connection between the canonical von Neumann parametrization 
of all self-adjoint extensions of $H_{\min}$ in terms of unitary operators mapping between the associated deficiency subspaces and the boundary data map $\La_\ABp^\AB(\cdot)$. 

According to von Neumann's theory \cite{Ne30}, the self-adjoint extensions of a densely defined closed symmetric operator $T_0:\dom(T_0)\rightarrow \cH$, $\overline{\dom(T_0)}=\cH$, with equal deficiency indices $n_{\pm}$ are in one-to-one correspondence with the set of linear isometric isomorphisms (i.e., unitary maps) from $\cN_+$ to $\cN_-$, where
\begin{equation}\lb{A.52}
\cN_{\pm}=\ker(T_0^*\mp iI_{\cH}),  \quad  n_{\pm} = \dim(\cN_{\pm}).
\end{equation}
We summarize some of the basic facts of the theory in the following theorem.

%%%%%%%%%%%%%
\begin{theorem}\lb{tA.11}
Let $T_0:\dom(T_0)\rightarrow \cH$, $\overline{\dom(T_0)}=\cH$, denote a symmetric operator with equal deficiency indices $n_+ = n_-$ and $\cN_{\pm}$ as defined in \eqref{A.52}.  Then the following items $(i)$--$(iii)$ hold.\\
$(i)$  The domain of $T_0^*$ is given by
\begin{equation}\lb{A.53}
\dom(T_0^*)=\dom(T_0)\dotplus \cN_+\dotplus \cN_-,
\end{equation}
where $\dotplus$ indicates the direct $($but not necessarily orthogonal\,$)$ sum of subspaces.\\
$(ii)$ For a linear isometric isomorphism $\cU: \cN_+\rightarrow \cN_-$, define the linear operator $T_{\cU}:\dom(T_{\cU})\rightarrow \cH$ by
\begin{align}
T_{\cU}=T_0^*|_{\dom(T_{\cU})}, \quad \dom(T_{\cU})=\dom(T_0)\dotplus \cN_+ \dotplus \cU\cN_+.\lb{A.54}
\end{align}
The mapping $\cU\mapsto T_{\cU}$ is a bijection from the set of linear isometric isomorphisms $\cU:\cN_+\rightarrow \cN_-$ and the set of self-adjoint extensions of $T_0$.\\
$(iii)$  If $T$ is a self-adjoint extension of $T_0$ and
\begin{equation}\lb{A.55}
\cC_T=(T+iI_{\cH})(T-iI_{\cH})^{-1}
\end{equation}
denotes the unitary Cayley transform of $T$, then $T=T_{\cU}$ with
\begin{equation}\lb{A.56}
\cU=-\cC_T^{-1}|_{\cN_+}.
\end{equation}
\end{theorem}
%%%%%%%%%%%%%

Items $(i)$ and $(ii)$ in Theorem \ref{tA.11} are standard results in the theory of self-adjoint
extensions of symmetric operators and may be found, for example, in \cite[\S80]{AG81},
\cite[\S14.4 \& \S14.8]{Na68}, and \cite[\S8.2]{We80}.  Item $(iii)$ in Theorem \ref{tA.11} is taken
from \cite{GMT98}.

Our next result establishes a connection between von Neumann's isometric 
isomorphism $\cU$ in $T_{\cU}$, the boundary trace of bases in $\ker(H_{\max} \mp i I_{(a,b)})$, 
and the boundary data map $\La_\ABp^\AB(\cdot)$. To the best of our knowledge, this  
appears to be new.

%%%%%%%%%%%%%
\begin{theorem}\lb{tA.12a}
Suppose that $\cB_{\pm}=\{u_\pm,v_\pm\}$ denote ordered bases for
\begin{equation}\lb{A.57a}
\cN_{\pm}=\ker(H_{\max}\mp iI_{(a,b)}).
\end{equation}
For $\AB\in\bbC^{2\times2}$ satisfying \eqref{A.7}, assume that $\cU_\AB:\cN_+\rightarrow \cN_-$ denotes the unique linear isometric isomorphism with
\begin{equation}\lb{A.58a}
\dom(H_\AB)=\dom(H_{\min})\dotplus \cN_+\dotplus \cU_\AB\cN_+,
\end{equation}
guaranteed to exist by Theorem \ref{tA.11}\,$(ii)$. Suppose that  $\big[\cU_\AB\big]$ denotes the matrix representation of $\cU_\AB$ with respect to the bases $\cB_{\pm}$.  Then
\begin{equation}\lb{A.60a}
\big[\cU_\AB\big]= - \begin{pmatrix} \ga_\AB(u_-) & \ga_\AB(v_-)\end{pmatrix}^{-1} 
\begin{pmatrix} \ga_\AB(u_+) & \ga_\AB(v_+) \end{pmatrix},
\end{equation}
where the boundary trace map $\ga_\AB$ is given by \eqref{A.26}.

In particular, if $\ABp\in\bbC^{2\times2}$ is another pair for which \eqref{A.7} holds and the bases $\cB_\pm=\{u_\pm,v_\pm\}$ consist of functions satisfying the boundary conditions
\begin{align} \lb{A.61a}
\ga_\ABp(u_\pm) = \begin{pmatrix} 1\\0\end{pmatrix}, \quad
\ga_\ABp(v_\pm) = \begin{pmatrix} 0\\1\end{pmatrix},
\end{align}
then \eqref{A.60a} becomes
\begin{equation}\lb{A.62a}
\big[\cU_\AB\big]= -\La_\ABp^\AB(-i)^{-1}\La_\ABp^\AB(i),
\end{equation}
where the boundary data map $\La_\ABp^\AB(\cdot)$ is given by \eqref{4.17}.
\end{theorem}
%%%%%%%%%%%%%
\begin{proof}
Suppose $\big[\wti\cU_\AB\big]\in\bbC^{2\times 2}$ denotes the right-hand side of \eqref{A.60a} and define $\wti\cU_\AB$ to be the linear map from $\cN_+$ to $\cN_-$ with the matrix representation $\big[\wti\cU_\AB\big]$ in the bases $\cB_\pm$. That is, for $f\in\cN_+$ let $c_1,c_2\in\bbC$ be such that
\begin{align}\lb{A.63a}
f= \begin{pmatrix} u_+ & v_+ \end{pmatrix} \begin{pmatrix}c_1\\c_2\end{pmatrix},
\end{align}
then
\begin{align}\lb{A.64a}
\wti\cU_\AB f = \begin{pmatrix} u_- & v_-\end{pmatrix} \, \big[\wti\cU_\AB\big] \begin{pmatrix}c_1\\c_2\end{pmatrix}.
\end{align}
It follows from \eqref{A.60a}, \eqref{A.63a}, \eqref{A.64a}, and the linearity of $\ga_\AB$ that
\begin{align}
\begin{split}
&\ga_\AB(f+\wti\cU_\AB f) \\
&\quad
= \begin{pmatrix} \ga_\AB(u_+) & \ga_\AB(v_+) \end{pmatrix} \begin{pmatrix}c_1\\c_2\end{pmatrix}
+ \begin{pmatrix} \ga_\AB(u_-) & \ga_\AB(v_-)\end{pmatrix} 
\big[\wti\cU_\AB\big] \begin{pmatrix}c_1\\c_2\end{pmatrix} \\
&\quad
= \begin{pmatrix} \ga_\AB(u_+) & \ga_\AB(v_+) \end{pmatrix} \begin{pmatrix}c_1\\c_2\end{pmatrix}
- \begin{pmatrix} \ga_\AB(u_+) & \ga_\AB(v_+) \end{pmatrix} \begin{pmatrix}c_1\\c_2\end{pmatrix} \\
&\quad
= 0.
\end{split}
\end{align}
Thus, by \eqref{A.26b}, $f+\wti\cU_\AB f \in \dom(H_\AB)$ for every $f\in\cN_+$. By Theorem \ref{tA.11}\,$(ii)$, also $f+\cU_\AB f\in \dom(H_\AB)$ for all $f\in\cN_+$. Hence, 
\begin{equation} 
(\cU_\AB-\wti\cU_\AB)f = (f+\cU_\AB f)-(f+\wti\cU_\AB f) \in \dom(H_\AB). 
\end{equation} 
Since both $\cU_\AB$ and $\wti\cU_\AB$ map into $\cN_-$ it follows that 
\begin{equation} 
(\cU_\AB-\wti\cU_\AB)f\in\cN_-\cap\dom(H_\AB)=\{0\}, \quad f\in\cN_+, 
\end{equation} 
that is, $\cU_\AB=\wti\cU_\AB$.
\end{proof}
%%%%%%%%%%%%%

Our final result in this section provides an explicit matrix representation of von Neumann's isometric isomorphism $\cU$ in $T_{\cU}$ in terms of a particular basis of solutions in 
$\ker(H_{\max} \mp i I_{(a,b)})$.

%%%%%%%%%%%%%
\begin{theorem}\lb{tA.12}
Suppose that $\cB_{\pm}=\{u_1(\pm i,\cdot),u_2(\pm i,\cdot)\}$ denote the ordered bases for
\begin{equation}\lb{A.57}
\cN_{\pm}=\ker(H_{\max}\mp iI_{(a,b)}),
\end{equation}
with $u_j(\pm i,\cdot)$, $j=1,2$, given by \eqref{3.8}.  For $\theta_a,\theta_b\in [0,\pi)$, assume 
that $\cU_{\theta_a,\theta_b}:\cN_+\rightarrow \cN_-$ denotes the unique linear isometric isomorphism 
with
\begin{equation}\lb{A.58}
 \dom(H_{\theta_a,\theta_b})=\dom(H_{\min})\dotplus \cN_+\dotplus \cU_{\theta_a,\theta_b}\cN_+,
 \end{equation}
 guaranteed to exist by Theorem \ref{tA.11}\,$(ii)$.
For $F\in \SL_2(\bbR)$, $\phi\in [0,2\pi)$, let $\cU_{F,\phi}$ denote the unique linear isometric isomorphism with
\begin{equation}\lb{A.59}
\dom(H_{F,\phi})=\dom(H_{\min})\dotplus \cN_+\dotplus \cU_{F,\phi}\cN_+,
\end{equation}
guaranteed to exist by Theorem \ref{tA.11}\,$(ii)$.  Let $\big[\cU_{\theta_a,\theta_b} \big]$ and $\big[\cU_{F,\phi} \big]$ denote the matrix representations $\cU_{\theta_a,\theta_b}$ and $\cU_{F,\phi}$ with respect to the bases $\cB_{\pm}$.  Then the following items $(i)$--$(vi)$ hold.\\
$(i)$ If $\theta_a\neq 0$ and $\theta_b\neq 0$, then
\begin{equation}\lb{A.60}
\big[\cU_{\theta_a,\theta_b} \big]= -D_{\theta_a,\theta_b}(-i)^{-1}D_{\theta_a,\theta_b}(i),
\end{equation}
where $D_{\theta_a,\theta_b}(\cdot)$ is given by \eqref{3.16}.\\
$(ii)$ If $\theta_a\neq 0$ and $\theta_b=0$, then
\begin{equation}\lb{A.61}
\big[\cU_{\theta_a,0} \big]= \begin{pmatrix}
-1 & 0\\
-d_{\theta_a,0}(-i)^{-1}\Big[u_2^{[1]}(-i,b)+u_1^{[1]}(i,a)\Big] & -d_{\theta_a,0}(-i)^{-1}d_{\theta_a,0}(i)
\end{pmatrix},
\end{equation}
where $d_{\theta_a,0}(\cdot)$ is given by \eqref{3.19}.\\
$(iii)$  If $\theta_a=0$ and $\theta_b\neq 0$, then
\begin{equation}\lb{A.62}
\big[\cU_{0,\theta_b} \big]= \begin{pmatrix}
-d_{0,\theta_b}(-i)^{-1}d_{0,\theta_b}(i) & d_{0,\theta_b}(-i)^{-1}\Big[u_2^{[1]}(i,b)+u_1^{[1]}(-i,a)\Big]\\
0 & -1
\end{pmatrix},
\end{equation}
where $d_{0,\theta_b}(\cdot)$ is given by \eqref{3.22}.\\
$(iv)$  If $\theta_a=\theta_b=0$, then
\begin{equation}\lb{A.63}
\big[\cU_{0,0}\big]=-I_2.
\end{equation}
$(v)$  If $F_{1,2}\neq 0$, then
\begin{equation}\lb{A.64}
\big[\cU_{F,\phi} \big] = -Q_{F,\phi}(-i)^{-1}Q_{F,\phi}(i),
\end{equation}
where $Q_{F,\phi}(\cdot)$ is given by \eqref{3.53}.\\
$(vi)$  If $F_{1,2}=0$, then
\begin{equation}\lb{A.65}
\big[\cU_{F,\phi}\big]=q_{F,\phi}(-i)^{-1}\begin{pmatrix} \widetilde{c}_{1,1}(F,\phi) & \widetilde{c}_{2,1}(F,\phi)\\
\widetilde{c}_{1,2}(F,\phi) & \widetilde{c}_{2,2}(F,\phi)\end{pmatrix}-I_2,
\end{equation}
where $q_{F,\phi}(\cdot)$ is given by \eqref{3.56} and
\begin{align}
&\widetilde{c}_{1,1}(F,\phi)=u_1^{[1]}(i,b)-u_1^{[1]}(-i,b)-e^{i\phi}F_{2,2}\Big[u_2^{[1]}(-i,b)
+u_1^{[1]}(i,a)\Big],       \lb{A.65a}\\
&\widetilde{c}_{1,2}(F,\phi)=F_{2,2}\Big\{e^{-i\phi}\Big[u_1^{[1]}(i,b)-u_1^{[1]}(-i,b) \Big]
-F_{2,2}\Big[u_2^{[1]}(-i,b)+u_1^{[1]}(i,a) \Big]  \Big\},     \lb{A.65b}\\
&\widetilde{c}_{2,2}(F,\phi)=F_{2,2}\Big\{F_{2,2}\Big[ u_2^{[1]}(-i,a)-u_2^{[1]}(i,a) \Big]
+e^{-i\phi}\Big[ u_2^{[1]}(i,b)+u_1^{[1]}(-i,a)\Big]\Big\},      \lb{A.65c}\\
&\widetilde{c}_{2,1}(F,\phi)=e^{i\phi}F_{2,2}\Big[u_2^{[1]}(-i,a)-u_2^{[1]}(i,a) \Big]+u_2^{[1]}(i,b)
+u_1^{[1]}(-i,a).      \lb{A.65d}
\end{align}
\end{theorem}
%%%%%%%%%%%%%

%%%%%%%%%%%%%
\begin{proof}
We begin with a few general observations in order to set the stage for the proofs of items $(i)$--$(iv)$.  Since $\cU_{\theta_a,\theta_b}$, $\theta_a,\theta_b\in[0,\pi)$, maps $\cN_+$ onto $\cN_-$ and $\cB_-$ is a basis for $\cN_-$,
\begin{equation}\lb{A.66}
\cU_{\theta_a,\theta_b}u_{\ell}(i,\cdot)=c_{\ell,1}(\theta_a,\theta_b)u_1(-i,\cdot)+c_{\ell,2}(\theta_a,\theta_b)u_2(-i,\cdot), \quad \ell=1,2,
\end{equation}
for suitable scalars $\{c_{\ell,k}(\theta_a,\theta_b)\}_{1\leq \ell,k\leq 2}$.  Then by definition, the matrix representation of $\cU_{\theta_a,\theta_b}$ with respect to the bases $\cB_{\pm}$ is given by
\begin{equation}\lb{A.67}
\big[\cU_{\theta_a,\theta_b}\big]=\begin{pmatrix} c_{1,1}(\theta_a,\theta_b) & c_{2,1}(\theta_a,\theta_b)\\
c_{1,2}(\theta_a,\theta_b) & c_{2,2}(\theta_a,\theta_b)\end{pmatrix}.
\end{equation}
By Theorem \ref{tA.11}\,$(iii)$,
\begin{equation}\lb{A.68}
\cU_{\theta_a,\theta_b}=-(H_{\theta_a,\theta_b}-iI_{(a,b)})(H_{\theta_a,\theta_b}+iI_{(a,b)})^{-1}|_{\cN_+},
\end{equation}
and as a result,
\begin{align}
\cU_{\theta_a,\theta_b}u_{\ell}(i,\cdot)&=-(H_{\theta_a,\theta_b}-iI_{(a,b)})(H_{\theta_a,\theta_b}+iI_{(a,b)})^{-1}u_{\ell}(i,\cdot)\no\\
&=-(H_{\theta_a,\theta_b}+(i-2i)I_{(a,b)})(H_{\theta_a,\theta_b}+iI_{(a,b)})^{-1}u_{\ell}(i,\cdot)\no\\
&=2i(H_{\theta_a,\theta_b}+iI_{(a,b)})^{-1}u_{\ell}(i,\cdot)-u_{\ell}(i,\cdot), \quad \ell=1,2.\lb{A.69}
\end{align}
{\it Proof of item $(i)$:} Applying Krein's formula \eqref{3.17} with $z=-i$ to the resolvent in \eqref{A.69}, relation \eqref{A.66} can be recast as
\begin{align}
&c_{\ell,1}(\theta_a,\theta_b)u_1(-i,\cdot)+c_{\ell,2}(\theta_a,\theta_b)u_2(-i,\cdot)=2i(H_{0,0}+iI_{(a,b)})^{-1}u_{\ell}(i,\cdot)-u_{\ell}(i,\cdot)\no\\
&\quad -2i\sum_{j,k=1}^2D_{\theta_a,\theta_b}(-i)_{j,k}^{-1}(u_k(i,\cdot),u_{\ell}(i,\cdot))_{L^2((a,b);rdx)}u_j(-i,\cdot), \quad \ell=1,2,\lb{A.70}
\end{align}
where $D_{\theta_a,\theta_b}(-i)$ is defined by \eqref{3.16}.
Taking $\ell=1$ in \eqref{A.70}, evaluating separately at $x=a$ and $x=b$, and using \eqref{3.8} along with
\begin{equation}\lb{A.71}
\big[(H_{0,0}-iI_{(a,b)})^{-1}u_1(i,\cdot)\big](a)=\big[(H_{0,0}-iI_{(a,b)})^{-1}u_1(i,\cdot)\big](b)=0,
\end{equation}
yields
\begin{align}
c_{1,1}(\theta_a,\theta_b)&=-2i\sum_{k=1}^2D_{\theta_a,\theta_b}(-i)_{1,k}^{-1}(u_k(i,\cdot),u_1(i,\cdot))_{L^2((a,b);rdx)}-1,     \lb{A.72}\\
c_{1,2}(\theta_a,\theta_b)&=-2i\sum_{k=1}^2D_{\theta_a,\theta_b}(-i)_{2,k}^{-1}(u_k(i,\cdot),u_1(i,\cdot))_{L^2((a,b);rdx)}.     \lb{A.72a}
\end{align}
On the other hand, taking $\ell=2$ in \eqref{A.70}, evaluating separately at $x=a$ and $x=b$, and using \eqref{3.8} along with
\begin{equation}\lb{A.73}
\big[(H_{0,0}-iI_{(a,b)})^{-1}u_2(i,\cdot)\big](a)=\big[(H_{0,0}-iI_{(a,b)})^{-1}u_2(i,\cdot)\big](b)=0,
\end{equation}
one concludes that
\begin{align}
c_{2,1}(\theta_a,\theta_b)&=-2i\sum_{k=1}^2D_{\theta_a,\theta_b}(-i)_{1,k}^{-1}(u_k(i,\cdot),u_2(i,\cdot))_{L^2((a,b);rdx)},     \lb{A.74}\\
c_{2,2}(\theta_a,\theta_b)&=-2i\sum_{k=1}^2D_{\theta_a,\theta_b}(-i)_{2,k}^{-1}(u_k(i,\cdot),u_2(i,\cdot))_{L^2((a,b);rdx)}-1.    \lb{A.74a}
\end{align}
Comparing \eqref{A.67} with \eqref{A.72}, \eqref{A.72a}, \eqref{A.74}, and \eqref{A.74a}, one infers
\begin{equation}\lb{A.75}
\big[\cU_{\theta_a,\theta_b}\big]=-2iD_{\theta_a,\theta_b}(-i)^{-1}\big[(u_j(i,\cdot),u_k(i,\cdot))_{L^2((a,b);rdx)}\big]_{1\leq j,k\leq 2}-I_2,
\end{equation}
where $\big[(u_j(i,\cdot),u_k(i,\cdot))_{L^2((a,b);rdx)}\big]_{1\leq j,k\leq 2}$ is the Gram matrix corresponding to the basis $\cB_+$.  Taking \eqref{B.32} in the case at hand (i.e., with $P(\cdot)=-D_{\theta_a,\theta_b}(\cdot)^{-1}$) with $z=-i$ and $z'=i$, one obtains
\begin{equation}\lb{A.76}
\big[(u_j(i,\cdot),u_k(i,\cdot))_{L^2((a,b);rdx)}\big]_{1\leq j,k\leq 2}=(-2i)^{-1}\big[D_{\theta_a,\theta_b}(-i)-D_{\theta_a,\theta_b}(i)\big].
\end{equation}
Using \eqref{A.76} in \eqref{A.75}, one arrives at \eqref{A.60}.

\noindent
{\it Proof of item $(ii)$:}  Applying Krein's formula \eqref{3.20} with $z=-i$ to the resolvent in \eqref{A.69}, relation \eqref{A.66} (with $\theta_b=0$) can be recast as
\begin{align}
&c_{\ell,1}(\theta_a,0)u_1(-i,\cdot)+c_{\ell,2}(\theta_a,0)u_2(-i,\cdot)=2i(H_{0,0}+i)^{-1}u_{\ell}(i,\cdot)-u_{\ell}(i,\cdot)\no\\
&\quad -2id_{\theta_a,0}(-i)^{-1}(u_2(i,\cdot),u_1(i,\cdot))_{L^2((a,b);rdx)}u_2(-i,\cdot), \quad \ell=1,2,\lb{A.77}
\end{align}
where $d_{\theta_a,0}(-i)$ is defined by \eqref{3.19}.
Taking $\ell=1$ in \eqref{A.77}, evaluating separately at $x=a$ and $x=b$, using \eqref{3.8} and \eqref{A.71}, yields
\begin{align}
c_{1,1}(\theta_a,0)&= -1,   \lb{A.78}\\
c_{1,2}(\theta_a,0)&= -2id_{\theta_a,0}(-i)^{-1}(u_2(i,\cdot),u_1(i,\cdot))_{L^2((a,b);rdx)}.   \lb{A.79}
\end{align}
Similarly, taking $\ell=2$ in \eqref{A.77}, evaluating separately at $x=a$ and $x=b$, using \eqref{3.8} and \eqref{A.73}, implies
\begin{align}
c_{2,1}(\theta_a,0)&=0,     \lb{A.80}\\
c_{2,2}(\theta_a,0)&=-2id_{\theta_a,0}(-i)^{-1}(u_2(i,\cdot),u_2(i,\cdot))_{L^2((a,b);rdx)}-1.    \lb{A.81}
\end{align}
The inner products in \eqref{A.79} and \eqref{A.81} can be calculated explicitly.  In fact, all entries
of the Gram matrix $\big[(u_j(i,\cdot),u_k(i,\cdot))_{L^2((a,b);rdx)}\big]_{1\leq j,k\leq 2}$ can be explicitly computed.  To this end, one observes that
\begin{align}
&\frac{d}{dx}W\big(u_j(-i,\cdot),u_k(i,\cdot)\big)(x)      \lb{A.82} \\
&\quad= u_j(-i,x)\frac{d}{dx}u_k^{[1]}(i,x)-u_k(i,x)\frac{d}{dx}u_j^{[1]}(-i,x)
\, \text{ for a.e.\ $x\in (a,b)$}, \; 1\leq j,k\leq 2.    \no
\end{align}
On the other hand, by the very definition of $u_j(\pm i,\cdot)$, $j=1,2$, one has
\begin{align}
\frac{d}{dx}u_j^{[1]}(\pm i,x)=(q(x)\mp i r(x))u_j(\pm i,x) \, \text{ for a.e.\ $x\in (a,b)$},
\; j=1,2.    \lb{A.83}
\end{align}
Taking \eqref{A.82} together with \eqref{A.83}, and accounting for cancellations, one concludes that
\begin{align}
\begin{split}
&r(x)u_j(-i,x)u_k(i,x)=-\frac{1}{2i}\frac{d}{dx}W\big(u_j(-i,\cdot),u_k(i,\cdot)\big)(x)   \lb{A.84}\\
& \hspace*{3.6cm} \text{ for a.e.\ $x\in (a,b)$}, \; 1\leq j,k\leq 2.
\end{split}
\end{align}
With \eqref{A.84} in hand, the inner product of $u_j(i,\cdot)$ with $u_k(i,\cdot)$ can be explicitly computed:
\begin{align}
& (u_j(i,\cdot),u_k(i,\cdot))_{L^2((a,b);rdx)}=\int_a^b r(x)dx \, u_j(-i,x)u_k(i,x)    \no\\
&\quad =-\frac{1}{2i}\int_a^b dx \, \frac{d}{dx}W\big(u_j(-i,\cdot),u_k(i,\cdot) \big)(x)   \no\\
&\quad =-\frac{1}{2i}W\big(u_j(-i,\cdot),u_k(i,\cdot) \big)(x)\big|_a^b,\quad 1\leq j,k\leq 2.\lb{A.85}
\end{align}
Finally, \eqref{A.85} and \eqref{3.8} yield
\begin{align}
(u_1(i,\cdot),u_1(i,\cdot))_{L^2((a,b);rdx)}&=-\frac{1}{2i}\Big(u_1^{[1]}(i,b)-u_1^{[1]}(-i,b) \Big),  \lb{A.86}\\
(u_2(i,\cdot),u_2(i,\cdot))_{L^2((a,b);rdx)}&=-\frac{1}{2i}\Big(u_2^{[1]}(-i,a)-u_2^{[1]}(i,a)\Big),  \lb{A.87}\\
(u_1(i,\cdot),u_2(i,\cdot))_{L^2((a,b);rdx)}&=-\frac{1}{2i}\Big(u_2^{[1]}(i,b)+u_1^{[1]}(-i,a)\Big),  \lb{A.88}\\
(u_2(i,\cdot),u_1(i,\cdot))_{L^2((a,b);rdx)}&=\frac{1}{2i}\Big(u_2^{[1]}(-i,b)+u_1^{[1]}(i,a)\Big).  \lb{A.89}
\end{align}
Combining \eqref{3.19} with \eqref{A.87} in \eqref{A.81} implies
\begin{align}
c_{2,2}(\theta_a,0)=-d_{\theta_a,0}(-i)^{-1}d_{\theta_a,0}(i),\lb{A.90}
\end{align}
and taking \eqref{3.19} with \eqref{A.89} in \eqref{A.79} yields
\begin{align}
c_{1,2}(\theta_a,0)=-d_{\theta_a,0}(-i)^{-1}\Big(u_2^{[1]} (-i,b)+u_1^{[1]}(i,a)\Big).   \lb{A.91}
\end{align}
Finally, \eqref{A.61} follows from \eqref{A.78}, \eqref{A.80}, \eqref{A.90}, and \eqref{A.91}.

\noindent
{\it Proof of item $(iii)$:}  Applying Krein's formula \eqref{3.23} with $z=-i$ to the resolvent in \eqref{A.69}, relation \eqref{A.66} (with $\theta_a=0$) can be recast as
\begin{align}
&c_{\ell,1}(0,\theta_b)u_1(-i,\cdot)+c_{\ell,2}(0,\theta_b)u_2(-i,\cdot)=2i(H_{0,0}+i)^{-1}u_{\ell}(i,\cdot)-u_{\ell}(i,\cdot)\no\\
&\quad -2id_{0,\theta_b}(-i)^{-1}(u_1(i,\cdot),u_1(i,\cdot))_{L^2((a,b);rdx)}u_1(-i,\cdot),
\quad \ell=1,2,\lb{A.92}
\end{align}
where $d_{0,\theta_b}(-i)$ is defined by \eqref{3.22}.
Taking $\ell=1$ and evaluating \eqref{A.92} separately at $x=a$ and $x=b$ implies
\begin{align}
c_{1,1}(0,\theta_b)&=-2id_{0,\theta_b}(-i)^{-1}(u_1(i,\cdot),u_1(i,\cdot))_{L^2((a,b);rdx)}-1,   \lb{A.93}\\
c_{1,2}(0,\theta_b)&=0,    \lb{A.94}
\end{align}
and taking $\ell=2$, evaluating \eqref{A.92} separately at $x=a$ and $x=b$ yields
\begin{align}
c_{2,1}(0,\theta_b)&=-2id_{0,\theta_b}(-i)^{-1}(u_1(i,\cdot),u_2(i,\cdot))_{L^2((a,b);rdx)},    \lb{A.95}\\
c_{2,2}(0,\theta_b)&=-1.\lb{A.96}
\end{align}
Recalling \eqref{3.22} and \eqref{A.86} in \eqref{A.93} implies
\begin{align}
c_{1,1}(0,\theta_b)=-d_{0,\theta_b}(-i)^{-1}d_{0,\theta_b}(i),\lb{A.97}
\end{align}
and recalling \eqref{A.88} in \eqref{A.95} yields
\begin{align}
c_{2,1}(0,\theta_b)=d_{0,\theta_b}(-i)^{-1}\Big(u_2^{[1]}(i,b)+u_1^{[1]}(-i,a)\Big).\lb{A.98}
\end{align}
Therefore, \eqref{A.94}, \eqref{A.96}, \eqref{A.97}, and \eqref{A.98} imply \eqref{A.62}.

\noindent
{\it Proof of item $(iv)$:}  In this case, $\theta_a=\theta_b=0$, so that \eqref{A.69} may be recast as
\begin{align}
&c_{\ell,1}(0,0)u_1(-i,\cdot)+c_{\ell,2}(0,0)u_2(-i,\cdot)&\no\\
&\quad =2i(H_{0,0}+iI_{(a,b)})^{-1}u_{\ell}(i,\cdot)-u_{\ell}(i,\cdot), \quad \ell=1,2,\lb{A.99}
\end{align}
and \eqref{A.63} follows immediately by taking $\ell=1,2$ in \eqref{A.99} and separately evaluating at $x=a$ and $x=b$, using \eqref{3.8}.

To set the stage for proving items $(v)$ and $(vi)$, we write the analogs of \eqref{A.66}--\eqref{A.69} in the non-separated case.  Since $\cB_-$ is a basis for $\cN_-$ and  $\cU_{F,\phi}$ maps $\cN_+$ into $\cN_-$ , we write
\begin{align}
\cU_{F,\phi}u_{\ell}(i,\cdot)&=c_{\ell,1}(F,\phi)u_1(-i,\cdot)+c_{\ell,2}(F,\phi)u_2(-i,\cdot), \quad \ell=1,2,\lb{A.100}
\end{align}
for suitable scalars $\{c_{\ell,k}(F,\phi)\}_{1\leq \ell,k\leq 2}$, so that the matrix representation for $\cU_{F,\phi}$ with respect to the bases $\cB_{\pm}$ is given by
\begin{equation}\lb{A.101}
\big[\cU_{F,\phi}\big]=\begin{pmatrix} c_{1,1}(F,\phi) & c_{2,1}(F,\phi)\\
c_{1,2}(F,\phi) & c_{2,2}(F,\phi)\end{pmatrix}.
\end{equation}
By Theorem \ref{tA.11}\,$(iii)$, one has
\begin{equation}\lb{A.102}
\cU_{F,\phi}=-(H_{F,\phi}-i)(H_{F,\phi}+i)^{-1}|_{\cN_+},
\end{equation}
and as a result,
\begin{equation}\lb{A.103}
\cU_{F,\phi}u_{\ell}(i,\cdot)=2i(H_{F,\phi}+i)^{-1}u_{\ell}(i,\cdot)-u_{\ell}(i,\cdot), \quad \ell=1,2.
\end{equation}

\noindent
{\it Proof of item $(v)$:}  Applying Krein's formula \eqref{3.54} in \eqref{A.103}, one obtains
\begin{align}
&c_{\ell,1}(F,\phi)u_1(-i,\cdot)+c_{\ell,2}(F,\phi)u_2(-i,\cdot)=2i(H_{0,0}+i)^{-1}u_{\ell}(i,\cdot)-u_{\ell}(i,\cdot)\no\\
&\quad -2i\sum_{j,k=1}^2Q_{F,\phi}(-i)_{j,k}^{-1}(u_k(i,\cdot),u_{\ell}(i,\cdot))_{L^2((a,b);rdx)}u_j(-i,\cdot), \quad \ell=1,2,\lb{A.104}
\end{align}
where $Q_{F,\phi}(-i)$ is defined by \eqref{3.53}.
At this point, repeating the argument used in the proof of item $(i)$, systematically replacing
$\cU_{\theta_a,\theta_b}$ by $\cU_{F,\phi}$, $c_{j,k}(\theta_a,\theta_b)$, $1\leq j,k\leq 2$, by
$c_{j,k}(F,\phi)$, $1\leq j,k\leq 2$, and $D_{\theta_a,\theta_b}(z)$, $z=\pm i$, by $Q_{F,\phi}(z)$,
$z=\pm i$, one arrives at \eqref{A.64}.

\noindent
{\it Proof of item $(vi)$:}  Applying Krein's formula \eqref{3.57} in \eqref{A.103}, one obtains
\begin{align}
&c_{\ell,1}(F,\phi)u_1(-i,\cdot)+c_{\ell,2}(F,\phi)u_2(-i,\cdot)=2i(H_{0,0}+i)^{-1}u_{\ell}(i,\cdot)-u_{\ell}(i,\cdot)\no\\
&\quad -2iq_{F,\phi}(-i)^{-1}(u_{F,\phi}(i,\cdot),u_{\ell}(i,\cdot))_{L^2((a,b); rdx)} u_{F,\phi}(-i,\cdot), \quad \ell=1,2, \lb{A.105}
\end{align}
where $q_{F,\phi}(-i)$ is defined by \eqref{3.56}.  By \eqref{3.8} and the very definition of the function $u_{F,\phi}(-i,\cdot)$ (cf. \eqref{3.58}), one has
\begin{align}
u_{F,\phi}(-i,a) = e^{-i\phi}F_{2,2}, \quad u_{F,\phi}(-i,b) = 1.\lb{A.106}
\end{align}
Taking $\ell=1$ in \eqref{A.105} and evaluating separately at $x=a$ and $x=b$ using \eqref{3.8} and \eqref{A.106} yields
\begin{align}
c_{1,1}(F,\phi)&=-2iq_{F,\phi}(-i)^{-1}(u_{F,\phi}(i,\cdot),u_1(i,\cdot))_{L^2((a,b);rdx)}-1,   \lb{A.107}\\
c_{1,2}(F,\phi)&=-2ie^{-i\phi}F_{2,2}q_{F,\phi}(-i)^{-1}(u_{F,\phi}(i,\cdot),u_1(i,\cdot))_{L^2((a,b);dx)}.
\lb{A.108}
\end{align}
Taking $\ell=2$ in \eqref{A.105} and evaluating separately at $x=a$ and $x=b$ using \eqref{3.8} and \eqref{A.106} implies
\begin{align}
c_{2,2}(F,\phi)&=-2ie^{-i\phi}F_{2,2}q_{F,\phi}(-i)^{-1}(u_{F,\phi}(i,\cdot),u_2(i,\cdot))_{L_2((a,b);rdx)}-1,
\lb{A.109}\\
c_{2,1}(F,\phi)&=-2iq_{F,\phi}(-i)^{-1}(u_{F,\phi}(i,\cdot),u_2(i,\cdot))_{L^2((a,b);rdx)}.\lb{A.110}
\end{align}
One observes that the inner products in \eqref{A.107}--\eqref{A.110} can be computed explicitly in terms of 
$u_j^{[1]}(\pm i, a)$, $u_j^{[1]}(\pm i, b)$, $j=1,2$, the angle $\phi$, and $F_{2,2}$  using \eqref{3.58} together with sesquilinearity of the inner product $(\cdot, \cdot)_{L^2((a,b);rdx)}$ and 
\eqref{A.86}--\eqref{A.89}.  For example,
\begin{align}
\begin{split}
&2i(u_{F,\phi}(i,\cdot),u_1(i,\cdot))_{L^2((a,b);rdx)}    \\
&\quad=e^{i\phi}F_{2,2}\Big[u_2^{[1]}(-i,b)+u_1^{[1]}(i,a) \Big]-\Big[u_1^{[1]}(i,b)-u_1^{[1]}(-i,b) \Big].
\end{split}
\end{align}
A similar expression holds for the inner product of $u_{F,\phi}(i,\cdot)$ with $u_2(i,\cdot)$.   Equations \eqref{A.65a}--\eqref{A.65d} follow as a result of inserting these expressions for the inner products in \eqref{A.107}--\eqref{A.110}.
\end{proof}
%%%%%%%%%%%%%

%%%%%%%%%%%%%%%%%%%%%%%%%%%%%%%%%%%%%%%%
%%%%%%%%%%%%%%%%%%%%%%%%%%%%%%%%%%%%%%%%
\section{A Brief Outlook on Inverse Spectral Problems}    \lb{s6}
%%%%%%%%%%%%%%%%%%%%%%%%%%%%%%%%%%%%%%%%
%%%%%%%%%%%%%%%%%%%%%%%%%%%%%%%%%%%%%%%%

We present a very brief outlook on inverse spectral problems to
be developed in a forthcoming paper. Here we only describe a special case that indicates the
potential for results in this direction.

In this section we make the assumption that
\begin{equation}
p(\cdot) = r(\cdot) = 1 \, \text{ a.e.\ on $(a,b)$.}  
\end{equation}

Consider the special case
\begin{equation}
\Late(z) = \Lateq (z), \quad \te_a,\te_b \in [0, \pi), \quad z\in\bbC\backslash \sigma(\Hte),
\end{equation}
a generalization of the Dirichlet-to-Neumann map
\begin{equation}
\Lambda_{D,N}(z)=\Lazzqq( z) \equiv \Lazz(z), \quad
z\in\bbC\backslash \sigma(H_{0,0}).
\end{equation}

Introduce the Weyl--Titchmarsh $m$-functions with respect to the reference
point the left/right endpoint $a$, respectively, $b$, denoted by
$m_{+,\te_a}(z,\te_b)$, respectively,  $m_{-,\te_b}(z,\te_a)$.
Then $m_{+,\te_a}(\cdot,\te_b)$ and $- m_{-,\te_b}(\cdot,\te_a)$ are
Nevanlinna--Herglotz functions and asymptotically one verifies the relations,
\begin{align}
m_{+,\te_a}(z,\te_b) & \underset{z \to i \infty}{\longrightarrow} \cot(\te_a)
+ \oh(1),
\quad \te_a \in (0, \pi),  \\
m_{+,0}(z,\te_b) & \underset{z \to i \infty}{\longrightarrow} i z^{1/2}
+ \oh\big(z^{1/2}\big),    \\
m_{-,\te_b}(z,\te_a) & \underset{z \to i \infty}{\longrightarrow} - \cot(\te_b)
+ \oh(1),
\quad \te_b \in (0, \pi),  \\
m_{-,0}(z,\te_a) & \underset{z \to i \infty}{\longrightarrow} - i z^{1/2}
+ \oh\big(z^{1/2}\big).
\end{align}

%%%%%%%%%%%%%
\begin{theorem} \lb{t6.1}
Assume Hypothesis \ref{hA.1} with $p=r=1$ and let $\te_a,\te_b \in [0, \pi)$. Then each diagonal entry of $\Late(z)$
$($i.e., $\Lambda_{\te_a,\te_b}(z)_{1,1}$ or $\Lambda_{\te_a,\te_b}(z)_{2,2}$$)$ uniquely
determines $\Hte$, that is, it uniquely determines $q(\cdot)$ a.e.\ on $(a,b)$, and also $\te_a$
and $\te_b$.
\end{theorem}
%%%%%%%%%%%%%
\begin{proof}
It suffices to note the identity
\begin{equation}
\Late (z) = \begin{pmatrix} m_{+,\te_a}(z,\te_b) & \Late (z)_{1,2}
\\[1mm]
\Late (z)_{2,1} & -m_{-,\te_b}(z,\te_a) \end{pmatrix}
\end{equation}
(where $\Late (z)_{1,2} = \Late (z)_{2,1}$), and then apply Marchenko's
fundamental 1952 uniqueness result \cite{Ma73} formulated in terms of $m$-functions.
\end{proof}
%%%%%%%%%%%%%

One notes that this is in stark contrast to the usual $2 \times 2$ matrix-valued
Weyl--Titchmarsh $M$-matrix. Theorem \ref{t6.1} has instant consequences for Borg--Levinson-type
uniqueness results (such as, two spectra uniquely determine $\Hte$, etc.).

It is natural to conjecture that the role $m_{+,\te_a}(\cdot,\te_b)$ (resp., $m_{-,\te_b}(\cdot,\te_a)$)
plays for uniqueness results in the case of separated boundary conditions in connection with $\Hte$,
in general, is played by the boundary data map
$\Lambda_{\AB}^{A'(A,B),B'(A,B)} (\cdot)$ (for a very particular choice
of $A', B'$ as a function of $A, B$) in the case of general boundary conditions in connection with
$H_\AB$. This will be studied in detail in forthcoming work.

%%%%%%%%%%%%%%%%%%%%%%%%%%%%%%%%%%%%%%%%
%%%%%%%%%%%%  The Appendix  %%%%%%%%%%%%
%%%%%%%%%%%%%%%%%%%%%%%%%%%%%%%%%%%%%%%%
%%%%%%%%%%%%%%%%%%%%%%%%%%%%%%%%%%%%%%%%

\begin{appendix}

%%%%%%%%%%%%%%%%%%%%%%%%%%%%%%%%%%%%%%%%
%%%%%%%%%%%%%%%%%%%%%%%%%%%%%%%%%%%%%%%%
\section{Krein-Type Resolvent Formulas}    \lb{sB}
%%%%%%%%%%%%%%%%%%%%%%%%%%%%%%%%%%%%%%%%
%%%%%%%%%%%%%%%%%%%%%%%%%%%%%%%%%%%%%%%%

In this appendix we provide a brief survey of Krein resolvent formulas, closely following the discussion in \cite[\Sect . 84]{AG81} (with additional input taken from \cite{GMT98}).

First, we introduce some terminology.  Suppose $A$ is a densely defined symmetric operator in the Hilbert space $\cH$ with finite deficiency indices $(m,m)$.  Let $A_1$ and $A_2$ denote two self-adjoint extensions of $A$:
\begin{equation}\lb{B.1}
A\subseteq A_1, \quad A\subseteq A_2.
\end{equation}
Any operator $C$ that satisfies
\begin{equation}\lb{B.2}
C\subseteq A_1, \quad C\subseteq A_2,
\end{equation}
is called a \textit{common part} of the operators $A_1$ and $A_2$.  The operator $C'$ defined by
\begin{equation}\lb{B.3}
C'f=A_1f, \quad f\in \dom(C')=\{f\in \dom(A_1)\cap\dom(A_2) \, | \, A_1f=A_2f\}
\end{equation}
is called the \textit{maximal common part} of $A_1$ and $A_2$ since it satisfies \eqref{B.2} and is an extension of any common part of $A_1$ and $A_2$.  $C'$ is densely defined since $\dom(A)\subseteq \dom(C')$ and is either an extension of $A$ or coincides with $A$.  In the latter case, the extensions $A_1$ and $A_2$ are called \textit{relatively prime}.  Obviously, the two extensions $A_1$ and $A_2$ are relatively prime if and only if
\begin{equation}\lb{B.4}
\dom(A_1)\cap\dom(A_2)=\dom(A).
\end{equation}

We are interested in a formula that relates the resolvents of two different self-adjoint extensions of the symmetric operator $A$.  Thus, let $A_1$ be a fixed self-adjoint extension of $A$ (i.e., $A_1$ plays the role of a reference operator) and let $A_2$ be another self-adjoint extension of $A$, and suppose that $A_1$ and $A_2$ are relatively prime with respect to their maximal common part $A_0$ which has deficiency indices $(r,r)$ with $0\leq r\leq m$.

Since $A_1$ and $A_2$ are extensions of $A_0$,
\begin{align}\lb{B.5}
\begin{split}
[(A_1-z I_{\cH})^{-1}-(A_2-z I_{\cH})^{-1}](A_0-z I_{\cH})g=g-g=0,&  \\
g\in \dom(A_0), \ z\in \rho(A_1)\cap \rho(A_2).&
\end{split}
\end{align}
On the other hand,
\begin{align}
& \big([(A_1-z I_{\cH})^{-1}-(A_2-z I_{\cH})^{-1}]f,h\big)_{\cH}
=\big(f,[(A_1- \overline{z} I_{\cH})^{-1}
-(A_2-\overline{z} I_{\cH})^{-1}]h\big)_{\cH}    \no \\
& \quad =(f,0)_{\cH} =0, \quad  h\in \ran(A_0-\overline{z} I_{\cH}),
\quad z\in \rho(A_1)\cap\rho(A_2),\lb{B.6}
\end{align}
which makes use of the fact that $A_1$ and $A_2$ are extensions of $A_0$.  In summary,
\begin{align}\lb{B.7}
& [(A_1-z I_{\cH})^{-1}-(A_2-z I_{\cH})^{-1}]f \begin{cases}
=0, & f\in \ran(A_0-z I_{\cH}), \\
\in \ker(A_0^{\ast}-z I_{\cH}), & f\in \ker(A_0^{\ast}-\overline{z} I_{\cH}),
\end{cases}   \no \\
& \hspace*{8cm} z\in \rho(A_1)\cap \rho(A_2).
\end{align}
If one chooses $r$ linearly independent vectors (one recalls that $A_0$ has deficiency indices $(r,r)$)
\begin{equation}\lb{B.8}
g_1({z}),g_2({z}),\ldots,g_r({z})\in \ker(A_0^{\ast}-z I_{\cH}), \quad z\in \rho(A_1)\cap\rho(A_2),
\end{equation}
then it follows from \eqref{B.7} that
\begin{equation}\lb{B.9}
[(A_1-z I_{\cH})^{-1}-(A_2-z I_{\cH})^{-1}]f=\sum_{k=1}^rc_k(f;z)g_k(z), \quad f\in \cH, \; z\in \rho(A_1)\cap\rho(A_2),
\end{equation}
for suitable scalars $c_k(f;z)$, $k=1,\ldots,r$.  By \eqref{B.9}, each $c_k(\cdot;z)$ is a linear functional. Linearity follows from \eqref{B.9}; boundedness follows from boundedness of the resolvent difference in \eqref{B.9} and an application of
\cite[Lemma 2.4--1]{Kr78}. Thus, for each $z\in \rho(A_1)\cap\rho(A_2)$, there are vectors $\{h_k(z)\}_{k=1}^r$ such that
\begin{equation}\lb{B.10}
c_k(f;z)=(h_k(z),f)_{\cH}, \quad f\in \cH, \; z\in \rho(A_1)\cap \rho(A_2), \; k=1,\ldots,r.
\end{equation}
Moreover,
\begin{equation}\lb{B.11}
(h_k(z),f)_{\cH}=0, \quad f\in \ran(A_0-z I_{\cH}), \; z\in \rho(A_1)\cap\rho(A_2), \; k=1,\ldots,r,
\end{equation}
in light of \eqref{B.7}, \eqref{B.9}, and the fact that $\{g_k(z)\}_{k=1}^r$ are linearly independent.  By \eqref{B.11},
\begin{equation}\lb{B.12}
\{h_k(z)\}_{k=1}^r \subseteq
\ker(A_0^{\ast}-\overline{z} I_{\cH}), \quad z\in \rho(A_1)\cap\rho(A_2),
\end{equation}
so that each $h_k(z)$ may be represented as
\begin{equation}\lb{B.13}
h_j(z)= - \sum_{k=1}^r\overline{p_{j,k}(z)}g_k(\overline{z}), \quad z\in \rho(A_1)\cap \rho(A_2),
\; j=1,\ldots,r.
\end{equation}
Then \eqref{B.9} becomes
\begin{align}\lb{B.14}
\begin{split}
[(A_1-z I_{\cH})^{-1}-(A_2-z I_{\cH})^{-1}]f
= - \sum_{j,k=1}^rp_{k,j}(z)(g_j (\overline{z}),f)_{\cH} \, g_k(z),& \\
\quad f\in \cH, \; z\in \rho(A_1)\cap\rho(A_2).&
\end{split}
\end{align}
The $r\times r$ matrix $P(z)=\big(p_{j,k}(z)\big)_{1\leq j,k \leq r}$ turns out to be nonsingular for
all $z\in \rho(A_1)\cap \rho(A_2)$.  Indeed, if $P(z_0)$ were singular for some
$z_0\in \rho(A_1)\cap \rho(A_2)$, then by \eqref{B.13}, the vectors
$\{h_k(z_0)\}_{k=1}^r$ are linearly dependent, implying the existence of a nonzero
vector $h\in \ker(A_0^{\ast}-\overline{z_0} I_{\cH})$ such that $(h,h_k(z_0))=0$ for $k=1,\ldots,r$.  By \eqref{B.10} and \eqref{B.9},
\begin{equation}\lb{B.15}
[(A_1-z_0 I_{\cH})^{-1}-(A_2-z_0 I_{\cH})^{-1}]h=0,
\end{equation}
contradicting the assumption that $A_1$ and $A_2$ are relatively prime with respect to $A_0$.

One can rewrite \eqref{B.14} as the operator equation
\begin{equation}\lb{B.16}
(A_2-z I_{\cH})^{-1}=(A_1-z I_{\cH})^{-1}
- \sum_{j,k=1}^rp_{k,j}(z)(g_j (\overline{z}),\cdot)_{\cH} \, g_k(z),
\quad z\in \rho(A_1)\cap \rho(A_2).
\end{equation}

The choice of basis vectors \eqref{B.8} for $\ker(A_0^{\ast}-z I_{\cH})$ for each
$z\in \rho(A_1)\cap \rho(A_2)$ is completely arbitrary.  We now show how basis
vectors for $\ker(A_0^{\ast}-z I_{\cH})$, $z\in \rho(A_1)\cap \rho(A_2)$, can be specified
in a canonical manner by choosing a basis for $\ker(A_0^{\ast}-z_0 I_{\cH})$ for a single fixed
$z_0\in \rho(A_1)\cap \rho(A_2)$.

Let $z_0\in \rho(A_1)\cap \rho(A_2)$ be fixed.  The operator
\begin{equation}\lb{B.17}
U_{z,z_0}=(A_1-z_0 I_{\cH})(A_1-z I_{\cH})^{-1}=I_{\cH}+(z-z_0)(A_1-z I_{\cH})^{-1}, \quad z\in \rho(A_1)\cap \rho(A_2),
\end{equation}
defines an injection from $\cH$ to $\cH$.  In the case $z=\overline{z_0}$, the
operator $U_{z_0,z_0}$ is the unitary Cayley transform of $A_1$, and it
maps $\ker(A_0^{\ast}-z_0 I_{\cH})$ into $\ker(A_0^{\ast}-\overline{z_0} I_{\cH})$.  More generally, $U_{z,z_0}$ satisfies
\begin{equation}\lb{B.18}
U_{z,z_0}\big(\ker(A_0^{\ast}-z_0 I_{\cH}) \big)=\ker(A_0^{\ast}-z I_{\cH}), \quad z\in \rho(A_1)\cap \rho(A_2).
\end{equation}
In fact, if $g_1(z_0), \ldots, g_r(z_0)$ is a basis for
$\ker(A_0^{\ast}-z_0 I_{\cH})$,
then
\begin{align}
A_0^{\ast}U_{z,z_0}g_k(z_0)&=A_0^{\ast}\big(I_{\cH}
+ (z-z_0) (A_1-z I_{\cH})^{-1}\big)g_k(z_0)\no \\
&=z_0g_k(z_0)+(z-z_0)A_1(A_1-z I_{\cH})^{-1}g_k(z_0)\no \\
&=z_0g_k(z_0)+(z-z_0)\big(I_{\cH}+z(A_1-z I_{\cH})^{-1}\big)g_k(z_0)\no \\
&=z\big( I_{\cH}+(z-z_0)(A_1-z I_{\cH})^{-1}\big)g_k(z_0)\no \\
&=zU_{z,z_0}g_k(z_0), \quad z\in \rho(A_1), \; k=1,\ldots,r.  \lb{B.19}
\end{align}
Since $U_{z,z_0}$ is one-to-one, the vectors
$\{U_{z,z_0}g_k(z_0)\}_{k=1}^r\subset \ker(A_0^{\ast}-z I_{\cH})$ are linearly independent.  Thus, if we define
\begin{align}\lb{B.20}
\begin{split}
g_k(z)=U_{z,z_0}g_k(z_0)=g_k(z_0)+(z-z_0)(A_1-z I_{\cH})^{-1}g_k(z_0),&  \\
z\in \rho(A_1), \; k=1,\ldots,r,&
\end{split}
\end{align}
then $\{g_k(z)\}_{k=1}^r$ is a basis for $\ker(A_0^{\ast}-z I_{\cH})$ and \eqref{B.20} represents a systematic (canonical) way of choosing the bases in \eqref{B.8}, having first fixed a single basis $\{g_k(z_0)\}_{k=1}^r$ for $\ker(A_0^{\ast}-z_0 I_{\cH})$.  Moreover, each $g_k(z)$ is an analytic function of $z\in \rho(A_1)$, and the first resolvent equation for $A_1$ yields
\begin{equation}\lb{B.21}
g_k(z')=U_{z',z}g_k(z)=g_k(z)+(z'-z)(A_1-z' I_{\cH})^{-1}g_k(z),
\quad z,z'\in \rho(A_1).
\end{equation}

For $z\in \rho(A_1)\cap \rho(A_2)$, \eqref{B.20} fixes bases $\{g_k(z)\}_{k=1}^r$ and
$\{g_k(\overline{z})\}_{k=1}^r$ for $\ker(A_0^{\ast}-z I_{\cH})$ and
$\ker(A_0^{\ast}-\overline{z} I_{\cH})$, respectively.  There is a corresponding matrix $P(z)$ so that \eqref{B.16} holds.  The matrix $P(z)$ is completely determined by $P(z_0)$.  To see this, let $z\in \rho(A_1)\cap\rho(A_2)$ be fixed.  By \eqref{B.16},
\begin{align}
(A_2-z I_{\cH})^{-1}&=(A_1-z I_{\cH})^{-1} -
\sum_{j,k=1}^rp_{k,j}(z)(g_j(\overline{z}),\cdot)_{\cH} \, g_k(z),  \lb{B.22}\\
(A_2-z_0 I_{\cH})^{-1}&=(A_1-z_0 I_{\cH})^{-1}
- \sum_{j,k=1}^rp_{k,j}(z_0)(g_j(\overline{z_0}),\cdot)_{\cH} \, g_k(z_0).\lb{B.23}
\end{align}
Substituting both of \eqref{B.22} and \eqref{B.23} into the (first) resolvent equation for $A_2$,
\begin{equation}\lb{B.24}
(A_2-z I_{\cH})^{-1}=(A_2-z_0 I_{\cH})^{-1}
+ (z-z_0)(A_2-z I_{\cH})^{-1}(A_2-z_0 I_{\cH})^{-1},
\end{equation}
and using the first resolvent equation for $A_1$ yields
\begin{align}
& \sum_{j,k=1}^rp_{k,j}(z)(g_j(\overline{z},\cdot)_{\cH} \, g_k(z)
= \sum_{j,k=1}^rp_{k,j}(z_0)(g_j(\overline{z_0}),\cdot)_{\cH} \, g_k(z_0)\no \\
&\quad + (z-z_0)\sum_{j,k=1}^rp_{k,j}(z_0) (g_j(\overline{z_0}),\cdot)_{\cH} \,
(A_1-z I_{\cH})^{-1}g_k(z_0)\no \\
&\quad + (z-z_0)\sum_{j,k=1}^rp_{k,j}(z)(g_j(\overline{z}),
(A_1-z_0 I_{\cH})^{-1}\cdot)_{\cH} \, g_k(z)\no \\
&\quad - (z-z_0)\sum_{j,k,\ell,m=1}^rp_{k,j}(z) (g_j(\overline{z}),g_m(z_0))_{\cH} \,
p_{m,\ell}(z_0)(g_\ell(\overline{z_0}),\cdot)_{\cH} \, g_k(z).\lb{B.25}
\end{align}
Using \eqref{B.20}, the sum of the second and third summand on the right-hand
side of \eqref{B.25} can be rewritten as
\begin{equation}\lb{B.26}
- \sum_{j,k=1}^rp_{k,j}(z_0)(g_j(\overline{z_0}),\cdot)_{\cH} [g_k(z_0)-g_k(z)]
- \sum_{j,k=1}^rp_{k,j}(z)([g_j(\overline{z_0})-g_j(\overline{z})],\cdot)_{\cH} \, g_k(z).
\end{equation}
Substitution of \eqref{B.26} into \eqref{B.25} in place of the second and third term
on the right-hand side then yields a linear combination of $\{g_k(z)\}_{k=1}^r$:
\begin{align}
& \sum_{j,k=1}^rp_{k,j}(z_0) (g_j(\overline{z_0}),\cdot)_{\cH} \, g_k(z)
- \sum_{j,k=1}^rp_{k,j}(z)(g_j(\overline{z_0}),\cdot)_{\cH} \, g_k(z)\no \\
&\quad + (z-z_0)\sum_{j,k,\ell,m=1}^rp_{k,j}(z) (g_j(\overline{z}),g_m(z_0))_{\cH} \,
p_{m,\ell}(z_0)(g_\ell(\overline{z_0}),\cdot)_{\cH} \, g_k(z)=0.\lb{B.27}
\end{align}
Since $\{g_k(z)\}_{k=1}^r$ are linearly independent, it follows that
\begin{align}
& \sum_{j=1}^rp_{k,j}(z_0)(g_j(\overline{z_0}),\cdot)_{\cH}
- \sum_{j=1}^rp_{k,j}(z)(g_j(\overline{z_0}),\cdot)_{\cH}    \no \\
&\quad +(z-z_0)\sum_{\ell,m,n=1}^r p_{k,\ell}(z) (g_\ell(\overline{z}),g_m(z_0))_{\cH} \,
p_{m,n}(z_0)(g_n(z_0),\cdot)_{\cH} = 0,\lb{B.28}
\end{align}
and therefore,
\begin{equation}\lb{B.29}
p_{k,j}(z_0) - p_{k,j}(z)
+ (z-z_0) \sum_{\ell,m=1}^r p_{k,\ell}(z) (g_\ell(\overline{z}),g_m(z_0))_{\cH} \,
p_{m,j}(z_0)=0,
\end{equation}
since $\{g_k(\overline{z_0})\}_{k=1}^r$ are linearly independent.  As a matrix equation, \eqref{B.29} reads
\begin{equation}\lb{B.30}
P(z)-P(z_0) - (z-z_0)P(z)\big((g_j(\overline{z}),g_k(z_0))_{\cH}
\big)_{1\leq j,k\leq r}P(z_0)=0.
\end{equation}
Multiplying \eqref{B.30} on the left (resp., right) by $P(z)^{-1}$ (resp., $P(z_0)^{-1}$) yields
\begin{equation}\lb{B.31}
P(z)^{-1}=P(z_0)^{-1} - (z-z_0)\big((g_j(\overline{z}),g_k(z_0))_{\cH}
\big)_{1\leq j,k\leq r }.
\end{equation}
More generally, one can show that
\begin{equation}\lb{B.32}
- P(z)^{-1} = - P(z')^{-1} + (z-z')\big((g_j(\overline{z}),g_k(z'))_{\cH}
\big)_{1\leq j,k\leq r }, \quad z,z'\in \rho(A_1)\cap\rho(A_2).
\end{equation}

In summary, one has the following result:

%%%%%%%%%%
\begin{theorem} \lb{tB.1}
Suppose that $A$ is a densely defined, symmetric operator in $\cH$ with finite deficiency indices $(m,m)$.  Let $A_1$ and $A_2$ denote two self-adjoint extensions of $A$, relatively prime with respect to their maximal common part $A_0$.  For a fixed
$z_0\in \rho(A_1)\cap \rho(A_2)$, let
\begin{equation}\lb{B.33}
\{g_k(z_0)\}_{k=1}^r
\end{equation}
be a fixed basis for $\ker(A_0^{\ast}-z_0 I_{\cH})$ $($$0\leq r\leq m$$)$, and define
\begin{equation}\lb{B.34}
U_{z,z_0}=(A_1-z_0 I_{\cH})(A_1-z I_{\cH})^{-1}, \quad z\in \rho(A_1).
\end{equation}
Then the following hold: \\
$(i)$ $\{g_k(z)\}_{k=1}^r$ defined by
\begin{align}\lb{B.35}
\begin{split}
g_k(z)=U_{z,z_0}g_k(z_0)=g_k(z_0)+(z-z_0)(A_1-z I_{\cH})^{-1}g_k(z_0),&  \\
z\in \rho(A_1), \; k=1,\ldots,r,&
\end{split}
\end{align}
forms a basis for $\ker(A_0^{\ast}-z I_{\cH})$. \\
$(ii)$ $\{g_k(z)\}_{k=1}^r$ and $\{g_k(z')\}_{k=1}^r$ for $z,z'\in \rho(A_1)$ are related by \eqref{B.21}. \\
$(iii)$ For each $z\in \rho(A_1)\cap\rho(A_2)$, there is a unique, nonsingular,
$r\times r$ Nevanlinna--Herglotz matrix
$P(\cdot)= \big(p_{j,k}(\cdot)\big)_{1\leq j,k \leq r}$, depending on the choice of basis \eqref{B.33}, such that
\begin{equation}\lb{B.36}
(A_2 -z I_{\cH})^{-1}=(A_1-z I_{\cH})^{-1}
- \sum_{j,k=1}^r p_{j,k}(z)(g_k(\overline{z}),\cdot)_{\cH} \, g_j(z).
\end{equation}
In particular, $P(\cdot)$ is analytic on the open complex half-plane, $\bbC_+$, and
\begin{equation}
\Im\big(- P(z)^{-1}\big) = \Im(z)
\big((g_j(\overline{z}),g_k(\ol{z}))_{\cH}\big)_{1 \leq j,k \leq r} > 0, \quad z \in \bbC_+.
\lb{B.36a}
\end{equation}
$(iv)$ $P(z)$ and $P(z')$ for $z,z'\in \rho(A_1)\cap \rho(A_2)$ are related by \eqref{B.32}. \\
$(v)$  If $\{\widehat{g}_k(z_0)\}_{k=1}^r$ is any other basis for
$\ker(A_0^{\ast}-z_0 I_{\cH})$ and $\widehat{P}(z)
=\big(\widehat{p}_{j,k}(z)\big)_{1\leq j,k \leq r}$ is the corresponding unique,
$r\times r$ matrix-valued function such that
\begin{equation}\lb{B.37}
(A_2 -z I_{\cH})^{-1}=(A_1-z I_{\cH})^{-1}
- \sum_{j,k=1}^r\widehat{p}_{j,k}(z) (\widehat{g}_k(\overline{z}),\cdot)_{\cH} \,
\widehat{g}_j(z),
\quad z\in \rho(A_1)\cap\rho(A_2),
\end{equation}
then
\begin{equation}\lb{B.38}
\widehat{P}(z)=\big(T^{-1}\big)^\top P(z)\big((T^{-1})^\top\big)^{\ast},
\end{equation}
where $T$ is the $r\times r$ transition matrix corresponding to the change of basis
from $\{g_k(z_0)\}_{k=1}^r$ to $\{\widehat{g}_k(z_0)\}_{k=1}^r$.
\end{theorem}
%%%%%%%%%%%%
\begin{proof}
Choosing $z' = \ol z$, $z \in\bbC_+$, in \eqref{B.32} immediately proves the equality
part in \eqref{B.36a}. Since in general, $\big((g_j,g_k)_{\cH}\big)_{1 \leq j,k \leq N}$
represents the positive definite Gramian (cf., e.g., \cite[p.\ 109, 297]{PS76}) of the
system of linearly independent
elements $g_j \in \cH$, $1 \leq j \leq N$, for arbitrary $N \in \bbN$, this yields the
positive definiteness part in \eqref{B.36a}. In particular, $- P(\cdot)^{-1}$, and hence $P(\cdot)$, possesses the Nevanlinna--Herglotz property claimed in item $(iii)$.

To prove the uniqueness part of item $(iii)$, suppose that in addition to the
representation \eqref{B.37}, one has the representation
\begin{equation}\lb{B.39}
(A_2 -z I_{\cH})^{-1}=(A_1-z I_{\cH})^{-1} - \sum_{j,k=1}^r \widetilde{p}_{j,k}(z)
(g_k(\overline{z}),\cdot)_{\cH} \, g_j(z).
\end{equation}
Then it follows that
\begin{equation}\lb{B.40}
\sum_{j,k=1}^r\big(\widetilde{p}_{j,k}(z)-p_{j,k}(z) \big)(g_k(\overline{z}),f)_{\cH} \,
g_j(z)=0,  \quad f\in \cH,
\end{equation}
and since the vectors $\{g_j(z)\}_{j=1}^r$ are linearly independent,
\begin{equation}\lb{B.41}
\sum_{k=1}^r \big(\widetilde{p}_{j,k}(z)-p_{j,k}(z) \big)(g_k(\overline{z}),f)_{\cH}
=0, \quad f\in \cH, \; j=1,\ldots,r.
\end{equation}
Therefore,
\begin{equation}\lb{B.42}
\sum_{k=1}^r \overline{\big(\widetilde{p}_{j,k}(z)-p_{j,k}(z) \big)}g_k(\overline{z})=0, \quad j=1,\ldots,r,
\end{equation}
and linear independence of $\{g_j(\overline{z})\}_{j=1}^r$ yields
\begin{equation}\lb{B.43}
\widetilde{p}_{j,k}(z)-p_{j,k}(z)=0, \quad j,k=1,\ldots,r.
\end{equation}

Next we prove the uniqueness claim in item $(v)$:  Suppose that, in addition to $\{g_k(z_0)\}_{k=1}^r$, $\{\widehat{g}_k(z_0)\}_{k=1}^r$ is also a basis for $\ker(A_0^{\ast}-z_0 I_{\cH})$.  Then
\begin{align}
(A_2 -z I_{\cH})^{-1}&=(A_1-z I_{\cH})^{-1}
- \sum_{j,k=1}^rp_{j,k}(z)(g_k(\overline{z}),\cdot)_{\cH} \, g_j(z),    \lb{B.44}\\
(A_2-z I_{\cH})^{-1}&=(A_1-z I_{\cH})^{-1}
- \sum_{j,k=1}^r\widehat{p}_{j,k}(z)
(\widehat{g}_k(\overline{z}),\cdot)_{\cH} \, \widehat{g}_j(z),    \lb{B.45} \\
& \hspace*{4.4cm} z\in \rho(A_1)\cap\rho(A_2),    \no
\end{align}
with
\begin{equation}
g_k(z)=U_{z,z_0}g_k(z_0),    \quad
\widehat{g}_k(z) =U_{z,z_0}\widehat{g}_k(z_0), \quad  k=1,\ldots,r; \;
z\in \rho(A_1). \lb{B.46}
\end{equation}

  Let $T\in \mathbb{C}^{r\times r}$ denote the nonsingular transition matrix corresponding to the change of basis from $\{g_k(z_0)\}_{k=1}^r$ to $\{\widehat{g}_k(z_0)\}_{k=1}^r$ so that
\begin{equation}
  \widehat{g}_k(z_0)=\sum_{j=1}^rT_{k,j}g_j(z_0),   \quad
   g_k(z_0) = \sum_{j=1}^r(T^{-1})_{k,j}\widehat{g}_j(z_0),
  \quad k=1,\ldots,r.   \lb{B.49}
  \end{equation}
  From \eqref{B.46}--\eqref{B.49} one obtains the relations
\begin{equation}
  \widehat{g}_k(z)=\sum_{j=1}^rT_{k,j}g_j(z),  \quad
  g_k(z)=\sum_{j=1}^r(T^{-1})_{k,j}\widehat{g}_j(z), \quad
k=1,\ldots,r; \; z\in \rho(A_1).  \lb{B.51}
\end{equation}
One observes that by \eqref{B.51},
  \begin{align}
&  \sum_{j,k=1}^rp_{j,k}(g_k(\overline{z}),\cdot)_{\cH} \, g_k(z)
=\sum_{j,k=1}^rp_{j,k}(z) \bigg(\sum_{\ell=1}^r(T^{-1})_{k,\ell}\widehat{g}_{\ell}(\overline{z}),\cdot\bigg)_{\cH} \sum_{m=1}^r(T^{-1})_{j,m}\widehat{g}(z)    \no \\
  &\quad= \sum_{j,k=1}^r\sum_{\ell=1}^r \sum_{m=1}^rp_{j,k}(z)\overline{(T^{-1})_{k,\ell}} (T^{-1})_{j,m}(\widehat{g}_{\ell}(\overline{z}),\cdot)_{\cH} \,
  \widehat{g}_m(z), \quad
z\in \rho(A_1)\cap \rho(A_2).      \lb{B.52}
  \end{align}
  Using
  \begin{equation}\lb{B.53}
  \big(((T^{-1})^\top)^{\ast} \big)_{j,k}=\overline{(T^{-1})_{j,k}}\quad \text{and}
  \quad \big((T^{-1})^\top\big)_{j,k}=(T^{-1})_{k,j}, \quad 1\leq j,k \leq r,
  \end{equation}
  one has
\begin{align}
& \big((T^{-1})^\top P(z)((T^{-1})^\top)^{\ast} \big)_{m,\ell}
= \sum_{j,k=1}^r\big((T^{-1})^\top  \big)_{m,j}p_{j,k}(z)
\big(((T^{-1})^\top)^{\ast} \big)_{k,\ell}\no \\
 & \quad = \sum_{j,k=1}^r(T^{-1})_{j,m}p_{j,k}(z)\overline{(T^{-1})_{k,\ell}}, \quad
z\in \rho(A_1)\cap\rho(A_2).  \lb{B.54}
\end{align}
  By \eqref{B.52} and \eqref{B.54},
  \begin{align}
  \sum_{j,k=1}^rp_{j,k}(z)(g_k(\overline{z}),\cdot)_{\cH} \, g_k(z)
  &=\sum_{m,\ell=1}^r \big((T^{-1})^\top
  P(z)((T^{-1})^\top)^{\ast} \big)_{m,\ell}(\widehat{g}_{\ell}
  (\overline{z}),\cdot)_{\cH} \, \widehat{g}_m(z),\no \\
  &\hspace*{3.9cm} z\in \rho(A_1)\cap\rho(A_2).\lb{B.55}
  \end{align}
  Therefore, we have the following two representations:
  \begin{align}
  (A_2 -z I_{\cH})^{-1}&=(A_1-z I_{\cH})^{-1}
  - \sum_{j,k=1}^r\widehat{p}_{j,k}(z)
  (\widehat{g}_k(\overline{z}),\cdot)_{\cH} \, \widehat{g}_j(z),  \lb{B.56}\\
  (A_2 -z I_{\cH})^{-1}&=(A_1-z I_{\cH})^{-1} - \sum_{j,k=1}^r\big((T^{-1})^\top
  P(z)((T^{-1})^\top)^{\ast} \big)_{j,k}(\widehat{g}_k(\overline{z}),\cdot)_{\cH} \,
  \widehat{g}_j(z),
 \no \\
  & \hspace*{6cm}  z\in \rho(A_1)\cap\rho(A_2),    \lb{B.57}
  \end{align}
 and hence,
  \begin{equation}\lb{B.58}
  \widehat{P}(z)=(T^{-1})^\top P(z)((T^{-1})^\top)^{\ast}, \quad z\in \rho(A_1)\cap\rho(A_2).
  \end{equation}
\end{proof}
%%%%%%%%%%%%%

\end{appendix}

\medskip

%%%%%%%%%%%%%%%%%%%%%%%%%%%%%%%%%%%%%
{\bf Acknowledgments.} We are indebted to Malcolm Brown, Pavel Kurasov,
Annemarie Luger, and Sergey Naboko for very helpful discussions and constructive remarks. S.C. and F.G.\ wish to thank the organizers of the workshop ``Inverse Spectral Problems in One Dimension'' at the International Centre for Mathematical Sciences, Edinburgh, and the organizers of the six month meeting on ``Inverse Problems'' at the Isaac Newton Institute for Mathematical Sciences, Cambridge, for providing facilities and a most stimulating atmosphere that helped fostering work on this project.
%%%%%%%%%%%%%%%%%%%%%%%%%%%%%%%%%%%%%

%%%%%%%%%%%%%%%%%%%%%%%%%%%%%%%%%%%%
%%%%%%%%%%%%%%%%%%%%%%%%%%%%%%%%%%%%

\end{document}